\documentclass[a4paper]{article}
\UseRawInputEncoding

\usepackage{amsthm}

\newtheorem{thm}{Theorem}[section]
\newtheorem{lem}{Lemma}[section]

\newtheorem{prop}{Proposition}[section]

\usepackage[toc,page]{appendix}
\usepackage{url}
\usepackage{tikz}
\usepackage{amsmath}
\usepackage{mathtools}
\usepackage[a4paper, left=3cm, right=3cm, top=4cm]{geometry}
\usepackage{amssymb}
\usepackage{amsfonts}
\usepackage{scrlayer-scrpage}
\usepackage{hyperref}
\usepackage{accents}
\usepackage{ulem}
\clearscrheadfoot
\newcommand{\R}{\mathbb{R}}

\newcommand{\Z}{\mathbb{Z}}

\def\Xint#1{\mathchoice
{\XXint\displaystyle\textstyle{#1}}
{\XXint\textstyle\scriptstyle{#1}}
{\XXint\scriptstyle\scriptscriptstyle{#1}}
{\XXint\scriptscriptstyle\scriptscriptstyle{#1}}
\!\int}
\def\XXint#1#2#3{{\setbox0=\hbox{$#1{#2#3}{\int}$ }
\vcenter{\hbox{$#2#3$ }}\kern-.6\wd0}}

\def\dashint{\Xint-}

\AtBeginDocument{
  \let\div\relax
  \DeclareMathOperator{\div}{div}
}

\usepackage{setspace}
\makeatletter
\newcommand{\MSonehalfspacing}{%
  \setstretch{1.44}%  default
  \ifcase \@ptsize \relax % 10pt
    \setstretch {1.448}%
  \or % 11pt
    \setstretch {1.399}%
  \or % 12pt
    \setstretch {1.433}%
  \fi
}
\newcommand{\MSdoublespacing}{%
  \setstretch {1.92}%  default
  \ifcase \@ptsize \relax % 10pt
    \setstretch {1.936}%
  \or % 11pt
    \setstretch {1.866}%
  \or % 12pt
    \setstretch {1.902}%
  \fi
}
\makeatother

\ohead{\pagemark}

\begin{document}

	\title{Uniqueness and Regularity of the Fractional Harmonic Gradient Flow in $S^{n-1}$}
	
	\author{Jerome Wettstein}
\maketitle
\date{ }
\begin{abstract}
In this paper, we study the fractional harmonic gradient flow on $S^1$ taking values in $S^{n-1} \subset \R^n$ for every $n \geq 2$, in particular addressing uniqueness and regularity of solutions in the so-called energy class with sufficiently small energy, adding to the existing body of knowledge which includes existence of solutions, see \cite{schisirewang}, and bubbling phenomena as studied by \cite{sireweizheng}. We extend the techniques by Struwe in \cite{struwe1} and Rivi\`ere in \cite{riv} to the non-local framework and exploit integrability by compensation properties due to fractional Wente-type inequalities as in \cite{mazoschi}. Moreover, we briefly discuss convergence properties for solutions to the fractional gradient flow as $t \to \infty$.
\end{abstract}

\medskip

\tableofcontents

	\section{Introduction}
	
	In this paper, we shall study gradient flows associated with the half-harmonic map equation, in particular questions pertaining to uniqueness, regularity and convergence as $t \to +\infty$ of solutions of the fractional harmonic gradient flow in $S^{n-1} \subset \R^n$. In \cite{struwe1} and \cite{struwe2}, Struwe studied global existence and uniqueness for the gradient flow associated with the classical harmonic map equation both in dimension $2$ as well as higher dimensions. Recall that harmonic maps are critical points of the standard Dirichlet energy which is defined for all maps $u: M \to N \subset \R^n$ in $H^{1}(M;N)$ by:
	$$E(u) := \frac{1}{2} \int_{M} g^{\alpha \beta}(x) \gamma_{ij}(u(x)) \frac{\partial u^{i}}{\partial x_{\alpha}}(x) \frac{\partial u^{j}}{\partial x_{\beta}}(x) dx,$$ 
	where $(M,g), (N, \gamma)$ smooth Riemannian manifolds, $u = (u^{1}, \ldots, u^{n})$ and employing Einstein's summation convention. In case $M = \Omega \subset \R^m$ and $N \subset \R^{n}$ are isometrically embedded in $\R^m$ and $\R^n$ and equipped with the Riemannian metrics induced by the standard scalar product, this reduces to:
	$$E(u) = \frac{1}{2} \int_{\Omega} | \nabla u |^2 dx$$ 
	In domains of dimension $2$, Struwe actually showed that up to a bubbling process at finitely many points, the number of which can be bounded by the initial energy, there exists a unique regular solution for all times. To be more precise, Struwe proved the following for the target manifold $N = S^{n-1}$ (a completely analogous result holds for general $N$):
	
	\begin{thm}[Theorem 1, p.98, \cite{riv}]
	\label{mainstruwe1}
		Let $\Omega \subset \R^2$ as well as $u_0 \in H^{1}(\Omega; S^{n-1}), \gamma \in C^{\infty}(\partial \Omega; S^{n-1})$. Then there exists a solution $u \in H^{1}(]0, +\infty[;L^{2}(\Omega))$ of the harmonic gradient flow:
		\begin{equation}
		\label{harmonicflowstruwe1}
			\partial_{t} u - \Delta u = u | \nabla u |^2 \quad \text{ in } \mathcal{D}'(]0, T[ \times \Omega), \quad \forall T > 0,
		\end{equation}
		together with the boundary conditions:
		\begin{align}
		\label{harmonicflowstruwe1bound1}
			u(t,x) 	&= \gamma(x), \quad 	&\text{ for all } t \geq 0, x \in \partial \Omega \\
		\label{harmonicflowstruwe1bound2}
			u(0,x)	&= u_{0}(x), \quad		&\text{ for all } x \in \Omega,
		\end{align}
		and satisfying $E(u(t, \cdot)) \leq E(u_{0})$ for all times $t \geq 0$. The solution $u$ is regular on $]0,+\infty[ \times \overline{\Omega}$, except in a finite number of points $(t_{k}, x_{k})$, $k = 1, \ldots, K$, for some $K \in \mathbb{N}$. Additionally, $u$ is unique in the class $\mathcal{E} \subset H^{1}_{loc}([0,+\infty[ \times \overline{\Omega})$ defined by:
		$$\mathcal{E} := \Big{\{} u \ \Big{|}\ \exists m \in \mathbb{N}, \exists T_0 = 0 < T_1 < \ldots < T_{m} < \infty: u \in L^{2}([T_{i}, T_{i+1}[; W^{2,2}(\Omega)), \forall i \leq m-1 \Big{\}}$$
		Finally, there exists a constant $C > 0$ independent of $u_0$, such that:
		$$K \leq C \cdot E(u_0)$$
	\end{thm}
	A minor drawback of this result is the additional regularity requirement in the definition of $\mathcal{E}$ needed to ensure uniqueness. However, in \cite{riv}, Rivi\`ere managed to remove this condition for solutions in the \textit{energy class} and $N = S^{n-1}$, provided the initial energy is sufficently small. Solutions in the energy class actually refers to solutions $u$ which lie merely in $H^{1}(]0, +\infty[;L^{2}(\Omega)) \cap L^{\infty}([0, +\infty[; H^{1}(\Omega))$ satisfying the inequality $E(u(t, \cdot)) \leq E(u_{0})$. This approach exploited integrability by compensation phenomena inherent to the structure of the harmonic map equation, namely Wente's estimate. To be precise, the following was proven in \cite{riv}:
	
	\begin{thm}[Theorem 2, p.99, \cite{riv}]
		There exists $\varepsilon > 0$, such that for every $u_0 \in H^{1}(\Omega; S^{n-1})$ with:
		$$E(u_0) < \varepsilon,$$
		existence of a unique solution of \eqref{harmonicflowstruwe1}, \eqref{harmonicflowstruwe1bound1}, \eqref{harmonicflowstruwe1bound2} in $H^{1}_{loc}([0,+\infty[ \times \overline{\Omega})$ satisfying $E(u(t, \cdot)) \leq E(u_{0})$ for almost every time $t \geq 0$ is guaranteed. The solution $u$ is in fact regular in $]0, +\infty[ \times \overline{\Omega}$.
	\end{thm}
	
	A key point in the proof is the smallness of the energy that allows us to deduce slightly better regularity for the trace $u(t, \cdot)$ at a.e. fixed time. One should notice that if $\varepsilon > 0$ is sufficiently small, in Struwe's result, Theorem \ref{mainstruwe1}, the possibility of bubbling could be excluded, hence establishing global regularity. Later, in \cite{freire}, Freire was able to remove the small energy restriction and prove a general uniqueness result in the energy class for arbitrary $N$. He did so by employing H\'elein's moving frame technique in the context of the harmonic gradient flow.
	
	Our goal is to generalize the approach by Rivi\`ere in \cite{riv} to the non-local framework and thus to the half-harmonic gradient flow.\\
	
	In analogy to harmonic maps, we may say that a map $u: S^1 \to N \subset \R^n$ is weakly $1/2$-harmonic, if it is a critical point of the following energy:
	\begin{equation}
		E_{1/2}(u) := \frac{1}{2} \int_{S^1} | (-\Delta)^{1/4} u |^2 dx,
	\end{equation}
	with respect to variations in the following set:
	$$H^{1/2}(S^1;N) := \big{\{} v \in H^{1/2}(S^1; \mathbb{R}^{n})\ \big{|}\ u(x) \in N, \text{ for a.e. } x \in S^1 \big{\}}$$
	For convenience's sake, we shall abbreviate $E_{1/2}$ by $E$ throughout the paper. Observe that the criticality condition implies that for every $\Phi \in \dot{H}^{1/2}(S^1;\mathbb{R}^{n}) \cap L^{\infty}(S^1)$, in particular all $\Phi \in C^{\infty}(S^1;\R^n)$, we have:
	\begin{equation}
		\frac{d}{dt} E_{1/2} \left( \pi( u + t \Phi ) \right) \Big{|}_{t = 0} = 0,
	\end{equation}
	where $\pi$ is the orthogonal closest-point projection to $N$, which is defined in a sufficiently small neighbourhood of $N$ and smooth due to $N$ being smooth. As we shall see, this condition is equivalent to:
	\begin{equation}
	\label{1/2harmonicbyorthogpojection}
		d\pi(u) (-\Delta)^{1/2} u = 0 \quad \text{ in } \mathcal{D}'(S^1),
	\end{equation}
	which is sometimes also stated informally in the following form, observing that $d\pi(x)$ is the orthogonal projection to $T_{x} N$ for every $x \in N$:
	$$(-\Delta)^{1/2} u \perp T_{u}N$$
	In our case of interest, $N = S^{n-1}$, this could be restated as:
	$$u \wedge (-\Delta)^{1/2} u = 0 \quad \text{ in } \mathcal{D}'(S^1).$$
	It is clear that, in order to study the regularity of $1/2$-harmonic maps, the first step lies in the reformulation of \eqref{1/2harmonicbyorthogpojection}. Naturally, corresponding definitions for $\R$ instead of $S^1$ are possible.
	
	In fact, the regularity and reformulations were first studied by the authors in \cite{dalioriv}, only the domain being $\R$ instead of $S^1$, the same paper where $1/2$-harmonic maps were first introduced. Since \cite{dalioriv}, several extensions have been considered in \cite{daliocomment1}, \cite{schikorradaliocomment}, \cite{dalioschikorra}, \cite{daliocomment2}, \cite{daliopigati}. The regularity of $1/2$-harmonic maps relies on the following compensation phenomena discovered in \cite{dalioriv2}: If $\Omega \in L^{2}_{loc}(\R; so(m)), v \in L^{2}_{loc}(\R;\R^{m})$ and $f \in L^{p}_{loc}(\R; \R^{m})$, where $1 \leq p \leq 2$ satisfy
	$$(-\Delta)^{1/4} v = \Omega \cdot v + f \quad \text{ in } \mathcal{D}'(\R),$$
	then $(-\Delta)^{1/4} v \in L^{p}_{loc}(\R)$, i.e. $v \in \dot{W}^{1/2, p}(\R)$. This phenomena is based on the existence of special operators satsifying improved integrability properties due to compensation. One such operator is, for instance, given by the so-called \textit{three-term commutator}:
	$$\mathcal{T}: L^{2}(\R; \R^{m}) \times \dot{H}^{1/2}(\R; \R^{m \times m}) \to \dot{H}^{-1/2}(\R; \R^{m}),$$
	defined by:
	$$\mathcal{T}(v,Q) := (-\Delta)^{1/4} (Qv) - Q (-\Delta)^{1/4} v + (-\Delta)^{1/4} Q \cdot v$$
	It is proven in \cite{dalioriv} that:
	$$\| \mathcal{T}(v,Q) \|_{\dot{H}^{-1/2}} \lesssim \| Q \|_{\dot{H}^{1/2}} \| v \|_{L^2}$$
	We also refer to \cite{lenzschi} for an overview of different types of commutator estimates. Recently in \cite{mazoschi}, inspired also by \cite{millotsire}, the authors recast integrability by compensation for fractional operators and commutator estimates in a "classical local way", by applying the notions of fractional divergences and fractional gradients, see Section 2.2 for their definitions. In particular, they succeeded in recasting the integrability by compensation in terms of the following non-local result reminiscent of the result by Coifman, Lions, Meyer and Semmes \cite{coifman}:
	
	\begin{lem}[Theorem 2.1, \cite{mazoschi}]
	\label{fractwenteinintroductionfornow}
		Let $s \in (0,1)$ and $p \in (1,\infty)$. For $F \in L^{p}_{od}(\R \times \R)$ and $g \in \dot{W}^{s,p'}(\R)$, where $p'$ denotes the H\"older dual of $p$, we assume that $\div_s F = 0$. Then $F \cdot d_s g$ lies in the Hardy space $\mathcal{H}^{1}(\R)$ and we have the estimate:
		$$\| F \cdot d_s g \|_{\mathcal{H}^{1}(\R)} \lesssim \| F \|_{L^{p}_{od}(\R \times \R)} \cdot \| g \|_{\dot{W}^{s,p'}(\R)}.$$
	\end{lem}
	
	Lemma \ref{fractwenteinintroductionfornow} has permitted the authors in \cite{mazoschi} to show in an alternative way the regularising effect of non-local systems with anti-symmetric potentials.\\

	In this paper, we are going to study the gradient flow associated with the energy $E_{1/2}$ introduced above, referred to as the fractional or $1/2$-harmonic gradient flow. Namely, we shall study solutions $u$ of the following non-local PDE on $[0, +\infty[ \times S^1$ taking values in the sphere $S^{n-1} \subset \R^n$:
	\begin{equation}
	\label{gradflowintermsoforthogonality}
		d \pi(u) \left( u_t + E_{1/2}'(u) \right) = d \pi(u) \left( u_t + (-\Delta)^{1/2} u \right) = 0,
	\end{equation}
	with $u(0, \cdot) = u_0$ for some initial datum $u_0 \in H^{1/2}(S^1;S^{n-1})$. As in the case of fractional harmonic maps, a first step would be to rephrase the fractional harmonic flow and we shall obtain in the paper the reformulation:
	\begin{equation}
	\label{gradflowinintroduction}
		u_t + (-\Delta)^{1/2} u = u | d_{1/2} u |^2,
	\end{equation}
	where $u$ satisfies $u(0, \cdot) = u_0$. The notation used shall be introduced later on in the paper, however we emphasise that the RHS of the equation is closely related to the $1/2$-harmonic map equation. It should be noted that the formulation \eqref{gradflowinintroduction} mirrors some of the features found in the local case and builds upon the formulation of fractional harmonic maps in \cite{mazoschi}. This equation will be derived later on in the paper.
	
	One might ask what is known for the half-harmonic gradient flow \eqref{gradflowintermsoforthogonality}, \eqref{gradflowinintroduction}. For example, in \cite{schisirewang}, the authors studied and proved the existence of a solution to the half-harmonic gradient flow assuming the map takes values in a sufficiently nice target manifold, i.e. a closed homogeneous space such as the space of interest $N = S^{n-1}$. In fact, they consider for $0 < s < 1$ and $1 < p < +\infty$ the energy functional:
	$$E_{s,p}(u) := \frac{1}{p} \int_{\Omega \times \Omega} \frac{| u(x) - u(y) |^{p}}{| x-y |^{n+sp}} dx dy,$$
	where $\Omega \subset \R^m$ is smooth and bounded, and study the fractional gradient flow equation given informally by:
	\begin{equation}
		d\pi(u) \left( \partial_{t} u + E_{s,p}'(u) \right) = 0,
	\end{equation} 
	for closed manifolds $N \subset \R^n$ and the closest point projection $\pi$, showing existence of an appropriate candidate for general $N$ and verifying that the constructed candidate is a solution, provided $N$ is a homogeneous space. Their methods involve approximations by a piecewise minimization process and immediately yield, in contrast to the techniques employed by Struwe, a global existence result. We highlight that provided $p=2$ and $s = 1/2$, we recover the fractional harmonic gradient flow in \eqref{gradflowintermsoforthogonality}, and consequently \eqref{gradflowinintroduction}, which we will be studying, thus complementing the treatment in \cite{schisirewang} in the case $S^{n-1}$. We mention that using $S^1$ instead of a bounded interval $\Omega \subset \R$ does not obstruct the proof presented in \cite{schisirewang}, as all arguments carry over immediately, therefore the existence result continues to hold true for the domain $S^1$, at least for closed, homogeneous target manifolds.
	
	Nevertheless, the nature of the argument in \cite{schisirewang} does not allow for a uniqueness statement or provide an analysis of possible types of blow-ups in (in)finite time. Questions regarding blow-ups were studied for example in \cite{sireweizheng} where the authors exhibit that only blow-ups in infinite time may occur for certain initial data and conjecture that the same might hold in general.\\
	
	Our main result in this paper will be the following: 
	
	\begin{thm}
	\label{mainresultsection1}
		Let $u_0 \in H^{1/2}(S^1;S^{n-1})$ be any initial data. There exists $\varepsilon > 0$, such that if:
		$$\| ( - \Delta)^{1/4} u_0 \|_{L^2(S^1)} \leq \varepsilon,$$
		then there exists a unique energy class solution $u: \R_{+} \times S^1 \to S^{n-1} \subset \R^n$ of the weak fractional harmonic gradient flow:
		$$u_t + (-\Delta)^{1/2} u = u | d_{1/2} u |^2,$$
		satisfying $u(0, \cdot) = u_0$  in the sense $u(t,\cdot) \to u_0$ in $L^2$, as $t \to 0$. Moreover, the solution fulfills the energy decay estimate:
		$$\| (-\Delta)^{1/4} u(t) \|_{L^2 (S^1)} \leq \| (-\Delta)^{1/4} u_0 \|_{L^{2}(S^1)}.$$
		In fact, $u \in C^{\infty}(]0,\infty[ \times S^1)$ and for an appropriate subsequence $t_k \to \infty$, the sequence $u(t_k)$ converges weakly in $H^{1}(S^1)$ to a point.
	\end{thm}
	
	By energy class solution, we mean that $u$ possesses the following regularity:
	$$u \in L^{\infty}(\R_{+}; H^{1/2}(S^1)), \quad u_t \in L^{2}(\R_{+}; L^{2}(S^1)).$$
	
	The general strategy behind the proof of Theorem \ref{mainresultsection1} is the following: First, following the arguments in \cite{struwe1} for uniqueness, we show that uniqueness holds for slightly more regular solutions than those in energy class. Namely, we require in addition that $u \in L^{2}_{loc}(\R_{+}; H^{1}(S^1))$ and this improved regularity assumption combined with Sobolev-type embeddings for Triebel-Lizorkin spaces yields uniqueness in this class of functions. Then we establish, following the approach in \cite{riv}, that energy class solutions with monotone decreasing $1/2$-energy in time and sufficiently small initial energy are actually slightly more regular and satisfy the condition $u \in L^{2}_{loc}(\R_{+}; H^{1}(S^1))$. This gain in regularity crucially relies on the structure of an anti-symmetric potential hidden in the harmonic map equation and changes of gauge as in Rivi\`ere's seminal work \cite{riv2} adapted in a non-local framework and manifested in non-local Wente-type estimates like Lemma \ref{fractwenteinintroductionfornow} found in \cite{mazoschi}. Indeed, the emergence of an anti-symmetric potential and the resulting benefits are more apparent for $S^{n-1}$ than for general manifolds, since in this case, the potential is even $1/2$-divergence-free, a property which is in general only obtained after a change of gauge, cf. \cite{riv2}. The vanishing $1/2$-divergence actually leads to slightly better integrability properties of the potential and hence the improvement in regularity, see \cite{dalioriv2}, \cite{daliopigati} and \cite{mazoschi}.
	
	To be precise, the following regularity result will be the key point to derive uniqueness for small-energy solutions in the energy class:
	
	\begin{prop}
		Let $u$ satisfy the following regularity assumptions:
		$$u \in L^{\infty}(\R_{+}; H^{1/2}(S^1)); \quad u_t \in L^{2}(\R_{+}; L^{2}(S^{1}))$$ 
		Moreover, assume $u$ solves the half-harmonic gradient flow equation \eqref{gradflowinintroduction}. Then for almost every time $t > 0$, we have:
		$$u(t) \in H^{1}(S^1).$$
	\end{prop}
	
	Combining this with a fractional Ladyzhenskaya inequality and sufficiently small energy will show $u \in L^{2}_{loc}(\R_{+}; H^{1}(S^1))$, analogous to \cite{riv}.
	
	Proving smoothness of the solution relies on bootstrap techniques and a local regularity result from \cite{hamilton} adapted to the non-local setting. Indeed, we shall use the local Inversion Theorem in order to prove existence and regularity of solutions to the flow assuming the boundary data is smooth. The resulting solution will be smooth by using results from \cite{hieber} on parabolic PDEs and maximal estimates for heat flows using operator semigroups. Then, using a generalisation of a Lemma by Schoen-Uhlenbeck \cite{schoenuhlenbeck} (our proof following the presentation in \cite{struwelecturenotes}) and the extension of the harmonic map flow as presented in \cite{struwe1} in the case of the half-harmonic map flow, we deduce regularity in general and for all times, provided the initial energy is sufficently small. The ideas follow more or less \cite{struwe1} and we indicate the most significant changes by establishing the key estimates. Lastly, convergence is obtained just like in \cite{struwe} for the harmonic map flow.
	
	We would like to point out that we could have chosen the formulation of the fractional harmonic map equation introduced in \cite{dalioriv}. However, we did choose the formulation in \eqref{gradflowinintroduction} for its analogy with \eqref{harmonicflowstruwe1}, which also inspired the current investigation into half-harmonic gradient flows.\\

	Some of the main technical difficulties we will encounter in the course this paper will concern the translation of results for the real line $\R$ into results for the unit circle $S^1$ and working with Triebel-Lizorkin spaces over $S^1$. Regarding the former difficulties, some of results of this type may be obtained by an extension procedure, others by changes of variables involving the stereographic projection which connect the $1/2$-Laplacian on the circle to the one on the real line, see e.g. \cite{dalio}, \cite{daliopigati}, \cite{millotsire}, \cite{daliomartriv}. Both approaches seem to be necessary, as there are advantages to both of them. Many of the results derived by such procedures can also be obtained directly using Triebl-Lizorkin spaces. Once all these ingredients are introduced, the proof is based on the arguments found in \cite{struwe1} as well as \cite{daliopigati}, \cite{mazoschi}, \cite{riv}.
	
	In future work, the author plans to investigate uniqueness and regularity of solutions to the fractional harmonic gradient flow with small initial energy in an arbitrary closed manifold $N \subset \R^n$ and then to expand our considerations to solutions with arbitrary initial energy. Some bubbling phenomena are expected to be observable in this case, so the more delicate analysis of this will be carried out in a future paper. A paper dealing with uniqueness and regularity in the general setting of an arbitrary closed manifold $N \subset \R^n$ is already in preparation by the author (\cite{wettstein}).\\
	
	Let us present an outline of the paper: In Section 2, we introduce some of the most important notions and structures for our proofs. In particular, this includes Triebel-Lizorkin spaces on $S^1$ and the fractional Wente-type Lemma \ref{fractionalwentelemma}. Then, in Section 3, we turn to establishing our main result. First, in Section 3.1, we show the equivalence between \eqref{gradflowintermsoforthogonality} and \eqref{gradflowinintroduction}. Then, uniqueness is treated in Section 3.2 following the presentation in \cite{riv} and \cite{struwe1}, regularity in Section 3.3 by a bootstrap trick and using the techniques and results in \cite{hamilton}, \cite{hieber}, \cite{struwe1} and finally, we discuss convergence properties in Section 3.4 following the presentation in \cite{struwe} in the case of the harmonic map flow. The Appendices complement the presentation and add some technical details.\\
	
	\textbf{Acknowledgements}
	Lastly, I would like to thank my supervisors, Prof. Francesca Da Lio and Prof. Tristan Rivi\`ere, for suggesting this problem, providing advice throughout the process of working on this paper and many very helpful comments, mathematical and structural, on various versions of this paper.

	\section{Preliminaries}
	
	We briefly introduce some of the most important notions employed throughout this paper. These concern the fractional Laplacian, fractional divergences and gradients as well as a Wente-type result for fractional div-curl-structures as seen in \cite{mazoschi}.
	
	\subsection{Fractional Laplacian and Triebel-Lizorkin Spaces}
	
	In this section, we introduce the Triebel-Lizorkin spaces on the unit circle $S^1 \subset \R^2$ and recall some properties of the fractional Laplacian. Much of the current presentation is due to \cite{schiwang} and \cite{schmeitrieb}.\\
	
	Let us recall the following first: $S^1 \simeq \R / 2 \pi \Z$ is equipped with a natural distance function given by:
	\begin{align}
		| x-y |^2 	&= | e^{ix} - e^{iy} |^2 = | e^{i(x-y)} - 1 |^2 \notag \\
				&= (\cos(x-y) - 1)^2 + \sin(x-y)^2 = 2 - 2 \cos(x-y) \notag \\
				&= 4 \sin \left( \frac{x-y}{2} \right)^2,
	\end{align}
	so we have:
	$$| x-y | = 2 \left| \sin \left( \frac{x-y}{2} \right) \right|$$
	We shall tacitly use this distance function, whenever we are working over $S^1$. Moreover, we define for any $f: S^1 \to \R$:
	$$\mathcal{D}_{s,q}(f)(x) := \left( \int_{S^1} \frac{| f(x) - f(y) |^q}{| x-y |^{sq}} \frac{dy}{| x-y |} \right)^{1/q},$$
	for all $1 \leq q < \infty$ and $0 < s < 1$. Then:
	$$\| f \|_{\dot{W}^{s,(p,q)}(S^1)} := \| \mathcal{D}_{s,q}(f)(x) \|_{L^{p}(S^1)},$$
	for every $1 \leq p \leq \infty$. If $p = q$, these spaces correspond to the usual homogeneous Gagliardo-Sobolev spaces $\dot{W}^{s,p}(S^1)$. For a presentation of the operator $\mathcal{D}_{s,q}$ and its main properties, we refer to \cite{schiwang} and the references therein.\\
	
	We denote by $\mathcal{D}'(S^1)$ the collection of distributions on $S^1$ and sometimes denote by $\mathcal{D}(S^1)$ the space $C^{\infty}(S^1)$ of smooth functions. Let us from now on denote by $\hat{f}(k)$ the $k$-th Fourier coefficient of $f$, for all $f \in \mathcal{D}'(S^1)$:
	$$\hat{f}(k) := \frac{1}{2\pi} \langle f, e^{-ikx} \rangle = \frac{1}{2\pi} f \left( e^{-ikx} \right), \quad \forall k \in \Z$$
	One may also introduce the Triebel-Lizorkin spaces for $S^1$, denoted by ${F}^{s}_{p,q}(S^1)$ in the following way for all $s \in \R$, $p,q \in [1, \infty[$:
	$$F^{s}_{p,q}(S^1) := \big{\{} f \in \mathcal{D}'(S^1) \ \big{|}\ \| f \|_{F^{s}_{p,q}} < +\infty \big{\}}$$
	Here we write:
	$$\| f \|_{F^{s}_{p,q}} := \Bigg{\|} \Bigg{\|} \left( \sum_{k \in \mathbb{Z}} 2^{js} \varphi_{j}(k) \hat{f}(k) e^{ikx} \right)_{j \in \mathbb{N}} \Bigg{\|}_{l^{q}} \Bigg{\|}_{L^{p}(S^1)},$$
	for a partition of unity $(\varphi_{j})_{j \in \mathbb{N}}$ consisting of smooth, compactly supported functions on $\R$ satisfying:
	$$\operatorname{supp} \varphi_{0} \subset B_{2}(0), \quad \operatorname{supp} \varphi_{j} \subset \{ x \in \R\ |\ 2^{j-1} \leq | x | \leq 2^{j+1} \}, \forall j \geq 1$$
	as well as:
	$$\forall k \in \mathbb{N}: \sup_{j \in \mathbb{N}} 2^{jk} \| D^{k} \varphi_j \|_{L^{\infty}} \lesssim 1$$
	The Triebel-Lizorkin spaces on $S^1$, and more generally on the $n$-torus, possess an analogous theory to the classical case of these spaces on $\R^n$, see \cite{schmeitrieb}, Chapter 3. In particular, Sobolev embeddings continue to hold (\cite{schmeitrieb} Section 3.5.5), identifications with classical spaces such as $L^{p}(S^1)$ (\cite{schmeitrieb} Section 3.5.4) and duality results (\cite{schmeitrieb} Section 3.5.6). We shall use the properties of these spaces throughout this paper and shall refer to the given reference for details. The homogeneous spaces may be defined as well by omitting the Fourier coefficient of $0$th-order and adapting the notions accordingly.\\
	
	In \cite{schiwang}, the authors prove the following result:
	
	\begin{thm}[Theorem 1.4, \cite{schiwang}]
	\label{schiwangthm1.4}
		Let $s \in (0,1)$, $p,q \in ]1, \infty[$ and $f \in L^p(\R)$. Then:
		\begin{itemize}
			\item[(i)] We know $\dot{W}^{s, (p,q)}(\R^n) \subset \dot{F}^{s}_{p,q}(\R^n)$ together with:
			\begin{equation}
				\| f \|_{\dot{F}^{s}_{p,q}(\R^n)} \lesssim \| f \|_{\dot{W}^{s, (p,q)}(\R^n)}.
			\end{equation}
			
			\item[(ii)] If $p > \frac{nq}{n + sq}$, then we also have the converse inclusion together with:
			\begin{equation}
			\label{secondpartofthm2.1}
				\| f \|_{\dot{W}^{s, (p,q)}(\R^n)} \lesssim \| f \|_{\dot{F}^{s}_{p,q}(\R^n)}.
			\end{equation}
		\end{itemize}
		The constants depend on $s, p, q, n$.
	\end{thm}
	
	As seen in \cite{schiwang} and by using the properties in \cite{schmeitrieb}, \cite{triebel} for periodic functions, we can similarily discover the following equivalence with Triebel-Lizorkin spaces for all $1 < q < \infty$ and $1 < p < \infty$:
	\begin{equation}
		\dot{W}^{s, (p,q)}(S^1) = \dot{F}^{s}_{p,q}(S^1),
	\end{equation}
	with equivalence of the corresponding seminorms, provided $p > \frac{q}{1+sq}$. We shall prove the part of the identification that we will be using over and over, i.e. the second part of Theorem \ref{schiwangthm1.4}, in Appendix B. If $s = 1/2$ and $q = 2$, then $p > 1$ is the requirement in Theorem \ref{schiwangthm1.4} for the equality of $\dot{F}^{1/2}_{p,2}$ and $\dot{W}^{1/2, (p,2)}$ to hold. It should be observed that while $\dot{F}^{s}_{p,2}(S^1) \subset \dot{W}^{s,p}(S^1) = \dot{W}^{s,(p,p)}(S^1)$ for $p \geq 2$, there does not hold equality except for $p =2$. The arguments for the domain $S^{1}$ carry out in complete analogy to the case treated in Theorem 1.4 of \cite{schiwang}, where all the spaces are introduced over $\R$ and $\R^n$, by using the theory in \cite{schmeitrieb}. One just has to observe that the maximal function estimates used are also available on $S^1$, see Section 3.3.5 and 3.4 in \cite{schmeitrieb}, enabling the very same arguments to work. Therefore, Theorem \ref{schiwangthm1.4} continues to hold for $S^1$. We shall sometimes omit mention of the domain, if it is clear from the context.\\
	
	On $S^1$, the fractional $s$-Laplacian is defined as a Fourier multiplier operating on Fourier series:
	$$\widehat{(-\Delta)^{s} f}(k) = | k |^{2s} \hat{f}(k),$$
	for every $k \in \mathbb{Z}$ and all $0 < s < 1$. In particular, this can also be phrased as a principal value:
	$$(-\Delta)^{s} f(x) = C(s) \cdot P.V. \int_{S^{1}} \frac{f(x) - f(y)}{| x-y |^{1+2s}} dy,$$
	where $C(s) > 0$ denotes some constant depending on $s$. By the Fourier multiplier properties, fractional Laplacians interact in a natural way with Triebel-Lizorkin spaces $\dot{F}^{s}_{p,q}(S^1)$, as is usual for this type of function spaces. This means that it induces an isomorphism:
	$$(-\Delta)^{s}: \dot{F}^{t+2s}_{p,q} \to \dot{F}^{t}_{p,q},$$
	for all $p,q \in (1, \infty)$ and $t, t+2s \in \R$, see \cite{schmeitrieb} Section 3.6.3 and the proof of the analogous statement in the case $\R^n$.\\
	
	In analogy, the $s$-Laplacian can be defined on $\R$ as a Fourier multiplier using the Fourier transform rather than the Fourier series and leads again to an object which can also be characterised by a similar principal value. We omit the details, as the formulas are virtually the same as for the circle.

	\subsection{Fractional Gradients and Divergences}
	
	We present some of the notions introduced and studied in \cite{mazoschi}: We denote by $\mathcal{M}_{od}(\R \times \R)$ the collection of measurable functions $f: \R \times \R \to \R$ with respect to the measure $\frac{dx dy}{ | x-y |}$ and we do the same for $S^1$ instead of $\R$ on the domain-side. If both domains are possible, we shall merely denote this space by $\mathcal{M}_{od}$. For a measurable function $f: \R \to \R$ or $f: S^1 \to \R$, we define for $0 \leq s < 1$ the fractional $s$-gradient as follows:
	$$d_s f(x,y) = \frac{f(x) - f(y)}{| x-y |^s} \in \mathcal{M}_{od},$$
	and the corresponding $s$-divergence by means of duality. It is immediately clear, but nevertheless useful to observe:
	$$d_{s} f(y,x) = - d_{s} f(x,y)$$
	Observe that by duality, for $F \in \mathcal{M}_{od}(\R \times \R)$ or $F \in \mathcal{M}_{od}(S^1 \times S^1)$, we define for every $\varphi$ smooth and compactly supported on $\R$ or just smooth on $S^1$ in the latter case:
	$$\div_{s} F (\varphi) := \int \int F(x,y) d_{s} \varphi(x,y) \frac{dx dy}{| x-y |}$$
	This quantity is hence defined merely in a distributional sense. Lastly, we denote for $F, G \in \mathcal{M}_{od}$ over $\R$ or $S^1$:
	$$F \cdot G (x) := \int F(x,y) G(x,y) \frac{dy}{| x-y |}$$
	If $F = G$, we also write:
	$$F \cdot F(x) = | F |^{2}(x) \Rightarrow | F |(x) := \sqrt{F \cdot F(x)}$$
	Therefore, we immediately have:
	$$\| | d_{s} f | \|_{L^{p}(S^1)} = \| f \|_{\dot{W}^{s,(p,2)}(S^1)},$$
	which hints at an intimate connection between Triebel-Lizorkin spaces $\dot{F}^{s}_{p,q}(S^1)$ and fractional gradient $d_{s}$, under some technical conditions on $s,p,q$. We highlight that:
	$$(- \Delta)^{s} f = \div_{s} d_s f,$$
	which is particularily useful for the weak formulation of PDEs involving non-local operators. This equation is to be understood in the following sense:
	$$\int d_s f \cdot d_s g (x) dx = \int (-\Delta)^{s} f \cdot g dx = \int (-\Delta)^{s/2} f \cdot (-\Delta)^{s/2} g dx,$$
	for the domains $S^1$ and $\R$. Lastly, the following identity, sometimes referred to as fractional Leibniz' rule, is often useful:
	$$d_{s} \left( fg \right)(x,y) = d_{s} f(x,y) g(x) + f(y) d_{s} g(x,y)$$
	This identity can be verified by directly inserting the definition.\\
	
	In general, we may also introduce $L^{p}_{od}(S^1 \times S^1)$ or $L^{p}_{od}(\R \times \R)$ as the collection of measurable functions, such that the following norm is finite:
	$$\| F \|_{L^{p}_{od}} := \left( \int \int | F(x,y) |^{p} \frac{dy dx}{| x-y |} \right)^{1/p},$$
	for $1 \leq p < \infty$. The space $L^{\infty}_{od}(S^1 \times S^1)$ and $L^{\infty}_{od}(\R \times \R)$ could be introduced in the usual manner.\\
	
	One of the main results we shall be using later on in an appropriately modified formulation is the following non-local Wente-type result:
	
	\begin{lem}[Theorem 2.1, \cite{mazoschi}]
	\label{fractionalwentelemma}
		Let $s \in (0,1)$ and $p \in (1,\infty)$. For $F \in L^{p}_{od}(\R \times \R)$ and $g \in \dot{W}^{s,p'}(\R)$, where $p'$ denotes the H\"older dual of $p$, we assume that $\div_s F = 0$. Then $F \cdot d_s g$ lies in the Hardy space $\mathcal{H}^{1}(\R)$\footnote{We briefly recall that the Hardy space $\mathcal{H}^{1}(\R)$ is the subspace of $L^{1}(\R)$-functions such that: $$M_{\Phi}(f)(x) := \sup_{t > 0} | \Phi_{t} \ast f |(x) \in L^{1}(\R),$$ where $\Phi$ is a Schwartz function on $\R$ with $\int \Phi dx = 1$ and $\Phi_{t}(x) = 1/t \cdot  \Phi(x/t)$. Alternative characterisations using boundary values of harmonic maps, as the dual of $BMO(\R)$ and by the theory of function spaces exist. Hardy spaces are of interest, as they remedy some of the issues that appear when working with $L^{1}$-functions.} and we have the estimate:
		$$\| F \cdot d_s g \|_{\mathcal{H}^{1}(\R)} \lesssim \| F \|_{L^{p}_{od}(\R \times \R)} \cdot \| g \|_{\dot{W}^{s,p'}(\R)}.$$
	\end{lem}
	
	If, for example $s = 1/2$ and $p = p' = 2$, then we may also conclude that $F \cdot d_s g \in H^{-1/2}(\R)$ using the embedding of $\dot{H}^{1/2}(\R)$ into $BMO(\R)$. The estimate continues to hold in a similar manner. Similarily, we may deduce the following for the domain $S^1$:
	
		\begin{lem}
	\label{fractionalwentelemmaons1insec2}
		For $F \in L^{2}_{od}(S^1 \times S^1)$ and $g \in \dot{H}^{1/2}(S^1)$, we assume that $\div_{1/2} F = 0$. Then $F \cdot d_{1/2} g$ lies in the space $H^{-1/2}(S^1)$ and we have the estimate:
		$$\| F \cdot d_{1/2} g \|_{{H}^{-1/2}(S^1)} \lesssim \| F \|_{L^{2}_{od}(S^1 \times S^1)} \cdot \| g \|_{\dot{H}^{1/2}(S^1)}.$$
	\end{lem}
	
	The proof of this result is postponed to Appendix B.

	\section{The Fractional Harmonic Flow with Values in $S^{n-1}$}
	
	This section is devoted to the proof of our main result. For convenience's sake, we restate it once more:
	
	\begin{thm}
	\label{mainresultsection3}
		Let $u_0 \in H^{1/2}(S^1;S^{n-1})$ be any initial data. There exists $\varepsilon > 0$, such that if:
		$$\| ( - \Delta)^{1/4} u_0 \|_{L^2(S^1)} \leq \varepsilon,$$
		then there exists a unique energy class solution $u: \R_{+} \times S^1 \to S^{n-1} \subset \R^n$ of the weak fractional harmonic gradient flow:
		$$u_t + (-\Delta)^{1/2} u = u | d_{1/2} u |^2,$$
		satisfying $u(0, \cdot) = u_0$ and the energy decay estimate:
		$$\| (-\Delta)^{1/4} u(t) \|_{L^2 (S^1)} \leq \| (-\Delta)^{1/4} u_0 \|_{L^{2}(S^1)}.$$
		In fact, $u$ is even smooth and for an appropriate subsequence $t_k \to \infty$, the sequence $u(t_k)$ converges weakly in $H^{1}(S^1)$ to a point.
	\end{thm}
	
	We observe that existence is already clear due to the result in \cite{schisirewang}. Therefore, it remains to check uniqueness, regularity and convergence for $t \to \infty$. We shall treat each of these three different aspects in a separate subsection.
	
	\subsection{The $1/2$-Harmonic Gradient Flow Equation}
	
	First, we would like to prove the equivalence of the formulations in \eqref{gradflowintermsoforthogonality} and \eqref{gradflowinintroduction}. To do this, we assume that $u \in L^{\infty}(\R_{+}; H^{1/2}(S^1))$ and $u_t \in L^2 (\R_{+}; L^{2}(S^1))$ is a solution of \eqref{gradflowintermsoforthogonality} such that $u(t,x) \in S^{n-1}$ for almost every $(t,x) \in \R_{+} \times S^1$. Therefore, it satisfies the following equation:
	$$d \pi(u) \left( u_t + (-\Delta)^{1/2} u \right) = 0,$$
	where $\pi: \R_{n} \setminus \{ 0 \} \to S^{n-1}, x \mapsto x / |x|$ is the closest point projection to $S^{n-1}$. This means:
	$$\int_{0}^{\infty} \int_{S^1} \left( u_t + (-\Delta)^{1/2} u \right) d\pi(u) \varphi dx dt = 0, \quad \forall \varphi \in C^{\infty}_{c}(\R_{+} \times S^1; \R^n)$$
	Letting $\varphi \in C^{\infty}_{c}(\R_{+} \times S^1; \R^n)$, we therefore have, using the notation $d\pi^{\perp} = Id - d \pi$:
	\begin{align}
		\int_{0}^{\infty} \int_{S^1} \left( u_t + (-\Delta)^{1/2} u \right) \varphi dx dt	&= \int_{0}^{\infty} \int_{S^1} \left( u_t + (-\Delta)^{1/2} u \right) d\pi^{\perp}(u) \varphi dx dt \notag \\
																&= \int_{0}^{\infty} \int_{S^1} u_t \cdot u \tilde{\varphi} +  d_{1/2} u \cdot d_{1/2} \left( u \tilde{\varphi} \right) dx dt \notag \\
																&= \int_{0}^{\infty} \int_{S^1} d_{1/2} u \cdot d_{1/2} \left( u \tilde{\varphi} \right) dx dt \notag \\
																&= \int_{0}^{\infty} \int_{S^1} \int_{S^1} d_{1/2}u(t,x,y) \cdot u(t,y) d_{1/2} \tilde{\varphi}(t,x,y) \frac{dy dx}{| x-y |} dt \notag \\
																&+ \int_{0}^{\infty} \int_{S^1} \int_{S^1} d_{1/2} u(t,x,y) d_{1/2} u(t,x,y) \tilde{\varphi}(t,x) \frac{dy dx}{| x-y |} dt \notag \\
																&= \int_{0}^{\infty} \int_{S^1} \int_{S^1} d_{1/2} u(t,x,y) d_{1/2} u(t,x,y) \langle u(t,x), {\varphi}(t,x) \rangle \frac{dy dx}{| x-y |} dt \notag \\
																&= \int_{0}^{\infty} \int_{S^1} u(t,x) | d_{1/2} u |^2(t,x) \cdot {\varphi}(t,x) dx dt
	\end{align}
	where we used that $u_t$ is a.e. tangential to $S^{n-1}$ (seen by using approximation by convolutions), $d\pi(x) v = v$ for all $v \perp x$ and $x \in S^{n-1}$ and $d\pi(x) x = 0$. The latter was used to write:
	$$d\pi^{\perp}(u(t,x)) \varphi(t,x) = \langle \varphi(t,x), u(t,x) \rangle u(t,x) =: \tilde{\varphi}(t,x) u(t,x),$$
	with $\tilde{\varphi} \in {H}^{1/2}(S^1;\R) \cap L^{\infty}(S^1)$ by direct computation. Observe that we implicitely used:
	\begin{align}
		\int_{0}^{\infty} 	&\int_{S^1} \int_{S^1} d_{1/2}u(t,x,y) \cdot u(t,y) d_{1/2} \tilde{\varphi}(t,x,y) \frac{dy dx}{| x-y |} dt \notag \\
									&= \int_{0}^{\infty} \int_{S^1} \int_{S^1} d_{1/2}u(t,x,y) \cdot \frac{u(t,x) + u(t,y)}{2} d_{1/2} \tilde{\varphi}(t,x,y) \frac{dy dx}{| x-y |} dt = 0,
	\end{align}
	since:
	$$d_{1/2}u(t,x,y) \cdot \left( u(t,x) + u(t,y) \right) = \frac{u(t,x) - u(t,y)}{| x-y |^{1/2}} \cdot \left( u(t,x) + u(t,y) \right) = \frac{| u(x) |^2 - | u(y) |^2}{| x-y |^{1/2}} = 0,$$
	since $u \in S^{n-1}$ for almost all $(t,x)$. Therefore, we have shown that:
	$$u_{t} + (-\Delta)^{1/2} u = u | d_{1/2} u |^{2} \quad \text{ in } \mathcal{D}'(\R_{+} \times S^1),$$
	which is the formulation provided in \eqref{gradflowinintroduction}. This proves the aforementioned equivalence between the two formulations.
	
	\subsection{Uniqueness}
	
	The first property we verify is uniqueness. As already mentioned in the introduction, the key idea is to first show uniqueness under slightly better regularity assumptions similar to \cite{struwe1}. Then, we use the fractional Wente-Lemma \ref{fractionalwentelemmaons1insec2} and argue similar to \cite{riv2} in order to show that energy class solutions of sufficiently small energy actually are slightly more regular and thus the uniqueness result for more regular solutions applies in this situation.
	
	\subsubsection{Uniqueness under Higher Regularity Assumptions}
	
	Let us assume that $u, v$ are two solutions to the fractional gradient flow taking a.e. values in $S^{n-1} \subset \R^n$ and such that the following holds true:
	\begin{equation}
	\label{regass}
		u,v \in L^{\infty}(\R_{+}; H^{1/2}(S^1)) \cap L^{2}_{loc}(\R_{+}; H^{1}(S^{1})); \quad u_t, v_t \in L^{2}(\R_{+}; L^{2}(S^{1}))
	\end{equation}
	The local integrability is meant with respect to the domain $[0, \infty[$. It should be noticed that we include more regularity than is actually required/given by the existence result in \cite{schisirewang}, which is in agreement with the uniqueness treatment in \cite{struwe1}. Additionally, it is easy to see thanks to $u, v \in S^{n-1}$ almost everywhere, that:
	$$u, v \text{ are bounded. }$$
	We assume that they satisfy the gradient flow associated with the $1/2$-harmonic map, which we have seen in the previous subsection to be equivalent to:
	\begin{equation}
	\label{gradflowuv}
		u_t + (-\Delta)^{1/2} u = u | d_{1/2} u |^2, \quad v_{t} + (-\Delta)^{1/2} v = v |d_{1/2} v |^2,
	\end{equation}
	together with the boundary condition:
	$$u(0, \cdot) = v(0, \cdot) = u_{0} \in H^{1/2}(S^1;S^{n-1})$$
	By the assumptions, we may evaluate the $1/2$-Laplacian for a.e. fixed time $t$ (as $\nabla u(t)$ for almost every fixed time $t$ is in $L^{2}(S^1)$), which shows that the gradient flow is satisfied in a strong sense by using fractional integration by parts on the weak formulation. Our goal is to prove the following result:
	
	\begin{thm}
	\label{uniqueness}
		Let $u,v$ as above be solutions to the fractional gradient flow with the same initial datum $u_0$. Assume that we have the following $1/2$-energy decay estimate:
		$$\| (-\Delta)^{1/4} u(t) \|_{L^{2}(S^1)}, \| (-\Delta)^{1/4} v(t) \|_{L^{2}(S^1)} \leq \| (-\Delta)^{1/4} u_{0} \|_{L^{2}(S^1)}, \quad \forall t \in \R_{+}$$
		Then we may conclude:
		$$u = v,$$
		i.e. the solutions agree for every time $t > 0$ as well.
	\end{thm}
	
	It will be clear from the proof of Theorem \ref{uniqueness} that it would also suffice to assume:
	$$\sup_{t \in \R_{+}} \| (-\Delta)^{1/4} u(t) \|_{L^{2}(S^1)} < + \infty,$$
	and similarily for $v$.
	
	\begin{proof}
		The proof relies on the same ideas as the proof of uniqueness under better regularity provided in Lemma 3.12 in \cite{struwe1} for the harmonic gradient flow. Therefore, we begin by defining $w := u - v$ and observe:
		$$w \in L^{\infty}(\R_{+}; H^{1/2}(S^1)) \cap L^{2}_{loc}(\R_{+}; H^{1}(S^{1})); \quad w_t \in L^{2}(\R_{+}; L^{2}(S^{1})),$$
		as well as the initial condition:
		\begin{equation}
		\label{initialdatumw}
			w(0, \cdot) = 0
		\end{equation}
		This is an immediate consequence of the regularity and initial data of $u$ and $v$. Let us now combine the equations in \eqref{gradflowuv} to determine the non-local PDE solved by $w$:
		\begin{align}
		\label{gradfloww}
			w_{t} + (-\Delta)^{1/2} w	&= u_{t} + (-\Delta)^{1/2} u - v_{t} - (-\Delta)^{1/2} v \notag \\
								&= u | d_{1/2} u |^{2} - v | d_{1/2} v |^2 \notag \\
								&= (u-v) | d_{1/2} u |^{2} + v ( | d_{1/2} u |^{2} - | d_{1/2} v |^2 ) \notag \\
								&= w | d_{1/2} u |^{2} + v ( | d_{1/2} u |^{2} - | d_{1/2} v |^2 ) =: R_1 + R_2
		\end{align}
		If we test \eqref{gradfloww} against $w$ itself, we obtain for any $T \in \R_{+}$:
		\begin{align}
		\label{estimate1}
			\int_{0}^{T} \int_{S^1} w_t \cdot w + (-\Delta)^{1/2} w \cdot w dx dt	&= \int_{0}^{T} \frac{d}{dt} \left( \frac{1}{2} \| w(t) \|_{L^{2}(S^1)} \right) dt + \int_{0}^{T} \| (-\Delta)^{1/4} w(t) \|_{L^{2}(S^1)} dt \notag \\
																&= \int_{0}^{T} \int_{S^1} | w |^2 |d_{1/2} u |^2 + v ( | d_{1/2} u |^{2} - | d_{1/2} v |^2 ) \cdot w dx dt \notag \\
																&\leq \int_{0}^{T} \int_{S^1} | w |^2 |d_{1/2} u |^2 dx dt + \int_{0}^{T} \int_{S^1} | w | | R_{2} | dx dt
		\end{align}
		So, using the fundamental theorem of calculus, we arrive at:
		\begin{equation}
		\label{estw1}
			\frac{1}{2} \| w(T) \|_{L^{2}(S^{1})} + \int_{0}^{T} \| (-\Delta)^{1/4} w(t) \|_{L^{2}(S^{1})} dt \leq \int_{0}^{T} \int_{S^1} | w |^2 |d_{1/2} u |^2 dx dt + \int_{0}^{T} \int_{S^1} | w | | R_{2} | dx dt
		\end{equation}
		We emphasise that we used \eqref{initialdatumw} in order to evaluate the integral of the derivative at $t = 0$. In order to proceed, we have to investigate the term $| d_{1/2} u |^{2} - | d_{1/2} v |^2$ more closely. To do this, let us write for $x \in S^1$ by means of the fundamental theorem:
		\begin{align*}
			| d_{1/2} u |^{2}(x) - | d_{1/2} v |^2(x)	&= \int_{0}^{1} \frac{d}{ds} \left( | d_{1/2} (v + s (u-v)) |^2(x) \right) ds \notag \\
										&= \int_{0}^{1} \frac{d}{ds} \left( \int_{S^1} \frac{\left( v(x) - v(y) + s (u(x) - v(x) - u(y) + v(y)) \right)^2}{| x-y |^{2}} dy \right) ds \notag \\
										&= \int_{0}^{1} \frac{d}{ds} \left( \int_{S^1} \frac{\left( v(x) - v(y) + s (w(x) - w(y)) \right)^2}{| x-y |^{2}} dy \right) ds \notag \\
										&= \int_{0}^{1} \int_{S^{1}} 2 \frac{(v(x) - v(y) + s(w(x) - w(y)))(w(x) - w(y))}{| x-y |^2} dy ds \notag \\
										&\leq 2 \int_{0}^{1} | d_{1/2}((1-s)v + s u) |(x) \cdot | d_{1/2} w |(x) ds \notag \\
										&\leq C \left( | d_{1/2} u |(x) + | d_{1/2} v |(x) \right) \cdot | d_{1/2} w |(x)
		\end{align*}
		where we used H\"older's inequality and the integrability properties of $u,v$ during the sequence of inequalities above. This implies the following estimate for $R_2$:
		\begin{align}
			| R_{2} |(x) 	&\leq C | v |(x) \left( | d_{1/2} u |(x) + | d_{1/2} v |(x) \right) \cdot | d_{1/2} w |(x) \notag \\
						&\leq C \left( | d_{1/2} u |(x) + | d_{1/2} v |(x) \right) \cdot | d_{1/2} w |(x),
		\end{align}
		where we implicitely used $| v | = 1$ almost everywhere. By using Cauchy-Schwarz and Young's inequality, we therefore find:
		$$\int_{0}^{T} \int_{S^1} | w | | R_{2} | dx dt \leq \delta \int_{0}^{T} \int_{S^{1}} | d_{1/2} w |^{2} dx dt + C(\delta) \int_{0}^{T} \int_{S^{1}} | w |^{2} \left( | d_{1/2} u | + | d_{1/2} v | \right)^{2} dx dt,$$
		for any $\delta > 0$. Observe that:
		$$\int_{0}^{T} \int_{S^{1}} | d_{1/2} w |^{2} dx dt \sim \int_{0}^{T} \| (-\Delta)^{1/4} w \|_{L^2(S^{1})} dt,$$
		by direct computations, see Lemma \ref{fracgradest} after this proof. Thus, we may choose $\delta > 0$, such that after absorbing and using Cauchy-Schwarz:
		\begin{align}
		\label{estw2}
		\frac{1}{2} \| w(T) \|_{L^{2}(S^{1})} 	&+ \frac{1}{2} \int_{0}^{T} \| (-\Delta)^{1/4} w(t) \|_{L^{2}(S^{1})} dt \notag \\
									&\leq C \int_{0}^{T} \int_{S^{1}} | w |^{2} \left( | d_{1/2} u | + | d_{1/2} v | \right)^{2} dx dt \notag \\
									&\leq \tilde{C} \left( \int_{0}^{T} \int_{S^{1}} | w |^{4} dx dt \right)^{1/2} \cdot \left(  \int_{0}^{T} \int_{S^{1}} \left( | d_{1/2} u | + | d_{1/2} v | \right)^{4} dx dt \right)^{1/2}
		\end{align}
		Using Lemma \ref{ladys1} below for each fixed $t$, we can estimate:
		\begin{align}
		\label{estl4w}
			\int_{0}^{T} \int_{S^{1}} | w |^{4} dx dt		&\leq C \int_{0}^{T} \| w(t) \|_{L^{2}(S^{1})}^{2} \cdot \left( \| w(t) \|_{L^{2}(S^{1})}^{2} + \| (-\Delta)^{1/4} w(t) \|_{L^{2}(S^1)}^{2} \right) dt \notag \\
											&\leq \tilde{C} \left( \sup_{0 \leq s \leq T} \| w(s) \|_{L^{2}(S^{1})}^{2} + \int_{0}^{T} \| (-\Delta)^{1/4} w(t) \|_{L^{2}(S^{1})}^{2} dt \right)^2
		\end{align}
		Here, $\tilde{C}$ depends on $T$, which may be chosen sufficiently small, as seen afterwards, by using an iteration process to increase $T$ step by step. We notice that there is no dependence of the constant on the $L^2$-norm of $w(s)$, as we estimate:
		\begin{align}
			\int_{0}^{T} \| w(t) \|_{L^{2}(S^{1})}^{2} 	&\cdot \left( \| w(t) \|_{L^{2}(S^{1})}^{2} + \| (-\Delta)^{1/4} w(t) \|_{L^{2}(S^1)}^{2} \right) dt \notag \\
											&\leq \int_{0}^{T} \sup_{0 \leq s \leq T} \| w(s) \|_{L^{2}(S^1)}^2 \cdot \left( \sup_{0 \leq s \leq T} \| w(s) \|_{L^{2}(S^{1})}^{2} + \| (-\Delta)^{1/4} w(t) \|_{L^{2}(S^1)}^{2}\right) dt \notag \\
											&= \sup_{0 \leq s \leq T} \| w(s) \|_{L^{2}(S^1)}^2 \cdot \left( T \cdot \sup_{0 \leq s \leq T} \| w(s) \|_{L^{2}(S^{1})}^{2} + \int_{0}^{T} \| (-\Delta)^{1/4} w(t) \|_{L^{2}(S^1)}^{2}\right) dt \notag \\
											&\leq \tilde{C} \left( \sup_{0 \leq s \leq T} \| w(s) \|_{L^{2}(S^{1})}^{2} + \int_{0}^{T} \| (-\Delta)^{1/4} w(t) \|_{L^{2}(S^{1})}^{2} dt \right)^2
		\end{align}
		which is precisely the estimate presented above and the dependence on $T$ is benign, i.e. $\tilde{C}$ remains bounded as $T \to 0$. Observe that we used the inequality:
		$$\sup_{0 \leq s \leq T} \| w(s) \|_{L^{2}(S^1)}^2 \leq  \sup_{0 \leq s \leq T} \| w(s) \|_{L^{2}(S^{1})}^{2} + \int_{0}^{T} \| (-\Delta)^{1/4} w(t) \|_{L^{2}(S^{1})}^{2} dt,$$
		which is trivially true.\\
		
		\textbf{Claim 1:} For every $\varepsilon > 0$, there is $T > 0$ small enough, such that:
		\begin{equation}
		\label{suffsmall}
			\int_{0}^{T} \int_{S^{1}} \left( | d_{1/2} u | + | d_{1/2} v | \right)^{4} dx dt < \varepsilon,
		\end{equation}

		However, before we prove this claim, we observe that this is indeed sufficient to conclude our proof of Theorem \ref{uniqueness}, as then we may choose $T > 0$ as in the proof of Lemma 3.12 in \cite{struwe1} to maximize the $L^{2}(S^1)$-norm on $[0,T]$, i.e. $\| w(T) \|_{L^{2}(S^1)} = \sup_{0 \leq s \leq T} \| w(s) \|_{L^{2}(S^{1})}^{2}$. This is possible by continuity of $t \mapsto w(t)$ with respect to the $L^2$-norm due to the assumptions in \eqref{regass}, in particular the integrability of the weak derivative in time-direction $u_t, v_t$ and thus $w_t$. We refer to \cite{evans}, Chapter 5.9.2 for details regarding this continuity. Alternatively, \eqref{regass} implies that $u,v$ and therefore also $w$ are in $H^{1}_{loc}(S^1 \times ]0,\infty[)$. Thus, by the trace theorem $u(t, \cdot), v(t,\cdot), w(t,\cdot)$ lie in $H^{1/2}(S^1)$ and depend continuously on $t$, the latter property owing to the continuity of the trace operator.\\

		In addition, for some fixed, but arbitrary, a-priori time $0 < T_0 \leq T$, $T$ as in the claim, the maximum of the $L^2$-norms at a given time over the interval $[0,T_0]$ must be attained at some $t_0 > 0$, as otherwise we have $w = 0$ for all times $t \leq T_0$, which would also show the desired equality and hence we could restart the argument from $T_0$ on. Thus, for sufficiently small $\varepsilon > 0$, we may absorb the right hand side in \eqref{estw2} into the left hand side, immediately giving the desired result, i.e. $w(T) = 0$ and thus for all $0 \leq t \leq T$. More precisely, from \eqref{estw1}, \eqref{estw2} and a sufficiently small $\varepsilon$, we obtain:
		\begin{align*}
			\frac{1}{2} &\| w(T) \|_{L^{2}(S^{1})} + \int_{0}^{T} \| (-\Delta)^{1/4} w(t) \|_{L^{2}(S^{1})} dt \notag \\
										&\leq C \delta \int_{0}^{T} \| (-\Delta)^{1/4} w \|_{L^2(S^{1})} dt + \tilde{C}(\delta) \sqrt{\varepsilon} \left( \sup_{0 \leq s \leq T} \| w(s) \|_{L^{2}(S^{1})}^{2} + \int_{0}^{T} \| (-\Delta)^{1/4} w(t) \|_{L^{2}(S^{1})}^{2} dt \right),
		\end{align*}
		which for $\delta > 0$ sufficiently small and observing maximality of the choice of $T$ (note that we may choose $T$ smaller than required by \eqref{suffsmall} to ensure that it is also the time with maximal $L^{2}$-norm) then becomes:
		\begin{align*}
			\frac{1}{2} \sup_{0 \leq s \leq T} \| w(s) \|_{L^{2}(S^{1})} 	&+ \frac{1}{2} \int_{0}^{T} \| (-\Delta)^{1/4} w(t) \|_{L^{2}(S^{1})} dt \notag \\
										&\leq \tilde{C} \sqrt{\varepsilon} \left( \sup_{0 \leq s \leq T} \| w(s) \|_{L^{2}(S^{1})}^{2} + \int_{0}^{T} \| (-\Delta)^{1/4} w(t) \|_{L^{2}(S^{1})}^{2} dt \right),
		\end{align*}
		and if $\tilde{C} \sqrt{\varepsilon} = \frac{1}{4}$, we may absorb this contribution into the left hand side to find:
		$$\sup_{0 \leq s \leq T} \| w(s) \|_{L^{2}(S^{1})} = 0$$
		The argument then is completed by the usual connectedness-type argument using convergence $u(t) \to u(t_0)$ for $t \to t_0$ in $L^{2}$ and iterating, see \cite{struwe1}. Indeed, consider the set:
		$$I \subset [0, \infty[, \quad I := \big{\{} t \in [0, \infty[ \ \big{|}\ w(s) = 0, \forall s \leq t \big{\}}$$
		Clearly, $0 \in I$ by construction and hence $I \neq \emptyset$. Moreover, $I$ is open, since if $t \in I$, we may use the arguments above to deduce that $w(s) = 0$ for all $t \leq s < t + \varepsilon$, for some sufficiently small $\varepsilon$. Finally, $I$ is closed, which then shows $I = [0, \infty [$ and finishes the argument. This is clear as:
		$$\lim_{s \to t} \| w(s) \|_{L^{2}} = \| w(t) \|_{L^2},$$
		which proves that if all $s < t$ satisfy $s \in I$, then $w(t) = 0$ in $L^2(S^1)$ and hence $t \in I$.\\
		
		\textbf{Proof of Claim 1:} Now, let us return to \eqref{suffsmall}. We shall provide two justifications of this estimate. The first argument postponed to Appendix A relies on some properties found in \cite{dalio}, \cite{daliomartriv} connecting the fractional Laplacian on the circle to the one on the real line. The precise results shall be stated and proven in Appendix A. A different apporach uses Theorem \ref{schiwangthm1.4} for $S^1$ directly. Here, we shall just present an outline of the argument:\\
		
		We observe that it suffices to find corresponding estimates for $d_{1/2} u$ and $d_{1/2} v$ respectively. For these, we have:
		\begin{align}
		\label{estd1/2u}
			\int_{0}^{T} \int_{S^1} | d_{1/2} u |^4 dx dt	&\leq C \int_{0}^{T} \| (- \Delta)^{1/4} u(t) \|_{L^{4}}^4 dt \notag \\
											&\leq \tilde{C} \int_{0}^{T} \| (-\Delta)^{1/4}u(t) \|_{L^{2}}^{2} \cdot \| (-\Delta)^{1/2} u \|_{L^{2}}^2 dt \notag \\
											&\leq \tilde{C} \sup_{0 \leq s \leq T} \| (-\Delta)^{1/4} u(s) \|_{L^2}^2 \cdot \int_{0}^{T} \int_{S^{1}} | (-\Delta)^{1/2} u |^2 dx dt \notag \\
											&\leq \tilde{C} \| (-\Delta)^{1/4} u_{0} \|_{L^2}^2 \cdot \int_{0}^{T} \int_{S^{1}} | (-\Delta)^{1/2} u |^2 dx dt 
		\end{align}
		where we used $u \in H^{1}(S^1)$, which immediately implies $(-\Delta)^{1/4} u \in H^{1/2}(S^1)$, as well as Lemma \ref{ladys1}. 
		We refer to Appendix A for the details on the proof of the estimate using an extension procedure, in particular the first inequality which is actually the only missing step here. Alternatively, directly using the second part of Theorem \ref{schiwangthm1.4} on the domain $S^1$, see \eqref{secondpartofthm2.1}, as described in the preliminary section and proven in Appendix B, the first inequality could also be obtained immediately and the rest follows by using Lemma \ref{ladys1}. The main difference between these two approaches lies in the use of Theorem \ref{schiwangthm1.4} either on $\R$ or on $S^1$, depending on which techniques are used.
		
		The claim now follows by \eqref{regass} and the $L^{2}$-integrability of the $1/2$-Laplacian of $u$. Notice that the supremum is finite due to the assumptions in the statement of the Theorem.
	\end{proof}
	
	We now prove some useful results that were invoked in the proof above or motivate the conditions of the main result of this subsection, Theorem \ref{uniqueness}:
	
	\begin{lem}
	\label{ladys1}
		Let $u \in H^{1/2}(S^{1})$. Then the following estimate holds for some $C > 0$:
		$$\| u \|_{L^{4}} \leq C \| u \|_{L^{2}}^{1/2} \| u \|_{H^{1/2}}^{1/2}$$
	\end{lem}
	
	\begin{proof}
		By Sobolev embeddings, we immediately find for some $C > 0$:
		$$\| u \|_{L^4} \leq C \| u \|_{H^{1/4}}$$
		Additionally, we have by definition:
		\begin{align*}
			\| u \|_{H^{1/4}}^{2}	&= \sum_{n \in \Z} ( 1 + |n|^2 )^{1/4} | \hat{u}(n) |^2 \notag \\
							&= \sum_{n \in \Z} ( 1 + |n|^2 )^{1/4} | \hat{u}(n) | \cdot | \hat{u}(n) | \notag \\
							&\leq \left( \sum_{n \in \Z} ( 1 + |n|^2 )^{1/2} | \hat{u}(n) |^2 \right)^{1/2} \cdot \left( \sum_{n \in \Z} | \hat{u}(n) |^2 \right)^{1/2} \notag \\
							&= \| u \|_{H^{1/2}} \cdot \| u \|_{L^{2}}
		\end{align*}
		This now yields:
		$$\| u \|_{L^4} \leq C \| u \|_{H^{1/4}} \leq C \sqrt{\| u \|_{H^{1/2}} \cdot \| u \|_{L^{2}}} = C \| u \|_{L^{2}}^{1/2} \| u \|_{H^{1/2}}^{1/2}$$
		Thus, the Lemma is proven.
	\end{proof}
	
	We highlight that Lemma \ref{ladys1} continues to be true on $\R$ by using classical rescaling techniques or relying, for example, on Littlewood-Paley theory.
	
	\begin{lem}
	\label{fracgradest}
		It holds the following for every $u \in H^{1/2}(S^1)$:
		\begin{equation}
			\int_{S^1} | d_{1/2} u |^2 dx \sim \| (-\Delta)^{1/4} u \|_{L^2(S^1)}^2
		\end{equation}
	\end{lem}
	
	\begin{proof}
		Let us observe the following for smooth functions $u$:
		\begin{align}
			\int_{S^1} | d_{1/2} u |^2 dx	&= \int_{S^1} \int_{S^1} \frac{| u(x) - u(y) |^2}{| x-y |^2} dydx \notag \\
									&= \int_{S^{1}} P.V. \int_{S^1} \frac{(u(x) - u(y))}{| x-y |^2} (u(x) - u(y)) dy dx \notag \\
									&= \int_{S^{1}} 2 P.V. \int_{S^1} \frac{u(x) - u(y)}{|x-y|^2} dy \cdot u(x) dx \notag \\
									&= \tilde{C} \int_{S^{1}} (-\Delta)^{1/2}u(x) \cdot u(x) dx \notag \\
									&= \tilde{C} \int_{S^{1}} | (-\Delta)^{1/4} u(x) |^2 dx \notag \\
									&= \tilde{C} \| (-\Delta)^{1/4} u \|_{L^2(S^1)}^2
		\end{align}
		for some $\tilde{C} > 0$, see also the definition of the fractional Lapacian in section 2. Here, $P.V.$ stands for principal value. For complete rigor, one has to take the integral on a subset of $S^1 \times S^1$ omitting the diagonal and letting the neighbourhood become arbitrarily small to deduce the second equality, to ensure the principal value can be taken and the fractional Laplacian emerges. The statement for general $u$ follows now by approximation.
	\end{proof}

	Finally, let us motivate the decay assumption on solutions of the fractional gradient flow in Theorem \ref{uniqueness}:
	$$\| (-\Delta)^{1/4} u(t) \|_{L^{2}(S^1)}\leq \| (-\Delta)^{1/4} u_{0} \|_{L^{2}(S^1)}, \quad \forall t \in \R_{+}$$
	It should be noted that this is a "classical" assumption when working with gradient flows, nevertheless we present the idea: To do this, let us assume that $u$ is a smooth solution of the fractional gradient flow. Then, we may test against $u_{t}$ and find:
	\begin{equation}
		\int_{0}^{T} \int_{S^1} | u_t |^2 + (-\Delta)^{1/2}u \cdot u_t dx dt = \int_{0}^{T} \int_{S^1} u | d_{1/2} u |^2 \cdot u_t dx dt = 0,
	\end{equation}
	where the last equality follows by observing that $u$ assumes values in a sphere, hence the derivative in $t$-direction will be tangential to the sphere and, as a result, orthogonal to $u$, implying:
	$$u \cdot u_{t} = 0$$
	In addition, we have:
	\begin{align}
		\int_{0}^{T} \int_{S^1} (-\Delta)^{1/2}u \cdot u_t dx dt	&= \int_{0}^{T} \int_{S^1} (-\Delta)^{1/4}u \cdot (-\Delta)^{1/4} u_t dx dt \notag \\
												&= \int_{0}^{T}\frac{1}{2} \frac{d}{dt} \left( \int_{S^1} | (-\Delta)^{1/4}u |^2 dx \right) dt \notag \\
												&= \frac{1}{2} \| (-\Delta)^{1/4} u(T) \|_{L^2(S^1)}^{2} - \frac{1}{2} \| (-\Delta)^{1/4} u(0) \|_{L^2(S^1)}^{2} \notag. \\
												&= \frac{1}{2} \| (-\Delta)^{1/4} u(T) \|_{L^2(S^1)}^{2} - \frac{1}{2} \| (-\Delta)^{1/4} u_0 \|_{L^2(S^1)}^{2}
	\end{align}
	Consequently, this computation shows that in the case of regular solutions:
	\begin{equation}
		\frac{1}{2} \| (-\Delta)^{1/4} u(T) \|_{L^2(S^1)}^{2} \leq \frac{1}{2} \| (-\Delta)^{1/4} u(T) \|_{L^2(S^1)}^{2} + \int_{0}^{T} \int_{S^1} | u_t |^2 dx dt = \frac{1}{2} \| (-\Delta)^{1/4} u_0 \|_{L^2(S^1)}^{2}
	\end{equation}
	This yields the desired boundedness of energy (in fact monotone decay of energy) and thus motivates the assumption we had in Theorem \ref{uniqueness}. We formulate this in the following slightly imprecise:
	
	\begin{lem}
	\label{monotone1/2energydecayinglemma}
		Let $u$ be a sufficiently regular solution of the $1/2$-harmonic gradient flow as previously defined with $u(0, \cdot) = u_0$. Then the following holds for all $T \geq 0$:
		$$\frac{1}{2} \| (-\Delta)^{1/4} u(T) \|_{L^2(S^1)}^{2} \leq \frac{1}{2} \| (-\Delta)^{1/4} u_0 \|_{L^2(S^1)}^{2}$$
		In fact, the energy $T \mapsto \| (-\Delta)^{1/4} u(T) \|_{L^2(S^1)}$ monotonically decreases in $T$.
	\end{lem}

	\subsubsection{Improved Regularity of the Solution}
	
	We would like to show how we may obtain the required improvement in regularity for energy-class solutions, i.e. solutions which do a-priori not satisfy a $L^2$-local bound on the first derivative in space-direction, to the fractional gradient flow \eqref{gradflowuv} in a similar manner as in \cite{riv}. The key idea is that we may fix some time $t$ and consider the corresponding equation for fixed time to obtain improved regularity. Namely, we will obtain the following result:
	
	\begin{thm}
	\label{uniqueinenergyclass}
		Let $u: \R_{+} \times S^1 \to S^{n-1} \subset \R^n$ be a solution of the weak fractional harmonic gradient flow \eqref{gradflowuv} with initial datum $u_0 \in H^{1/2}(S^1)$ and satisfying the following regularity assumptions:
		$$u \in L^{\infty}(\R_{+}; H^{1/2}(S^1)); \quad u_t \in L^{2}(\R_{+}; L^{2}(S^{1}))$$
		Then there exists $\varepsilon > 0$ such that among all such $u$ satisfying the smallness condition:
		$$\| ( - \Delta)^{1/4} u(t) \|_{L^2(S^1)} \leq \varepsilon, \quad \forall t \in \R_{+},$$
		the solution to the fractional harmonic gradient flow \eqref{gradflowuv} with initial datum $u_0$ is unique.
	\end{thm}
	
	If we assume that the energy is bounded by some sufficiently small $\varepsilon > 0$, then it is sufficient to show that $u(t) \in H^{1}(S^1)$ for almost every $t \in \R_{+}$. In fact, the following holds:
	
	\begin{prop}
	\label{uniqueinenergyclassunderh1assumption}
		Let $u: \R_{+} \times S^1 \to S^{n-1} \subset \R^n$ be a solution of the weak fractional harmonic gradient flow \eqref{gradflowuv} with initial datum $u_0 \in H^{1/2}(S^1)$ and satisfying the following regularity assumptions:
		$$u \in L^{\infty}(\R_{+}; H^{1/2}(S^1)); \quad u_t \in L^{2}(\R_{+}; L^{2}(S^{1})); \quad u(t) \in H^{1}(S^1) \text{ for a.e. } t \in \R_{+}$$
		Then there exists $\varepsilon > 0$ such that among all such $u$ satisfying the smallness condition:
		$$\| ( - \Delta)^{1/4} u(t) \|_{L^2(S^1)} \leq \varepsilon, \quad \forall t \in \R_{+},$$
		the solution to the fractional harmonic gradient flow \eqref{gradflowuv} with initial datum $u_0$ is unique.
	\end{prop}
	
	\begin{proof}
	To verify this, let us observe that if $u(t) \in H^1(S^1)$ for almost every $t \in \mathbb{R}_{+}$, we may deduce for a fixed time $t$:
	$$(-\Delta)^{1/2} u(t) = u(t) | d_{1/2} u(t) |^{2} - \partial_{t} u(t)$$
	Hence, by standard elliptic estimates for the fractional Laplacian or simply observing that with $\mathcal{R}$ being the Riesz transform, we have:
	\begin{equation}
		\nabla u(t) = \mathcal{R} \left( u(t) | d_{1/2} u(t) |^{2} - \partial_{t} u(t) \right)
	\end{equation}
	Keeping in mind that $\mathcal{R}$ is a continuous linear operator on $L^{2}(S^1)$, we are led to the following estimate:
	\begin{align}
	\label{estforh1}
		\| u(t) \|_{H^{1}(S^1)}^2 	&\leq C \left( \| u(t) \|_{L^2}^2 + \| | d_{1/2} u(t) |^{2} \|_{L^2}^2 + \| \partial_{t} u(t) \|_{L^2}^2 \right) \notag \\
							&\leq C \left( 1 + \| | d_{1/2} u(t) |^{2} \|_{L^2}^2 + \| \partial_{t} u(t) \|_{L^2}^2 \right),
	\end{align}
	where we used $u(t) \in S^{n-1}$ almost everywhere for almost every time $t$. It is clear that regarding local $L^2$-integrability with respect to time, it thus remains to study the following contribution:
	$$\| | d_{1/2} u(t) |^{2} \|_{L^2}^2$$
	Using the same ideas as in the proof of \eqref{estd1/2u} for the uniqueness statement Theorem \ref{uniqueness} (which are proved in Appendix A or rely on the second part of Theorem \ref{schiwangthm1.4} for $S^1$, see also Appendix B), we may estimate this term by:
	$$\| | d_{1/2} u(t) |^{2} \|_{L^2}^2 \leq C' \| u(t) \|_{H^{1/2}}^{2} \| u(t) \|_{H^{1}}^{2}$$
	By applying this inequality to $u - \hat{u}(0)$ instead of $u$, we may replace the $H^{1/2}$- and $H^{1}$-norms by the corresponding seminorms:
	$$\| | d_{1/2} u(t) |^{2} \|_{L^2}^2 \leq C' \| u(t) \|_{\dot{H}^{1/2}}^{2} \| u(t) \|_{\dot{H}^{1}}^{2}$$
	We emphasise that adding a constant to $u$ does not affect the LHS of the estimate above.
	Therefore, we have the energy term appearing:
	$$\| | d_{1/2} u(t) |^{2} \|_{L^2}^2 \leq C' \| (-\Delta)^{1/4}u(t) \|_{L^{2}}^{2} \| u(t) \|_{H^{1}}^{2} \leq C' \varepsilon \cdot \| u(t) \|_{H^{1}}^{2},$$
	where $\varepsilon > 0$ is an a priori energy estimate as in \cite{riv} and we may still choose $\varepsilon > 0$ appropriately. Indeed, if $\varepsilon > 0$ is sufficiently small, for example $\varepsilon \leq 1/(2C C')$, we may absorb this term in the left hand side of \eqref{estforh1} to arrive at:
	\begin{equation}
		(1 - CC' \varepsilon ) \cdot \| u(t) \|_{H^{1}(S^1)}^2 \leq \tilde{C} \left( 1 + \| \partial_{t} u(t) \|_{L^2}^2 \right) \Rightarrow \| u(t) \|_{H^1} \leq \frac{C}{1 - C' C \varepsilon} \left( 1 + \| \partial_{t} u(t) \|_{L^2} \right),
	\end{equation}
	which thus yields an estimate for the $H^1$-norm. We observe that hence, by the integrability properties of $\partial_{t} u$ and the constant function (which rely on the compactness of $S^1$):
	\begin{equation}
	\label{reglemmaestimatewithint}
		u \in L^{2}_{loc}(\R_{+}; H^{1}(S^1))
	\end{equation}
	Thus, we may apply the previous uniqueness statement in Theorem \ref{uniqueness} even if we merely know:
	$$u \in L^{\infty}(\R_{+}; H^{1/2}(S^1)); \quad u_t \in L^{2}(\R_{+}; L^{2}(S^{1})); \quad \| ( - \Delta)^{1/4} u(0) \|_{L^2(S^1)} \leq \varepsilon,$$
	with $\varepsilon > 0$ sufficiently small as above and assuming the energy decrease holds, provided we get increased regularity for $u(t)$.
	\end{proof} 
	
	In particular, if we assume that the $1/2$-energy is non-increasing in time, as seen to be true for smooth solutions to the fractional harmonic gradient flow in Lemma \ref{monotone1/2energydecayinglemma}, the smallness condition could be rephrased as:
	$$\| ( - \Delta)^{1/4} u_0 \|_{L^2(S^1)} \leq \varepsilon$$
	
	Consequently, all that remains is to deduce $H^1$-regularity for a.e. fixed time to apply Proposition \ref{uniqueinenergyclassunderh1assumption} and deduce Theorem \ref{uniqueinenergyclass}. The following Lemma in the spirit of \cite{riv} takes care of this by investigating the regularity for a fixed time $t \in \mathbb{R}_{+}$:
	
	\begin{lem}
	\label{reglemma}
		Let $f \in L^{2}(S^1)$ and assume that $u \in H^{1/2}(S^1)$ solves the following equation:
		\begin{equation}
		\label{eqforimprovedreg}
			(-\Delta)^{1/2} u = u | d_{1/2} u |^{2} + f.
		\end{equation}
		Then, we have the following improved regularity property:
		$$u \in H^{1}(S^1).$$
	\end{lem}
	
	The key point in the proof will be the appearance of an anti-symmetric potential $\Omega$ satisfying $\div_{1/2} \Omega = 0$ to which we can apply the non-local Wente-type inequality in Lemma \ref{fractionalwentelemma} or \eqref{fracwenteons1} in Appendix B. If we apply the result in Lemma \ref{reglemma} to $f = \partial_{t} u(t)$ and $u = u(t)$, we may deduce $u(t) \in H^{1}(S^1)$ for almost every $t$, given a sufficiently small bound on the $1/2$-energy at a given time $t$. Thus we may derive Theorem \ref{uniqueinenergyclass} by combining the statements in Proposition \ref{uniqueinenergyclassunderh1assumption} and Lemma \ref{reglemma}.
	
	In addition, let us observe that for given $f \in L^2(S^1)$, this seems to be an optimal result, as any solution $(-\Delta)^{1/2} u = f$ would satisfy $u \in H^{1}(S^1)$, but no higher regularity can be deduced in general.
	
	\begin{proof}
		As in \cite{mazoschi}, we know that there exists a map $\Omega \in L^{2}_{od}(S^1 \times S^1;\R^{n \times n})$ depending on $u$, such that $\Omega^{T} = - \Omega$ and $\div_{1/2} \Omega = 0$, such that we derive from \eqref{eqforimprovedreg}:
		\begin{equation}
		\label{eq01}
			(-\Delta)^{1/2} u = \Omega \cdot d_{1/2} u + T(u) + f,
		\end{equation}
		where $T(u)$ is as in \cite{mazoschi}. In fact, we have by using the components $u = (u^{1}, \ldots, u^{n})$ and Einstein's summation convention:
		\begin{align}
			u^{i}(x) d_{1/2} u^{k}(x,y) d_{1/2}u^{k}(x,y)	&= u^{i}(x) d_{1/2} u^{k}(x,y) d_{1/2}u^{k}(x,y) - u^{k}(x) d_{1/2} u^{i}(x,y) d_{1/2} u^{k}(x,y) \notag \\
											&+ u^{k}(x) d_{1/2} u^{i}(x,y) d_{1/2} u^{k}(x,y) \notag \\
											&=: \Omega_{ik}(x,y) d_{1/2} u^{k}(x,y) + u^{k}(x) d_{1/2} u^{i}(x,y) d_{1/2} u^{k}(x,y) \notag \\
											&= \Omega_{ik}(x,y) d_{1/2} u^{k}(x,y) + \frac{1}{2} d_{1/2} u^{i}(x,y) | d_{1/4} u^{k}(x,y) |^2 \notag \\
											&=: \Omega_{ik}(x,y) d_{1/2} u^{k}(x,y) + T^{i}(u)
		\end{align}
		Thus, the following formula for every $i = 1, \ldots n$ holds:
		$$T^{i}(u) := \sum_{k=1}^{n} \int_{S^1} d_{1/2} u^{i}(x,y) | d_{1/4} u^{k}(x,y) |^2 \frac{dy}{| x-y |}, \quad T(u) = (T^{1}(u), \ldots, T^{n}(u)),$$
		and moreover:
		$$\Omega_{ik}(x,y) := u^{i}(x) d_{1/2} u^{k}(x,y) - u^{k}(x) d_{1/2} u^{i}(x,y), \quad \forall i,k \in \{ 1, \ldots n \}$$
		We introduce the following notion $T(u,v,w) := (T^{1}(u,v,w), \ldots, T^{n}(u,v,w))$:
		\begin{equation}
			T^{i}(u,v,w) := \sum_{k=1}^{n} \int_{S^1} d_{1/2} u^{i}(x,y)  d_{1/4} v^{k}(x,y)  d_{1/4}w^{k}(x,y) \frac{dy}{| x-y |}, \quad \forall i \in \{ 1, \ldots, n \},
		\end{equation}
		and clearly $T(u,u,u) = T(u)$. We have the following estimates, refining the ones already found in \cite{mazoschi}:\\
		
		Assume that $p > 2$ as well as $u \in \dot{F}^{1/2}_{p,2}(S^1)$ and $v,w \in \dot{H}^{1/2}(S^1)$. Then we have by using H\"older's inequality
		\begin{align}
		\label{estimateforTpg2}
			\| T(u,v,w) \|_{L^{\frac{2p}{p+2}}(S^1)}	&\leq \left( \int_{S^1} \mathcal{D}_{1/2, 2}(u) \mathcal{D}_{1/4, 4}(v) \mathcal{D}_{1/4,4}(w) dx \right)^{\frac{p+2}{2p}} \notag \\
											&\lesssim \| u \|_{\dot{W}^{1/2,(p,2)}} \| v \|_{\dot{W}^{1/4,(4,4)}} \| w \|_{\dot{W}^{1/4, (4,4)}} \notag \\
											&\lesssim \| u \|_{\dot{F}^{1/2}_{p,2}} \| v \|_{\dot{F}^{1/4}_{4,4}} \| w \|_{\dot{F}^{1/4}_{4,4}} \notag \\
											&\lesssim \| u \|_{\dot{F}^{1/2}_{p,2}} \| v \|_{\dot{F}^{1/4}_{4,2}} \| w \|_{\dot{F}^{1/4}_{4,2}} \notag \\
											&\lesssim \| u \|_{\dot{F}^{1/2}_{p,2}} \| v \|_{\dot{F}^{1/2}_{2,2}} \| w \|_{\dot{F}^{1/2}_{2,2}} \notag \\
											&= \| u \|_{\dot{F}^{1/2}_{p,2}} \| v \|_{\dot{H}^{1/2}} \| w \|_{\dot{H}^{1/2}},
		\end{align}
		where we used the second part of Theorem \ref{schiwangthm1.4} for the circle $S^1$, see also Appendix B. Furthermore, standard embeddings for Triebel-Lizorkin spaces were used in the estimates above. One should notice that:
		$$4 > \frac{1 \cdot 4}{1+\frac{1}{4} 4} = 2,$$
		meaning that the second part of Theorem \ref{schiwangthm1.4} applies to $\cdot{F}^{1/4}_{4,4}(S^1)$. This also implies thanks to the Sobolev-type embedding:
		$$\dot{F}^{1/2}_{\frac{p}{p-1},2}(S^1) \hookrightarrow L^{\frac{2p}{p-2}}(S^1), \quad \forall p > 2,$$
		that we have an estimate of the following form by \eqref{estimateforTpg2} and using duality of Triebel-Lizorkin spaces:
		\begin{equation}
		\label{estimateforTpg2dual}
			\| T(u,v,w) \|_{\dot{F}^{-1/2}_{p,2}(S^1)} \lesssim \| u \|_{\dot{F}^{1/2}_{p,2}(S^1)} \| v \|_{\dot{H}^{1/2}(S^1)} \| w \|_{\dot{H}^{1/2}(S^1)}
		\end{equation}
		Moreover, if $u,v,w \in \dot{H}^{1/2}(S^1)$, we also know by first switching $x,y$ and then using H\"older's inequality and Sobolev-type embeddings:
		\begin{align}
			\int_{S^1} 		&\int_{S^1} \varphi^{i}(x) d_{1/2} u^{i}(x,y)  d_{1/4} v^{k}(x,y)  d_{1/4}w^{k}(x,y) \frac{dy dx}{| x-y |}	 \notag \\
						&= \int_{S^1} \int_{S^1} (\varphi^{i}(x) - \varphi^{i}(y)) d_{1/2} u^{i}(x,y)  d_{1/4} v^{k}(x,y)  d_{1/4}w^{k}(x,y) \frac{dy dx}{| x-y |} \notag \\
						&\lesssim \int_{S^1} \mathcal{D}_{1/2,2}(\varphi) \mathcal{D}_{1/6,6}(u) \mathcal{D}_{1/6,6}(v) \mathcal{D}_{1/6,6}(w) dx \notag \\
						&\lesssim \| \varphi \|_{\dot{W}^{1/2,(2,2)}} \| u \|_{\dot{W}^{1/6,(6,6)}} \| v \|_{\dot{W}^{1/6,(6,6)}} \| w \|_{\dot{W}^{1/6,(6,6)}} \notag \\
						&\lesssim \| \varphi \|_{\dot{F}^{1/2}_{2,2}} \| u \|_{\dot{F}^{1/6}_{6,6}} \| v \|_{\dot{F}^{1/6}_{6,6}} \| w \|_{\dot{F}^{1/6}_{6,6}} \notag \\
						&\lesssim \| \varphi \|_{\dot{H}^{1/2}} \| u \|_{\dot{H}^{1/2}} \| v \|_{\dot{H}^{1/2}} \| w \|_{\dot{H}^{1/2}},
		\end{align}
		again using Theorem \ref{schiwangthm1.4} on the circle as well as:
		$$6 > \frac{1 \cdot 6}{1 + \frac{1}{6} 6} = 3,$$
		justifying the application in this case. This immediately yields:
		\begin{equation}
		\label{estimateforTpe2}
			\| T(u,v,w) \|_{\dot{H}^{-1/2}(S^1)} \lesssim \| u \|_{\dot{H}^{1/2}(S^1)} \| v \|_{\dot{H}^{1/2}(S^1)} \| w \|_{\dot{H}^{1/2}(S^1)}
		\end{equation}
		Finally, if $w$ is smooth (and similarily for $v$ smooth), we find by similar arguments:
		\begin{equation}
		\label{estimateforTsmooth}
			\| T(u,v,w) \|_{L^{\frac{2p}{p+2}}} \lesssim \| \nabla w \|_{L^{\infty}} \| u \|_{\dot{H}^{1/2}} \| v \|_{\dot{F}^{1/2- 1/p}_{p,2}} \lesssim \| \nabla w \|_{L^{\infty}} \| u \|_{\dot{H}^{1/2}} \| v \|_{\dot{H}^{1/2}},
		\end{equation}
		for all $2 \leq p < + \infty$ and therefore, $T(u,v,w) \in L^{r}(S^1)$ for all $r \in [1,2[$, provided either $v$ or $w$ smooth.\\
		
		Let us return to \eqref{eq01}. This may now be rewritten as:
		\begin{equation}
		%\label{eq02}
			(-\Delta)^{1/2} u - {T}(u,u,u)  = \Omega \cdot d_{1/2} u + f,
		\end{equation}
		Letting $v = u - \dashint_{S^1} u = u - \hat{u}(0)$, we thus see:
		\begin{equation}
		\label{eq02}
			(-\Delta)^{1/2} v - T(v,u,u) = \Omega \cdot d_{1/2} v + f
		\end{equation}
		We notice that since each summand in \eqref{eq02} is integrable, we may include a summand with each, such that each summand has mean $0$. This allows us to apply $(-\Delta)^{-1/2}$ and renders this operator injective in an appropriate sense.
		
		The next step is to approximate the terms appearing in \eqref{eq02}, similar to \cite{riv2}. Therefore, let $\Omega_0$ be a smooth map from $S^1 \times S^1$ into the anti-symmetric $n \times n$-matrices, such that:
		$$\| \Omega_0 - \Omega \|_{L^2_{od}} < \varepsilon,$$
		for $\varepsilon > 0$ to be determined. This can for example be obtained by cutting $\Omega$ off in a sufficiently small neighbourhood of the diagonal and then using convolutions to smoothen the function and thus approximate $\Omega$ by regular functions. This way, $\Omega_0$ can also be assumed to be supported outside of the diagonal and thus to vanish in a neighbourhood of it. We may also assume that $\div_{1/2} \Omega_0 = 0$. This can be achieved by otherwise solving:
		$$(-\Delta)^{1/2} h = \div_{1/2} \Omega,$$
		in the weak sense and using $\Omega_0 - d_{1/2} h$ instead of $\Omega_{0}$. One might argue by noticing that $\div_{1/2} \Omega_0$ is smooth by the function vanishing in a neighbourhood of the diagonal and then solving for $h$ which immediately will be smooth as well. The right estimate can be obtained by the following train of thought:
		\begin{align}
			\| h \|_{\dot{H}^{1/2}}^{2}	&= \int_{S^1} (-\Delta)^{1/2} h \cdot h dz \notag \\
								&= \int_{S^1} \div_{1/2}( \Omega_0 - \Omega) h dz \notag \\
								&= \int_{S^1} \int_{S^1} (\Omega_0(z,w) - \Omega(z,w)) d_{1/2} h(z,w) \frac{dz dw}{| z-w |} \notag \\
								&\lesssim \| \Omega_0 - \Omega \|_{L^{2}_{od}} \| h \|_{\dot{H}^{1/2}},
		\end{align}
		providing an estimate for $d_{1/2} h$ that is required to ensure that $\Omega_0 - d_{1/2} h$ remains close to $\Omega$, while becoming divergence-free. In addition, we may choose a smooth function $\tilde{u}$ to be arbitrarily close to $u$ in $H^{1/2}(S^1)$, i.e. for any $\varepsilon > 0$ given, we can take $\tilde{u}$ in such a way that:
		$$\| u - \tilde{u} \|_{H^{1/2}(S^1)} < \varepsilon$$ 
		
		One proceeds now as in \cite{riv}: We may introduce the solution operator:
		\begin{align}
			\tau(v) := v + 					&(-\Delta)^{-1/2} \left( (\Omega_0 - \Omega) \cdot d_{1/2} v + T(v, \tilde{u} - u, \tilde{u} - u ) \right) \notag \\
										&= (-\Delta)^{-1/2} \left( (-\Delta)^{1/2} v + (\Omega_{0} - \Omega) \cdot d_{1/2} v + T(v, \tilde{u} - u, \tilde{u} - u )  \right) \notag \\
										&= (-\Delta)^{-1/2} \left( \Omega_0 \cdot d_{1/2} v + T(v, u, \tilde{u}) + T(v, \tilde{u}, u-\tilde{u})  + f \right)
		\end{align}
		We notice that the solution operator $\tau$ is well-defined, as we assumed that all summands have mean $0$, thus we could also apply it to each summand individually. To deduce the desired regularity result, we now show that $\tau$ defines a bijective operator from $\dot{F}^{1/2}_{p,2}$ to itself for each $p \geq 2$. As in \cite{riv}, let us split our considerations into two distinct cases:
		
		\paragraph{The "easy" Case: $p>2$}
		This case is an immediate consequence of the ellipticity of the fractional $1/2$-Laplacian and the analogue of the fractional Sobolev embeddings. Fixing $v$ to be the solution $u \in H^{1/2}(S^1)$ as in the Lemma on the RHS, we would like to solve:
		\begin{align}
		\label{tauinpge2}
			\tau(v) 	&= v + (-\Delta)^{-1/2} \left( (\Omega_0 - \Omega) \cdot d_{1/2} v + T(v, \tilde{u} - u, \tilde{u} - u ) \right) \notag \\
					&= (-\Delta)^{-1/2} \left( \Omega_0 \cdot d_{1/2} u + T(u, u, \tilde{u}) + T(u, \tilde{u}, u-\tilde{u}) + f \right),
		\end{align}
		and we may conclude that the RHS on the last line of \eqref{tauinpge2} lies in $\dot{F}^{1/2}_{p,2}$ thanks to Sobolev embeddings and the $L^{r}$-estimate for $1 \leq r < 2$ in \eqref{estimateforTsmooth}, smoothness of $\Omega_0, \tilde{u}$ and the properties of $(-\Delta)^{1/2}$. Observe that the smallness of $\Omega_0 - \Omega$ in $L^{2}_{od}$ and of $u - \tilde{u}$ in $H^{1/2}(S^1)$ is used in order to conclude that the solution operator is invertible by the usual perturbation argument. One invokes here the estimate proved above, see \eqref{estimateforTpg2dual}, \eqref{estimateforTpe2} and H\"older's inequality applied to $(\Omega - \Omega_0) \cdot d_{1/2} u$ together with:
		\begin{align}
		\label{estimateforomegaminomega0}
			\int_{S^1} (\Omega - \Omega_0) \cdot d_{1/2} u \varphi dx		&\leq \| (\Omega - \Omega_0) \cdot d_{1/2} u \|_{L^{\frac{2p}{p+2}}} \| \varphi \|_{L^{\frac{2p'}{2-p'}}} \notag \\
															&\lesssim \| \Omega - \Omega_0 \|_{L^{2}_{od}} \| | d_{1/2} u | \|_{L^{p}} \| \varphi \|_{\dot{F}^{1/2}_{p',2}} \notag \\
															&\lesssim \| \Omega - \Omega_0 \|_{L^{2}_{od}} \| u \|_{\dot{F}^{1/2}_{p,2}} \| \varphi \|_{\dot{F}^{1/2}_{p',2}},
		\end{align}
		where $1/p + 1/p' = 1$ and we used the Sobolev embeddings for Triebel-Lizorkin spaces, hence showing boundedness in $\dot{F}^{-1/2}_{p,2}$ of $(\Omega - \Omega_0) \cdot d_{1/2} u$. Similar estimates using H\"older's inequality as in \eqref{estimateforTpg2} show that:
		\begin{equation}
		\label{estimateforTinvertible}
			\| T(v, \tilde{u} - u, \tilde{u} - u ) \|_{\dot{F}^{-1/2}_{p,2}} \lesssim \| v \|_{\dot{F}^{1/2}_{p,2}} \| u - \tilde{u} \|_{\dot{H}^{1/2}}^2
		\end{equation}
		Using that $(-\Delta)^{1/2}$ defines an isomorphism between $\dot{F}^{1/2}_{p,2}$ and $\dot{F}^{-1/2}_{p,2}$, see \cite{schmeitrieb}, the required estimate follows from \eqref{estimateforomegaminomega0} and \eqref{estimateforTinvertible}:
		\begin{align}
			\Big{\|} (-\Delta)^{-1/2} \big{(} (\Omega_0 - \Omega) \cdot d_{1/2} v + &T(v, \tilde{u} - u, \tilde{u} - u ) \big{)} \Big{\|}_{\dot{F}^{1/2}_{p,2}} \notag \\
								&\lesssim \left( \| \Omega - \Omega_0 \|_{L^{2}_{od}} + \| u - \tilde{u} \|_{\dot{H}^{1/2}}^2 \right) \| v \|_{\dot{F}^{1/2}_{p,2}}
		\end{align}
		Hence, the perturbation is small, if $\Omega_0, \tilde{u}$ are sufficiently good approximations and thus $\tau$ invertible in this case. Notice that the RHS of \eqref{tauinpge2} lies in $L^{q}$ for all $q < 2$. Thus, using the same arguments as in \eqref{estimateforomegaminomega0} and also \eqref{estimateforTpg2}, we could deduce that $\Omega_0 \cdot d_{1/2} u + T(u, u, \tilde{u}) + T(u, \tilde{u}, u-\tilde{u}) + f \in \dot{F}^{-1/2}_{p,2}$, and therefore the RHS of \eqref{tauinpge2} is in $\dot{F}^{1/2}_{p,2}$.
		
		We emphasise that this step is crucially relying on the homogeneous Triebel-Lizorkin spaces $\dot{F}^{1/2}_{p,2}(S^1)$ for $p > 2$ which prove to be the right one and possess an equivalent norm description in terms of the $L^p$-norm of $| d_{1/2} f |$ for all $f$ inside this space, see Theorem \ref{schiwangthm1.4} as well as Appendix B. See also \cite{schmeitrieb}, \cite{triebel}, \cite{schiwang}, \cite{stein} and the references therein for further details.
		
		\paragraph{The "hard" Case: $p=2$}
		On the other hand, this case is more delicate and requires a version of the Wente-type result in \cite{mazoschi}, see Lemma \ref{fractionalwentelemma}, for the circle. This can be obtained by a set of changes of variables under the stereographic projection and using a partition of unity, see \eqref{fracwenteons1} and the computations preceeding it. We postpone the details of this computation to Appendix B, as it is basically a technical analysis of a sequence of changes of variables. This immediately allows us to again proceed as in \cite{riv}, as we may now estimate the perturbation of the solution operator by the Wente-type estimate below:
		$$\| (\Omega_0 - \Omega) \cdot d_{1/2} v - c \|_{\dot{H}^{-1/2}} \leq C \| \Omega_0 - \Omega \|_{L^{2}_{od}} \| v \|_{\dot{H}^{1/2}} \leq C \varepsilon \| v \|_{\dot{H}^{1/2}},$$
		where we take $c$ as the constant rendering the term to have mean $0$. To conclude the proof of Lemma \ref{reglemma}, we now observe that for $\varepsilon > 0$ sufficiently small, the perturbation:
		$$(-\Delta)^{-1/2} \left( (\Omega_0 - \Omega) \cdot d_{1/2} v \right),$$
		is small as well (when considered as an operator from $\dot{H}^{1/2}(S^1)$ to itself and hence the operator:
		$$v \mapsto v + (-\Delta)^{-1/2} \left( (\Omega_0 - \Omega) \cdot d_{1/2} v \right),$$
		becomes an isomorphism, now for $\dot{H}^{1/2}(S^1)$ to itself as well as for different integrability exponents $p > 2$. If $u$ is sufficiently small in $\dot{H}^{1/2}(S^1)$, the same remains true if we include the missing term under the smallness assumption on $\| u - \tilde{u} \|_{\dot{H}^{1/2}}$:
		$$v \mapsto v + (-\Delta)^{-1/2} \left( (\Omega_0 - \Omega) \cdot d_{1/2} v \right) + (-\Delta)^{-1/2} \left( T(v, \tilde{u} - u, \tilde{u} - u ) \right)$$
		This is clear, as we may use \eqref{estimateforTpe2} to deduce that the second summand is again a small perturbation which does not affect invertibility.\\
		
		As in \cite{riv}, we may now deduce thanks to existence and uniqueness that $u \in \dot{F}^{1/2}_{4,2}(S^1)$ and, consequently, $| d_{1/2} u | \in L^4$, using the embedding:
		\begin{equation}
		\label{triebelembeddingforp4}
			\dot{F}^{1/2}_{4,2}(S^1) \subset \dot{F}^{1/2}_{2,2}(S^1) = H^{1/2}(S^1),
		\end{equation}
		which shows that the unique solution $v$ of \eqref{tauinpge2} in $\dot{F}^{1/2}_{4,2}(S^1)$ (or, in fact, in any $\dot{F}^{1/2}_{p,2}(S^1)$ with $p \geq 2$ by replacing $4$ by $p$ in the argument) actually agrees with $u - \hat{u}(0)$, where $u$ is as given in Lemma \ref{reglemma}, due to the uniqueness in the case $p=2$ and the embedding \eqref{triebelembeddingforp4} of Triebel-Lizorkin spaces. So $(-\Delta)^{1/2} u \in L^{2}(S^1)$ by directly using \eqref{eqforimprovedreg}, which immediately yields $\nabla u = \mathcal{R} (-\Delta)^{1/2} u \in L^2(S^1)$ and so $u \in H^{1}(S^1)$. Hence, we have established the desired regularity result for $u$. This concludes our proof.
	\end{proof}

	\subsection{Regularity}
	
	Next, we show that solutions to the fractional gradient flow \eqref{gradflowuv} are smooth for all times $t > 0$. The main idea is to study the regularity of the RHS of \eqref{gradflowuv} and bootstrap this information. In fact, a key step lies in studying the Fourier series of
	$$| d_{1/2} u |^{2}(x),$$
	and establishing sufficient $H^{s}$-estimates to bootstrap the regularity.
	
	\subsubsection{Some useful Results}
	
	Let us assume that $u, v$ are trigonometric polynomials. Thus, they are of the form:
	$$u(x) = \sum_{n \in \mathbb{Z}} \hat{u}(n) e^{inx}, \quad v(x) = \sum_{n \in \mathbb{Z}} \hat{v}(n) e^{inx},$$
	where $\hat{u}(n), \hat{v}(n) = 0$ for all but finitely many $n \in \mathbb{N}$. Let us consider $d_{1/2} u \cdot d_{1/2} v(x)$, or more precisely its Fourier coefficients:
	\begin{align}
		\widehat{d_{1/2} u \cdot d_{1/2} v}(n)	&= \frac{1}{2\pi} \langle d_{1/2} u \cdot d_{1/2} v, e^{-inx} \rangle \notag \\
									&= \frac{1}{2\pi} \int_{-\pi}^{\pi} d_{1/2} u \cdot d_{1/2} v(x) e^{-inx} dx \notag \\
									&= \frac{1}{2\pi} \int_{-\pi}^{\pi} \int_{S^1} \sum_{j \in \Z} \sum_{k \in \Z} \frac{e^{ijh} - 1}{| h |} \frac{e^{ikh} - 1}{| h |} \hat{u}(j) \hat{v}(k) e^{i(j+k)x} dh e^{-inx} dx \notag \\
									&= \sum_{j \in \Z} \int_{S^1} \frac{e^{ijh} - 1}{| h |} \frac{e^{i(n-j)h} - 1}{| h |} dh \cdot \hat{u}(j) \hat{v}(n-j)
	\end{align}
	where $n \in \mathbb{Z}$ is arbitrary and we used the formulas for $u,v$ as trigonometric polynomials. Let us introduce:
	\begin{equation}
	\label{exponentialintegral}
		C(j,k) := \int_{S^1} \frac{e^{ijh} - 1}{| h |} \frac{e^{ikh} - 1}{| h |} dh
	\end{equation}
	Therefore, we may see using the previous computations:
	$$\widehat{d_{1/2} u \cdot d_{1/2} v}(n) = \sum_{j \in \Z} C(j,n-j) \hat{u}(j) \hat{v}(n-j), \quad \forall n \in \Z$$
	A first step to deduce the regularity of $d_{1/2} u \cdot d_{1/2} v$ lies in the study of $C(j,k)$. Namely, we observe that:
	$$C(j,j) = \int_{S^1} \frac{|e^{ijh} - 1 |^2}{|h|^2} dh = \int_{S^1} \frac{\sin \left( \frac{jh}{2} \right)^2}{\sin \left( \frac{h}{2} \right)^2} dh = | j | \int_{S^1} F_{| j |}(h) dh = | j |,$$
	where $F_{n}$ denotes the $n$-th F\'ejer kernel. Using Cauchy-Schwarz and F\'ejer kernels, we easily deduce:
	\begin{align}
		| C(j,k) | 	&\leq \left( \int_{S^1} \frac{| e^{ijh} - 1 |^2}{| h |^2} dh \right)^{1/2} \left( \int_{S^1} \frac{| e^{ikh} - 1 |^2}{| h |^2} dh \right)^{1/2} \notag \\
				&\leq \sqrt{| j |} \sqrt{| k |} \cdot \left( \int_{S^1} F_{|j|}(h) dh \right)^{1/2} \left( \int_{S^1} F_{|j|}(h) dh \right)^{1/2} \notag \\
				&= \sqrt{| j |} \sqrt{| k |},
	\end{align}
	for all $j,k \in \Z$. The main goal is now to deduce regularity estimates leading to conclusions like $d_{1/2} u \cdot d_{1/2} v \in H^{s}(S^1)$ for some $s \in \R$ together with appropriate estimates in terms of $u,v$. Namely, we will prove the following:
	
	\begin{lem}
	\label{regularityestimatelemmaforhalfgradient}
		Let $u,v$ be trigonometric polynomials as above. Then we have for all $\varepsilon > 0$:
		\begin{align}
		\label{estimatelemmaregbootstrap01}
			\| d_{1/2} u \cdot d_{1/2} v \|_{\dot{H}^{s}(S^1)} 	&\lesssim \| (-\Delta)^{1/4 + s/2 + \varepsilon} u \|_{L^{2}} \| (-\Delta)^{1/2} v \|_{L^{2}} + \| (-\Delta)^{1/2} u \|_{L^{2}} \| (-\Delta)^{1/4 + s/2 + \varepsilon} v \|_{L^{2}} \notag \\
												&\lesssim \| u \|_{\dot{F}^{1/2 + s + 2\varepsilon}_{2,2}} \| v \|_{\dot{H}^1} + \| u \|_{\dot{H}^{1}} \| v \|_{\dot{F}^{1/2 + s + 2\varepsilon}_{2,2}},
		\end{align}
		as well as:
		\begin{align}
		\label{estimatelemmaregbootstrap02}
			\| d_{1/2} u \cdot d_{1/2} v \|_{\dot{H}^{s}(S^1)} 	&\lesssim \| (-\Delta)^{1/4 + s/2 + 2\varepsilon} u \|_{L^{2}} \| (-\Delta)^{1/2 - \varepsilon} v \|_{L^{2}} + \| (-\Delta)^{1/2 - \varepsilon} u \|_{L^{2}} \| (-\Delta)^{1/4 + s/2 + 2\varepsilon} v \|_{L^{2}} \notag \\
												&\lesssim \| u \|_{\dot{F}^{1/2 + s + 4\varepsilon}_{2,2}} \| v \|_{\dot{F}^{1-2\varepsilon}_{2,2}} + \| u \|_{\dot{F}^{1-2\varepsilon}_{2,2}} \| v \|_{\dot{F}^{1/2 + s + 2\varepsilon}_{2,2}},
		\end{align}
		and by density, the same estimates continue to hold true for all $u,v$ in the corresponding spaces. The constants depend on $s > 0$ and $\varepsilon > 0$.
	\end{lem}
	
	\begin{proof}
		By definition, we have:
		\begin{align}
		\label{estimateregularityoffracgrad}
			\| d_{1/2} u \cdot d_{1/2} v \|_{\dot{H}^s}^{2}	&= \sum_{n \in \Z} | n |^{2s} | \widehat{d_{1/2} u \cdot d_{1/2} v}(n) |^2 \notag \\
												&\lesssim \sum_{n \in \Z} | n |^{2s} \left( \sum_{j \in \Z} | \hat{u}(j) | | \hat{v}(n-j) | \sqrt{| j | | n-j |} \right)^{2} \notag \\
												&\lesssim \sum_{n \in \Z} \left( \sum_{j \in \Z} | \widehat{(-\Delta)^{1/4} u}(j) | | \widehat{(-\Delta)^{1/4} v}(n-j) \left( | j |^{s} + | n-j |^{s} \right) \right)^2 \notag \\
												&\lesssim \sum_{n \in \Z} \left( \sum_{j \in \Z} | \widehat{(-\Delta)^{1/4 + s/2} u}(j) | | \widehat{(-\Delta)^{1/4} v}(n-j) | \right)^2 \notag \\
												&+ \sum_{n \in \Z} \left( \sum_{j \in \Z} | \widehat{(-\Delta)^{1/4 + s/2} v}(n-j) | | \widehat{(-\Delta)^{1/4} u}(j) | \right)^2
		\end{align}
		By symmetry, it suffices to restrict our attention to the first summand in \eqref{estimateregularityoffracgrad}. We observe:
		\begin{equation*}
			| \widehat{(-\Delta)^{1/4 + s/2} u}(j) | | \widehat{(-\Delta)^{1/4} v}(n-j) |	= | \widehat{(-\Delta)^{1/4 + s/2 + \varepsilon} u}(j) | | \widehat{(-\Delta)^{1/2} v}(n-j) | | n-j |^{-1/2} | j |^{-2 \varepsilon},
		\end{equation*}
		which can be used to deduce by using Cauchy-Schwarz and Young's inequality:
		\begin{align}
		\label{secondestimatefordfracgradregestimates}
			\| d_{1/2} u \cdot d_{1/2} v \|_{\dot{H}^s}^{2}	&\lesssim \sum_{n \in \Z} \left( \sum_{j \in \Z} | \widehat{(-\Delta)^{1/4 + s/2 + \varepsilon} u}(j) |^2 | \widehat{(-\Delta)^{1/2} v}(n-j) |^2 \right) \left( \sum_{j \in \Z} \frac{1}{| n-j | | j |^{2 \varepsilon}} \right) \notag \\
												&\lesssim \| u \|_{\dot{F}^{1/2 + s + 2 \varepsilon}_{2,2}}^2 \| v \|_{\dot{H}^1}^2,
		\end{align}
		and completely analogous for the second summand in \eqref{estimateregularityoffracgrad}. Young's inequality is used to bound:
		$$\sum_{j \in \Z} \frac{1}{| n-j | | j |^{2 \varepsilon}} \lesssim \sum_{j \in \Z} \left( \frac{1}{| n-j |^{p'}} + \frac{1}{| j |^{2\varepsilon p}} \right) < +\infty,$$
		and choosing $p \in ]1,+\infty[$ in such a way that $2 \varepsilon p > 1$ and $p'$ being the H\"older dual of $p$. The second estimate \eqref{secondestimatefordfracgradregestimates} follows analogously. This concludes therefore the proof of regularity.
	\end{proof}
	
	Another useful result will be the following:
	
	\begin{lem}
	\label{boundonfractionallaplacianunderholder}
		Assume that $\alpha \in ]0,1[$ and $u \in C^{0,\alpha}(S^1)$. Then:
		$$(-\Delta)^{s} u \in L^{\infty}(S^1),$$
		if we know:
		$$1 > \alpha > 2 s$$
	\end{lem}
	
	\begin{proof}
		Up to constants, we know:
		\begin{align}
			| (-\Delta)^{s} u(x) |	&\lesssim \int_{S^1} \frac{| u(x) - u(y) |}{| x-y |^{1+2s}} dy \notag \\
							&\lesssim \int_{S^1} \frac{1}{| x-y |^{1 + 2s - \alpha}} dy \cdot \| u \|_{C^{0, \alpha}},
		\end{align}
		which yields the desired boundedness if:
		$$1 + 2s - \alpha < 1 \Rightarrow 2s < \alpha,$$
		by exploiting the integrability of $| x-y |^{1 + 2s - \alpha}$. This concludes our proof.
	\end{proof}
	
	For convenience's sake, let us also state the version of Theorem 3.1 in \cite{hieber} that will be relevant in our discussion of the regularity of solutions to the fractional heat equation:
	
	\begin{lem}[Theorem 3.1 in \cite{hieber}]
	\label{theorem3.1hieberpruss}
		Let $1 < p < +\infty$ and $I = [0, T]$ be any interval with $T < +\infty$. Then there exists for each $f \in L^{p}(I \times S^1)$ a unique solution $u \in W^{1,p}(I \times S^1)$ of the equation:
		$$u_{t} + (-\Delta)^{1/2} u = f,$$
		and satisfying $u(0,\cdot) = 0$. Moreover, we have:
		\begin{equation}
			\| u \|_{W^{1,p}} \lesssim \| f \|_{L^{p}}.
		\end{equation}
	\end{lem}
	
	The result follows from Theorem 3.1 in \cite{hieber} by observing that the $1/2$-Laplacian is actually generating an analytic $C^0$-semigroup with the required properties (see for example \cite[3.2.E)]{hieber}).
	
	\subsubsection{Local Regularity}
	
	A key step in the study of regularity lies in the local regularity. Precisely, we will prove:
	
	\begin{prop}
	\label{localregprop}
		Let $u_0 \in C^{\infty}(S^1;S^{n-1})$ be any smooth map. Then there exists $T > 0$, possibly depending on $u_0$, and a smooth map $u \in C^{\infty}([0,T] \times S^1)$ which solves the half-harmonic gradient flow:
		\begin{equation}
			u_{t} + (-\Delta)^{1/2} u = u | d_{1/2} u |^2,
		\end{equation}
		and satisfies the initial condition $u(0,x) = u_{0}(x)$. Moreover, it holds for all $x \in S^1$ and $0 \leq t < T$:
		\begin{equation}
			u(t,x) \in S^{n-1},
		\end{equation}
		i.e. the solution $u$ indeed assumes values in the desired target manifold.
	\end{prop}
	
	A key observation is therefore, due to the previously proved uniqueness of the solution by Theorem \ref{uniqueness}, that any solution of the equation \eqref{gradflowuv} is indeed regular at least for sufficiently small times $t$ and provided the boundary data is smooth. If the $1/2$-energy at $t=0$ is small, the same holds for energy class solutions.\\
	
	\textit{Proof of Proposition \ref{localregprop}.} We shall follow the presentation in \cite{hamilton} and adapt the techniques to the non-local framework encountered here. Therefore, we want to study the following map for every $p > 2$:
	\begin{equation}
	\label{fracharmflowoperator}
		H: W^{1,p}([0,T] \times S^1) \to L^{p}([0,T] \times S^1), \quad H(u) := u_{t} + (-\Delta)^{1/2} u - u | d_{1/2} u |^2
	\end{equation}
	We want to prove that we may apply the local Inversion Theorem for Banach spaces to $H$ for sufficiently regular functions. This will then enable us to deduce the result in Proposition \ref{localregprop} by a slight modification, completely analogous to \cite[p.122-124]{hamilton}.
	
	Observe that as $p > 2$, any $u \in W^{1,p}([0,T] \times S^1)$ will be continuous and bounded. Therefore, by using Sobolev-embeddings, we immediately deduce that the map is well-defined. In fact, the only critical part is dealt with by the following computation:
	\begin{align}
	\label{computationsfornonlinearity}
		\| u | d_{1/2} u |^2 \|_{L^{p}}^{p}	&= \int_{0}^{T} \int_{S^1} | u |^{p} | d_{1/2} u |^{2p} dx dt \notag \\
								&\lesssim \| u \|_{L^{\infty}}^{p} \int_{0}^{T} \int_{S^1} | d_{1/2} u |^{2p} dx dt \notag \\
								&\lesssim \| u \|_{L^{\infty}}^{p} \| u \|_{F^{1/2}_{2p,2}([0,T] \times S^1)}^{2p} \notag \\
								&\lesssim \| u \|_{W^{1,p}([0,T] \times S^1)}^{3p},
	\end{align}
	where we used the Triebel-Lizorkin- as well as Morrey-embeddings $W^{1,p}([0,T] \times S^1) \hookrightarrow \dot{F}^{1/2}_{2p,2}([0,T] \times S^1)$ and $W^{1,p}([0,T] \times S^1) \hookrightarrow C^{0}([0,T] \times S^1) \subset L^{\infty}([0,T] \times S^1)$, see \cite[Theorem 3.3.1]{triebel1} or \cite[Theorem 3.5.5]{schmeitrieb}. Furthermore, the map $H$ is actually differentiable. Namely, we observe by computing directional derivatives with respect to $h \in W^{1,p}([0,T] \times S^1)$:
	\begin{equation}
	\label{differentialofoperator}
		DH(u) h = h_{t} + (-\Delta)^{1/2} h - h | d_{1/2} u |^2 - 2 u d_{1/2} u \cdot d_{1/2} h,
	\end{equation}
	and by observing:
	$$H(u+h) - H(u) - DH(u) h = u | d_{1/2} h |^2 + 2 h d_{1/2} u \cdot d_{1/2} h + h | d_{1/2} h |^2,$$
	one immediately sees, using similar estimates as above in \eqref{computationsfornonlinearity}, that $H$ is actually a $C^1$-function. 
	
	In order to apply the local Inversion theorem, we would like to study the behaviour of the differential, in particular whether it is an isomorphism of Banach spaces. Assume for the moment that $u$ is actually in $C^{0,\alpha}([0,T] \times S^1)$ for some $\alpha > 1/2$. %We would like to study \eqref{differentialofoperator} and, in particular, its kernel. 
	Firstly, we observe that the map:
	$$h \mapsto h | d_{1/2} u |^2 + 2 u d_{1/2} u \cdot d_{1/2} h,$$
	is a compact map from $h \in W^{1,p}([0,T] \times S^1)$ to $L^{p}([0,T] \times S^1)$. This is immediate due to the H\"older continuity of $u$ which implies boundedness of $| d_{1/2} u |$ and compactness of Sobolev embeddings. Therefore, because $h \mapsto h_{t} + (-\Delta)^{1/2} h$ is invertible on the set $\tilde{W}^{1,p}_{0}([0,T] \times S^1)$ containing precisely all $h$ with $h(0, \cdot) = 0$ (see \cite[Theorem 3.1]{hieber} which asserts existence and uniqueness), \eqref{differentialofoperator} defines an invertible linear operator $DH(u): \tilde{W}^{1,p}_{0}([0,T] \times S^1) \to L^{p}([0,T] \times S^1)$ if and only if the kernel of $DH(u)$ is trivial. This is clear, as the operator is Fredholm with index $0$, since it is a sum of an invertible (and thus Fredholm operator of index $0$) and a compact operator. Therefore, we merely have to study the kernel of $DH(u)$. The result we will be proving is the following:
	
	\begin{lem}
	\label{lemmakernelbyfredholm}
		Assume that $u$ is smooth. Then $DH(u)$ has trivial kernel and $DH(u)$ defines an invertible operator.
	\end{lem}
	
	To initiate the study of the kernel of \eqref{differentialofoperator} among all $h$ with vanishing initial datum, we first prove regularity of $h$ in the kernel of $DH(u)$. This will then allow us to employ maximum principles for fractional PDE similar to \cite{hamilton}:
	
	\begin{lem}
	\label{regularityinkernelh}
		Let $u$ be as in Lemma \ref{lemmakernelbyfredholm}. Then, if $h \in \tilde{W}^{1,p}_{0}([0,T] \times S^1)$ lies in the kernel of $DH(u)$, then $h$ is smooth on $[0,T] \times S^1$.
	\end{lem}
	
	\begin{proof}[Proof of Lemma \ref{regularityinkernelh}]
	We first observe that given $h \in \tilde{W}^{1,p}_{0}([0,T] \times S^1)$, we have:
	$$h | d_{1/2} u |^2 + 2 u d_{1/2} u \cdot d_{1/2} h \in L^{\frac{4p}{4-p}},$$
	if $p < 4$ and in any $L^{q}$ with $q < \infty$, if $p \geq 4$. This follows again by Sobolev embeddings and using $u$ smooth (a similar inclusion with $2p/(4-p)$ holds, if $u \in W^{1,p}$ only and a similar iteration applies in this case as well). Thus:
	$$h_{t} + (-\Delta)^{1/2} h \in L^{q},$$
	for $q = 4p/(4 - p)$ or $1 < q < \infty$, depending on $p$. Since:
	\begin{equation}
	\label{connectiontolaplacein2d}
		( \partial_{t} - (-\Delta)^{1/2} ) ( \partial_{t} + (-\Delta)^{1/2} ) = \Delta_{t,x},
	\end{equation}
	i.e. the composition equals the Laplacian in $2$D, we may invoke classical elliptic regularity theory to find:
	$$\Delta_{t,x} h \in W^{-1,q} \Rightarrow h \in W^{1,q}$$
	Observe that in case $p \geq 4$, this shows that $u \in W^{1,q}$ for all $1 < q < \infty$. If $p < 4$, then:
	$$1/q = 1/p - 1/4 < 1/2 - 1/4 = 1/4 \Rightarrow q > 4,$$
	so we may iterate the same argument to find ourselves in the case $p > 4$. In any case, we have that for $h$ with:
	$$DH(u) h = 0,$$
	and vanishing initial datum that $h \in W^{1,q}$, for all $1 < q < \infty$. In particular, we have that for such $h$, the inclusion $h \in C^{0,\beta}$ holds for all $0 < \beta < 1$. This is immediate by Morrey's embedding. Observe that, for instance, this also means that $| d_{1/2} h |$ is bounded.\\
	
	For the remainder of the argument, let us restrict our attention to $u$ being smooth.  By Lemma \ref{boundonfractionallaplacianunderholder}, we have:
	\begin{equation}
		(-\Delta)^{1/4 + t} h \in L^{\infty}([0,T] \times S^1), \quad \forall t \in [0, 1/4[
	\end{equation}
	Thus, combining this consideration with Lemma \ref{regularityestimatelemmaforhalfgradient}, we see that for all $0 < s < 1/2$ and $\varepsilon > 0$ sufficiently small:
	\begin{align}
		\| d_{1/2} u \cdot d_{1/2} h \|_{\dot{H}^{s}(S^1)} 	&\lesssim \| (-\Delta)^{1/4 + s/2 + \varepsilon} h \|_{L^{2}} \| (-\Delta)^{1/2} u \|_{L^{2}} + \| (-\Delta)^{1/4 + s/2 + \varepsilon} u \|_{L^{2}} \| (-\Delta)^{1/2} h \|_{L^{2}} \notag \\
											&\lesssim \| h \|_{C^{0,\alpha}} \| u \|_{H^{1}} + \| u \|_{C^{0,\alpha}} \| h \|_{H^{1}},
	\end{align}
	and therefore, we know that:
	$$(-\Delta)^{s/2} \left( d_{1/2} u \cdot d_{1/2} h \right) \in L^{2}([0,T] \times S^1)$$
	The same then holds for $u d_{1/2} u \cdot d_{1/2} h$ as well as $h | d_{1/2} u |^2$ and thus leads to, by using elliptic regularity and \cite[Theorem 3.1]{hieber}:
	$$(-\Delta)^{s/2} h \in H^{1}( [0,T] \times S^1),$$
	using \eqref{connectiontolaplacein2d}. Bootstrapping using Lemma \ref{regularityestimatelemmaforhalfgradient}, we may deduce the same for every $s > 0$. For example, using \eqref{estimatelemmaregbootstrap02} with $s=3/4$ and $\varepsilon > 0$ sufficiently small, we find for almost every given time $t$, noting that $H^{s} \cap L^{\infty}$ is a Banach algebra for any $s > 0$:
	\begin{align}
	\label{bootstrapestimatetype}
		&\| 2 u(t) d_{1/2} u(t) \cdot d_{1/2} h(t) + h(t) | d_{1/2} u(t) |^2 \|_{\dot{H}^{3/4}(S^1)} \notag \\
			&\lesssim \| u(t) \|_{L^{\infty}} \left( \| h(t) \|_{\dot{H}^{5/4 + 4 \varepsilon}} \| u(t) \|_{\dot{H}^{1-2\varepsilon}} + \| u(t) \|_{\dot{H}^{5/4 + 4 \varepsilon}} \| h(t) \|_{\dot{H}^{1-2\varepsilon}} \right) + \| u(t) \|_{\dot{H}^{3/4}} \| u \|_{C^{0,\beta}} \| h \|_{C^{0,\beta}} \notag \\
			&+ \| h(t) \|_{L^{\infty}} \| u(t) \|_{\dot{H}^{5/4 + 4 \varepsilon}} \| u(t) \|_{\dot{H}^{1-2\varepsilon}} + \| h \|_{\dot{H}^{3/4}} \| u \|_{C^{0,\beta}}^2,
	\end{align}
	where $\beta > 1/2$ and the $\dot{H}^{1- 2 \varepsilon}$- and $L^{\infty}$-norms can be uniformly bounded using H\"older continuity and Lemma \ref{boundonfractionallaplacianunderholder}. Therefore, we find by the previous step and integrating with respect to $t$:
	$$(-\Delta)^{3/8} \left( 2 u d_{1/2} u \cdot d_{1/2} h + h | d_{1/2} u |^2 \right) \in L^{2}([0,T] \times S^1),$$
	and so $D (-\Delta)^{3/8} h \in L^{2} ([0,T] \times S^1)$, thus $(-\Delta)^{3/8} h \in H^{1}([0,T] \times S^1)$ similarily as before. This now enables us to apply the second part of Lemma \ref{regularityestimatelemmaforhalfgradient} with $s = 3/2$ and $\varepsilon > 0$ sufficiently small to deduce, similar as in \eqref{bootstrapestimatetype}:
	$$(-\Delta)^{3/4} \left( 2 u d_{1/2} u \cdot d_{1/2} h + h | d_{1/2} u |^2 \right) \in L^{2}([0,T] \times S^1),$$
	and thus using \eqref{connectiontolaplacein2d} and \cite[Theorem 3.1]{hieber} to find $(-\Delta)^{3/4} h \in H^{1}([0,T] \times S^1)$. This may now be iterated arbitrarily for an increasing sequence of $s$. Moreover, by inserting these expressions into the main equation $DH(u) h = 0$, we may deduce the same for higher derivatives in time direction, leading to:
	$$h \in \bigcap_{s \in \mathbb{N}} H^{s}([0,T] \times S^1)$$
	This shows that $h \in C^{\infty}([0,T] \times S^1)$ by Morrey-embeddings. 
	
	It should be noted, that due to using the $2$D-Laplacian, we merely get regularity to times $t < T$, since we do not prescribe the boundary data at $t = T$. If we want regularity for all $t \leq T$, we have to use the result in \cite[Theorem 3.1]{hieber} regarding analytic operator semigroups and maximal $L^{p}$-regularity of heat flows (notice that the $1/2$-Laplacian generates an analytic operator semigroup), which actually guarantee existence, uniqueness and estimates up to $t = T$. By uniqueness and the regularity for $t < T$, which we may deduce by using elliptic regularity, we may extend the estimates to $t = T$ for the solution $h$ by \cite[Theorem 3.1]{hieber}. So the result is true as stated, but requires slightly more technical arguments at the endpoint. We emphasise that the treatment of $t < T$ is necessary, as the uniqueness result in \cite{hieber} requires some regularity to hold while $(-\Delta)^{s} h$ has a-priori not sufficient regularity for \cite[Theorem 3.1]{hieber} to be applied.
	\end{proof}
	
	\begin{proof}[Proof of Lemma \ref{lemmakernelbyfredholm}]
	The smoothness of $h$ satisfying $DH(u) h = 0$ with vanishing initial datum now enables us to prove that:
	$$h = 0$$
	Namely, let us compute the following:
	\begin{align}
		\partial_{t} \left( | h |^2 \right)	&= 2 h_{t} h \notag \\
								&= 2 \left( -(-\Delta)^{1/2} h + 2 u d_{1/2} u \cdot d_{1/2} h + h | d_{1/2} u |^2\right) h \notag \\
								&= -2 h (-\Delta)^{1/2} h + 2 u h d_{1/2} u \cdot d_{1/2} h + | h |^2 | d_{1/2} u |^2,
	\end{align}
	and observe that there exists a $C > 0$ (since $u$ is smooth and thereforeH\"older continuous), such that:
	$$| h |^2 | d_{1/2} u |^2 \leq C | h |^2$$
	Moreover, we may easily find:
	\begin{align}
	\label{computationofhalflaplacianonsomespecialfunction}
		h (-\Delta)^{1/2} h	&= P.V. \int_{S^1} \frac{h(x) - h(y)}{| x-y |^2} dy h(x) \notag \\
						&= \frac{1}{2} P.V. \int_{S^1} \frac{| h(x) |^{2} - | h(y) |^2}{| x-y |^2} dy + \frac{1}{2} P.V. \int_{S^1} \frac{| h(x) - h(y) |^2}{| x-y |^2} dy \notag \\
						&= \frac{1}{2} (-\Delta)^{1/2} \left( | h |^2 \right) + \frac{1}{2} | d_{1/2} h |^2,
	\end{align}
	where we used:
	$$h(x) = \frac{h(x) + h(y)}{2} + \frac{h(x) - h(y)}{2}$$
	as well as:
	$$(h(x) + h(y)) (h(x) - h(y)) = | h(x) |^2 - | h(y) |^2.$$
	Therefore, we may estimate:
	\begin{align}
		\partial_{t} \left( | h |^2 \right) + (-\Delta)^{1/2} \left( | h |^2 \right)	&\leq -| d_{1/2} h |^2 + 2 u h d_{1/2} u \cdot d_{1/2} h + C | h |^2 \notag \\
														&\leq -| d_{1/2} h |^2 + | u |^2 | h |^2 | d_{1/2} u |^2 + | d_{1/2} h | + C | h |^2 \notag \\
														&\leq \hat{C} | h |^2
	\end{align}
	using the arithmetic-geometric mean to absorb $| d_{1/2} h |^2$ as well as the regularity of $u$. Here, $\hat{C} > 0$ is a constant not depending on $h$. Following the arguments in \cite[p.101]{hamilton} for the maximum principle, we may here deduce:
	$$h = 0,$$
	due to the initial values vanishing. We emphasise that the argument merely relies on the fact that $(-\Delta)^{1/2} h(x) \geq 0$ at a global maximum and $h_{t} \geq 0$ and considering $e^{-(\hat{C} +1)t} h(t,x)$ instead of $h(t,x)$.
	\end{proof}
	
	\begin{proof}[Conclusion of the Proof of Proposition \ref{localregprop}]
	The operator in \eqref{differentialofoperator} is invertible for smooth $u$ between the spaces $\tilde{W}^{1,p}_{0}([0,T] \times S^1)$ and $L^{p}([0,T] \times S^1)$, as we have seen in Lemma \ref{lemmakernelbyfredholm}. Thus, arguing as in \cite[p.122]{hamilton} and invoking the Inverse Function Theorem for Banach spaces, we may deduce local existence of solutions to the fractional harmonic map equation in $W^{1,p}$, for $p > 2$. Observe that we use smooth boundary values $u_0$ at $t=0$ to construct a smooth solution $u$ to the fractional heat equation $u_t + (-\Delta)^{1/2} u = 0$ with $u(0) = u_0$.  Indeed, such a solution exists and is smooth by using the explicit formula obtained from the Fourier coefficients of $u_0$:
	\begin{equation}
		u(t,x) = \sum_{n \in \mathbb{Z}} \hat{u}_0 (n) e^{- | n | t} e^{inx}, \quad \forall t \in \R_{+}, \forall x \in S^1
	\end{equation}
	It can be directly verified that this is a smooth solution of the homogeneous fractional heat equation.
	
	We then consider the operator $h \mapsto H(u + h)$ for $h$ with vanishing initial datum, which is thus locally invertible. This is also the situation in \cite{hamilton} and the key idea is to observe that if $f := H(u)$, then for $\tilde{f}_{\delta}$ being $0$ for $[0,\delta]$ and agreeing with $f$ for other times, then for $\delta > 0$ sufficently small, we know that $\tilde{f}_{\delta}$ lies in the image of $h \mapsto H(u+h)$, meaning that there is a $\tilde{h}_{\delta}$ such that $H(u+\tilde{h}_{\delta}) = \tilde{f}_{\delta}$. Then, $\tilde{u}_{\delta} := u + \tilde{h}_{\delta}$ is a local solution of the half-harmonic map equation up to some time $\delta > 0$ with the initial data $u_0$.
	
	It should be observed that then the local solution, i.e. only on a subinterval of $[0,T]$, to the fractional harmonic gradient flow is also $C^{\infty}$ up to some time. This can be proven analogous to the bootstrap for $h$ above. Thanks to this smoothness property of the local solution, we may also deduce that $u$ assumes values in $S^{n-1}$ by following the arguments in \cite{hamilton} and using similar tricks as above when we were proving $h = 0$ for solutions to $DH(u) h = 0$ with vanishing initial datum. We emphasise that is suffices to verify:
	$$| u |^2 = 1 \text{ a.e. } \Rightarrow | u |^2 - 1 = 0 \text{ a.e. },$$
	which can be seen again by using uniqueness of the solution to a specific flow. Namely, if $u$ solves the half-harmonic gradient flow and is smooth, then we may deduce:
	\begin{align}
		\partial_{t} \left( | u |^2 - 1 \right)	&= 2 \partial_{t} u \cdot u \notag \\
								&= -2 (-\Delta)^{1/2} u \cdot u + | u |^2 | d_{1/2} u |^2 \notag \\
								&= -(\Delta)^{1/2} \left( | u |^2 - 1 \right) - | d_{1/2} u |^2 + | u |^{2} | d_{1/2} u |^2 \notag \\
								&= -(\Delta)^{1/2} \left( | u |^2 - 1 \right) + \left( | u |^{2} - 1 \right) | d_{1/2} u |^2,
	\end{align}
	using \eqref{computationofhalflaplacianonsomespecialfunction} and therefore, the function $v := | u |^2 - 1$ satisfies the flow equation:
	$$v_{t} + (-\Delta)^{1/2} v = v | d_{1/2} u |^2$$
	One should observe that by assumption, $u(0) \in S^{n-1}$ everywhere, so $v(0) = 0$. Thus, arguing completely analogous to the proof of Theorem \ref{uniqueness}, we can easily deduce that $v = 0$ everywhere and therefore that $u \in S^{n-1}$ for all $t$ and $x$. This now concludes the proof of Proposition \ref{localregprop}.
	\end{proof}
	
	By uniqueness of the solutions to the fractional harmonic gradient flow, this shows that the solutions to \eqref{gradflowuv} are smooth, provided the initial value is smooth, at least for small times.
	
	\subsubsection{Approximation and Global Regularity}
	
	It remains to check that regularity holds for all times and remove the restriction that the initial datum needs to be smooth. Both follow by arguing as in \cite{struwe1}. Firstly, we have the following result which will be crucial in reducing our considerations to the smooth case:
	
	\begin{lem}
	\label{approximationlemmauhlenbeck}
		Let $u \in H^{1/2}(S^1;S^{n-1})$. Then there exists a sequence $u_k \in C^{\infty}(S^1) \cap H^{1/2}(S^1;S^{n-1})$ such that:
		$$\| u_k - u \|_{H^{1/2}(S^1)} \to 0, \quad n \to \infty.$$
	\end{lem}
	
	This Lemma is a fractional version of an analogous result proven by Schoen-Uhlenbeck in \cite{schoenuhlenbeck}, our proof follows the computations in \cite{struwelecturenotes}.
	
	\begin{proof}
		Let $\rho$ be a smooth, non-negative function on $S^1$ supported on a strict compact subset of $S^1$ with $\int_{S^1} \rho dx = 1$ and define $\rho_{\varepsilon}$ as usual by:
		$$\rho_{\varepsilon}(x) := \frac{1}{\varepsilon} \rho \left( \frac{x}{\varepsilon} \right),$$
		for all $0 < \varepsilon < 1$. We shall assume that the support of $\rho$ is $B_{1}(0)$, using the identification $S^1 \simeq \R / 2 \pi \Z$. Then, as usual for approximations of the identity, we know:
		$$\tilde{u}_{\varepsilon} := \rho_{\varepsilon} \ast u \to u \text{ in $H^{1/2}(S^1;\R^n)$},$$
		and all convolutions are smooth. Moreover, we have:
		\begin{align}
		\label{closenesstosphere}
			 d\left( \rho_{\varepsilon} \ast u(x), S^{n-1} \right)	&\leq \inf_{z \in S^1} \left| \int_{S^1} \rho_{\varepsilon}(y) u(y) dy - u(z) \right| \notag \\
			 										&\leq \int_{S^1} \left| \int_{S^1} \rho_{\varepsilon}(x-y) u(y) dy - u(z) \right| \rho_{\varepsilon}(x-z) dz \notag \\
													&\leq \int_{S^1} \int_{S^1} \rho_{\varepsilon}(x-y) \rho_{\varepsilon}(x-z) \left| u(y) - u(z) \right| dy dz \notag \\
													&\leq \left( \int_{S^1} \int_{S^1} \rho_{\varepsilon}(x-y) \rho_{\varepsilon}(x-z) \left| u(y) - u(z) \right|^2 dy dz \right)^{1/2} \notag \\
													&\leq \frac{C}{\varepsilon} \left( \int_{B_{\varepsilon}(x)} \int_{B_{\varepsilon}(x)} \left| u(y) - u(z) \right|^2 dy dz \right)^{1/2} \notag \\
													&\lesssim \frac{1}{\varepsilon} \left( \int_{B_{\varepsilon}(x)} \int_{B_{\varepsilon}(x)} \frac{\left| u(y) - u(z) \right|^2}{| x-y |^{2}} \varepsilon^2 dy dz \right)^{1/2} \notag \\
													&\sim \left( \int_{B_{\varepsilon}(x)} \int_{B_{\varepsilon}(x)} \frac{\left| u(y) - u(z) \right|^2}{| x-y |^{2}} dy dz \right)^{1/2} \notag \\
													&\lesssim \| u \|_{\dot{H}^{1/2}(S^1)},								
		\end{align}
		where we used H\"older's inequality in the fourth line. Observe that we may thus use the absolute continuity of the integral in order to see that the distance between $\rho_{\varepsilon} \ast u$ and $S^{n-1}$ becomes arbitrarily small, as $\varepsilon > 0$ goes to $0$. Thus, for $\varepsilon > 0$ small enough, $\rho_{\varepsilon} \ast u$ is never $0$ and thus we may use the projection $\pi: \R^n \setminus \{ 0 \} \to S^{n-1}, \pi(x) = x/| x |$ and apply it to the convolution. Hence, we may define:
		$$u_{\varepsilon} := \pi \left( \rho_{\varepsilon} \ast u \right)$$
		Clearly, these functions satisfy:
		$$u_{\varepsilon} \in C^{\infty}(S^1) \cap H^{1/2}(S^1;S^{n-1})$$
		Moreover, as $\pi$ is Lipschitz on compact domains, we may also deduce, for $\varepsilon > 0$ sufficiently small, that:
		$$u_{\varepsilon} \text{ is bounded in } H^{1/2}(S^1)$$
		Therefore, an appropriate subsequence, which we shall now denote by $u_k$ converges weakly in $H^{1/2}(S^1)$ to $u$ and strongly in $L^{2}(S^1)$ as well as almost everywhere pointwise. Additionally, by weak lower semicontinuity of the seminorm:
		$$\| u \|_{H^{1/2}(S^1)} \leq \liminf_{n \to \infty} \| u_{k} \|_{H^{1/2}(S^1)}$$
		It suffices to check that we have:
		\begin{equation}
		\label{limsupboundv1}
			\limsup_{k \to \infty} \| u_{k} \|_{H^{1/2}(S^1)} \leq \| u \|_{H^{1/2}(S^1)},
		\end{equation}
		since then, we also know:
		$$\lim_{k \to \infty} \| u_{k} \|_{H^{1/2}(S^1)} = \| u \|_{H^{1/2}(S^1)},$$
		which, combined with the weak convergence and the Hilbert space structure of $H^{1/2}(S^1;\R^n)$, shows that $u_k \to u$ strongly in $H^{1/2}(S^1;\R^n)$.\\
		
		Instead of \eqref{limsupboundv1}, it also suffices to verify:
		\begin{equation}
		\label{limsupboundv2}
			\limsup_{k \to \infty} \left( \int_{S^1} \int_{S^1} \frac{| u_k (x) - u_k (y) |^2}{| x-y |^2} dy dx \right)^{1/2} = \limsup_{k \to \infty} \| u_{k} \|_{\dot{H}^{1/2}(S^1)} \leq \| u \|_{\dot{H}^{1/2}(S^1)},
		\end{equation}
		and to deduce this, let us notice:
		\begin{align}
			\left( \int_{S^1} \int_{S^1} \frac{| u_k (x) - u_k (y) |^2}{| x-y |^2} dy dx \right)^{1/2}	&= \left( \int_{S^1} \int_{S^1} \frac{| \pi( \tilde{u}_k (x)) - \pi( \tilde{u}_k (y) ) |^2}{| x-y |^2} dy dx \right)^{1/2} \notag \\
																			&= Lip(\pi) \cdot \left( \int_{S^1} \int_{S^1} \frac{| \tilde{u}_k (x) - \tilde{u}_k (y) |^2}{| x-y |^2} dy dx \right)^{1/2} \notag \\
																			&= Lip(\pi) \cdot \| \tilde{u}_{k} \|_{\dot{H}^{1/2}(S^1)}
		\end{align}
		where $Lip(\pi) > 0$ denotes the Lipschitz constant associated with $\pi$. We observe that for $k$ big, we may ensure that $\tilde{u}_{k}$ becomes arbitrarily close to $S^{n-1}$, see \eqref{closenesstosphere}. We now just have to argue that for sufficently small neighbourhoods of $S^{n-1}$, the constant $Lip(\pi)$ can be chosen arbitrarily close to $1$. If this was true, then for any $\delta > 0$ and $n$ big enough, we would find:
		$$\left( \int_{S^1} \int_{S^1} \frac{| u_k (x) - u_k (y) |^2}{| x-y |^2} dy dx \right)^{1/2} \leq (1+ \delta) | \tilde{u}_{k} \|_{\dot{H}^{1/2}(S^1)} \to (1 + \delta) \| u \|_{\dot{H}^{1/2}(S^1)},$$
		which would imply:
		$$\limsup_{k \to \infty} \| u_{k} \|_{H^{1/2}(S^1)} \leq (1+\delta) \| u \|_{H^{1/2}(S^1)},$$
		for all $\delta > 0$. The desired result then follows by letting $\delta \to 0$.\\
		
		To deduce that the Lipschitz constant becomes arbitrarily small, we first observe that the function is Lipschitz in neighbourhoods of the $n-1$-sphere due to smoothness. Assume now that there is a $\delta > 0$, such that:
		$$\sup_{k \in \mathbb{N}} \sup_{x,y \in B_{\frac{1}{k}}(S^{n-1})} \frac{| \pi(x) - \pi(y) |}{| x-y |} \geq 1 + \delta,$$
		where $B_{\frac{1}{k}}(S^{n-1})$ denotes $1/k$-neighbourhood of $S^{n-1}$. Choose then sequences $x_k, y_k \in B_{\frac{1}{k}}(S^{n-1}) \subset \R^n$ such that:
		$$\frac{| \pi(x_k) - \pi(y_k) |}{| x_k -y_k |} \geq 1 + \delta$$
		Since the sequences are bounded, they have converging subsequences, still denoted by $x_k, y_k$, with limits $x_0, y_0 \in S^{n-1}$. If $x_0 \neq y_0$, then we see:
		$$1 + \delta \leq \frac{| \pi(x_k) - \pi(y_k) |}{| x_k -y_k |} \to \frac{| \pi(x_0) - \pi(y_0) |}{| x_0 -y_0 |} = 1,$$
		which is a contradiction. So $x_0 = y_0$. But in this case, we know that for $k$ sufficiently large, we also have that $x_k$ and $y_k$ remain in any given neighbourhood of $x_0 = y_0$. Choosing the neighbourhood small enough, we may assume that it is convex and that the differential of $\pi$ has operator norm $< 1 + \delta/2$ for all points in the neighbourhood. The former is clear and the latter relies on $\pi$ being smooth and $d \pi(x)$ being the orthogonal projection to the tangent plane at any given point $x \in S^{n-1}$, thus having operatornorm $1$. Standard arguments then show that on such a neighbourhood of $x_0 = y_0$, $\pi$ is Lipschitz with Lipschitz constant $\leq 1 + \delta/2$, again contradicting our choice of $x_k, y_k$. Therefore, we may conclude as previously outlined.
	\end{proof}
	
	If we obtain uniform existence intervals and bounds depending merely on the energy, we may deduce regularity for general $u_0$ by the same result for smooth initial data using Lemma \ref{approximationlemmauhlenbeck} and treat the general case analogous to \cite{struwe1} by approximation. So we may focus our attention on the smooth case.\\
	
	The main idea is now to establish uniform bounds for solutions to the half-harmonic gradient flow that shall only depend on the energy and other harmless quantities and apply results like in \cite{hieber} to establish higher regularity and extensiability of solutions in a smooth way after any given time, similar to \cite{struwe1}. In order to do so, we shall first adapt Lemma 3.1 and Lemma 3.2 in \cite{struwe1} to our current situation:
	
	\begin{lem}
	\label{struwelemma3.1}
		There exist $C >  0$ not depending on $R, u, T$, such that for any smooth $u$ on $[0,T] \times S^1$ and $0 < R < 1$, the following estimate holds for all $x_0 \in S^1$:
		\begin{align}
		\label{struweest01}
			\int_{0}^{T} \int_{B_{\frac{3R}{4}}(x_0)} | (-\Delta)^{1/4} u |^4 dx dt 	&\leq C \sup_{0\leq t \leq T} \int_{B_{R}(x_0)} | (-\Delta)^{1/4} u(t) |^2 dx \notag \\
																&\cdot \left( \int_{0}^{T} \int_{B_{R}(x_0)} | (-\Delta)^{1/2} u |^2 dx dt + \frac{1}{R^2} \int_{0}^{T} \int_{S^1} | (-\Delta)^{1/4}u |^2 dx dt \right),
		\end{align}
		by density the same result applies for all $u \in H^{1}([0,T] \times S^1)$ with bounded $1/2$-Dirichlet energy. Similarily, we have:
		\begin{align}
		\label{struweest02}
			\int_{0}^{T} \int_{S^1} | (-\Delta)^{1/4} u |^4 dx dt 	&\lesssim \sup_{0\leq t \leq T, x \in S^1} \int_{B_{R}(x)} | (-\Delta)^{1/4} u(t) |^2 dx \notag \\
															&\cdot \left( \int_{0}^{T} \int_{S^1} | (-\Delta)^{1/2} u |^2 dx dt + \frac{1}{R^3} \int_{0}^{T} \int_{S^1} | (-\Delta)^{1/4}u |^2 dx dt \right).
		\end{align}
	\end{lem}
	
	The proof follows \cite{struwe1} and we refer to this reference for further details.
	
	\begin{proof}
		We only treat the case $x = 0$, again using $S^1 \simeq \R / 2\pi \Z$, the general one follows by a simple rotation. Let $\varphi$ be a smooth function supported on $B_{1}(0)$ and satisfying $0 \leq \varphi \leq 1$ as well as $\varphi = 1$ on $B_{3/4}(0)$. Then we define $\varphi_{R}(x) := \varphi(\frac{x}{R})$ for any $0 < R < 1$. For brevity, we shall suppress the subscript $R$ in the following computations. We estimate using the Ladyzhenskaya-type inequality in Lemma \ref{ladys1} on lines with fixed time $t$:
		\begin{align}
		\label{proofstruwelemma3.1est01}
			&\int_{0}^{T} \int_{B_{\frac{3R}{4}}(0)} | (-\Delta)^{1/4} u |^4 dx dt	\notag \\
														&\leq \int_{0}^{T} \int_{S^1} | (-\Delta)^{1/4} u |^4 | \varphi |^4 dx dt \notag \\
														&\lesssim \int_{0}^{T} \int_{S^1} | (-\Delta)^{1/4} u \cdot \varphi - c |^4 dx dt + \int_{0}^{T} \int_{S^1} | c |^4 dx dt \notag \\
														&\lesssim \int_{0}^{T} \int_{S^1} | (-\Delta)^{1/4} u \cdot \varphi - c |^2 dx \cdot \int_{S^1} \left| (-\Delta)^{1/4} \left( (-\Delta)^{1/4} u \cdot \varphi \right) \right|^2 dx dt +\int_{S^1} | c |^4 dx dt \notag \\
														&\lesssim \sup_{0 \leq t \leq T} \int_{S^1} | (-\Delta)^{1/4} u(t) \cdot \varphi |^2 dx \cdot \int_{0}^{T} \int_{S^1} \left| (-\Delta)^{1/4} \left( (-\Delta)^{1/4} u \cdot \varphi \right) \right|^2 dx dt +\int_{S^1} | c |^4 dx dt \notag \\
														&\lesssim \sup_{0 \leq t \leq T} \int_{B_{R}(0)} | (-\Delta)^{1/4} u(t) |^2 dx \cdot \int_{0}^{T} \int_{S^1} \left| (-\Delta)^{1/4} \left( (-\Delta)^{1/4} u \cdot \varphi \right) \right|^2 dx dt +\int_{S^1} | c |^4 dx dt
		\end{align}
		where $c$ is defined to be the average of $(-\Delta)^{1/4} u \cdot \varphi$ over $S^1$ and thus also the $0$-th Fourier coefficient. Observe that the removal of the Fourier coefficient at $0$ actually justifies the use of the seminorm above. Moreover, in the fourth inequality, we use that we can remove $c$ due to minimality, cf. \cite{struwe1}.
		
		Let us now observe the following:
		\begin{align}
		\label{proofstruwelemma3.1est02}
			\int_{0}^{T} \int_{S^1} | c |^4 dx dt	&= \int_{0}^{T} \int_{S^1} \left| \int_{S^1} (-\Delta)^{1/4} u(y) \varphi(y) dy \right|^4 dx dt \notag \\
										&\lesssim \int_{0}^{T} \int_{S^1} \left| \int_{S^1} | (-\Delta)^{1/4} u(y) |^2 \varphi(y) dy \right|^2 \cdot R^2  dx dt \notag \\ 
										&\lesssim \sup_{0\leq t \leq T} \int_{S^1} | (-\Delta)^{1/4} u(t) |^2 \varphi dx \cdot \int_{0}^{T} \int_{S^1} | (-\Delta)^{1/4} u(y) |^2 \varphi(y) dy  dt \notag \\
										&\leq \sup_{0\leq t \leq T} \int_{B_{R}(0)} | (-\Delta)^{1/4} u(t) |^2 dx \cdot \int_{0}^{T} \int_{B_{R}(0)} | (-\Delta)^{1/4} u(y) |^2 dy  dt
		\end{align}
		as $R < 1$. On the other hand, we may observe that:
		\begin{align}
		\label{proofstruwelemma3.1est03}
			(-\Delta)^{1/4} \left( (-\Delta)^{1/4} u \cdot \varphi \right)(x)	&= P.V. \int_{S^1} \frac{(-\Delta)^{1/4}u(x) \varphi(x) - (-\Delta)^{1/4}u(y) \varphi(y)}{| x-y |^{3/2}} dy \notag \\
														&= (-\Delta)^{1/4} \left( (-\Delta)^{1/4} u \right)(x) \varphi(x) + P.V. \int_{S^1} (-\Delta)^{1/4}u(y) \frac{\varphi(x) - \varphi(y)}{| x-y |^{3/2}} dy \notag \\
														&= (-\Delta)^{1/2} u(x) \varphi(x) + P.V. \int_{S^1} (-\Delta)^{1/4}u(y) \frac{\varphi(x) - \varphi(y)}{| x-y |^{3/2}} dy 
		\end{align}
		The latter summand satisfies the following estimate:
		\begin{align}
		\label{proofstruwelemma3.1est04}
			&\int_{0}^{T} \int_{S^1} \left| P.V. \int_{S^1} (-\Delta)^{1/4}u(y) \frac{\varphi(x) - \varphi(y)}{| x-y |^{3/2}} dy \right|^2 dx dt \notag \\
			%&\lesssim \int_{0}^{T} \int_{S^1} \int_{S^1} | (-\Delta)^{1/4}u(y) |^2 \left| \frac{\varphi(x) - \varphi(y)}{| x-y |^{3/2}} \right|^2 dy dx dt \notag \\
			&\lesssim \int_{0}^{T} \int_{S^1} \int_{S^1} | (-\Delta)^{1/4}u(y) |^2 \frac{1}{| x-y |^{1/2}} dy \cdot \int_{S^1} \frac{| \varphi(x) - \varphi(y) |^2}{| x-y |^{5/2}} dy dx dt \notag \\
			&\lesssim \int_{0}^{T} \int_{S^1} \int_{S^1} | (-\Delta)^{1/4}u(y) |^2 \frac{1}{| x-y |^{1/2}} dy \cdot \int_{S^1} \frac{\| \varphi \|_{L^{\infty}}^2}{| x-y |^{1/2}} dy dx dt \notag \\
			&\lesssim \frac{1}{R^2} \int_{0}^{T} \int_{S^1} \int_{S^1} | (-\Delta)^{1/4}u(y) |^2 \frac{1}{| x-y |^{1/2}} dy dx dt \notag \\
			&\lesssim \frac{1}{R^2} \int_{0}^{T} \int_{S^1} | (-\Delta)^{1/4}u(y) |^2 dy dt
		\end{align}
		where we used that $\varphi = \varphi_{R}$ to obtain the uniform estimate in $R$. Combining \eqref{proofstruwelemma3.1est01}, \eqref{proofstruwelemma3.1est02}, \eqref{proofstruwelemma3.1est03} and \eqref{proofstruwelemma3.1est04}, we therefore have:
		\begin{align}
		\label{proofstruwelemma3.1est05}
			\int_{0}^{T} \int_{B_{\frac{3R}{4}}(0)} | (-\Delta)^{1/4} u |^4 dx dt 	&\lesssim \sup_{0\leq t \leq T} \int_{B_{R}(0)} | (-\Delta)^{1/4} u(t) |^2 dx \notag \\
															&\cdot \left( \int_{0}^{T} \int_{B_{R}(0)} | (-\Delta)^{1/2} u |^2 dx dt + \frac{1}{R^2} \int_{0}^{T} \int_{S^1} | (-\Delta)^{1/4}u |^2 dx dt \right)
		\end{align}
		and the constant does not depend on $R$, $u$ or $T$. As already noted at the beginning, the same inequality holds for all $x \in S^1$ instead of $0$.\\
		
		We may now cover $S^1$ by $\lceil \frac{2\pi}{\frac{3R}{4}} \rceil = \lceil \frac{8\pi}{3R} \rceil$ balls of radius $\frac{3R}{4}$ around points on $S^1$, such that each point is contained in at most $3$ balls. Then, by adding the inequalities \eqref{proofstruwelemma3.1est05} in these points, we find:
		\begin{align}
		\label{proofstruwelemma3.1est06}
			\int_{0}^{T} \int_{S^1} | (-\Delta)^{1/4} u |^4 dx dt 	&\lesssim \sup_{0\leq t \leq T, x \in S^1} \int_{B_{R}(x)} | (-\Delta)^{1/4} u(t) |^2 dx \notag \\
															&\cdot \left( \int_{0}^{T} \int_{S^1} | (-\Delta)^{1/2} u |^2 dx dt + \frac{1}{R^3} \int_{0}^{T} \int_{S^1} | (-\Delta)^{1/4}u |^2 dx dt \right)
		\end{align}
		Observe that we have the power $R^{-3}$ showing up due to including $\sim 1/R$ balls for any given $R$.
	\end{proof}
	
	As in \cite{struwe1}, we shall use the following notation for $0 < R < 1$ and $t \in [0,T]$:
	\begin{equation}
		E_{R}(u;x,t) := \frac{1}{2} \int_{B_{R}(x)} | (-\Delta)^{1/4} u(t) |^2 dx,
	\end{equation}
	for the local energy and introduce:
	\begin{equation}
		\varepsilon(R) = \varepsilon(R;u,T) := \sup_{x \in S^1, t \in [0,T]} E_{R}(u;x,t)
	\end{equation}
	In analogy to Lemma 3.6 in \cite{struwe1}, we have the following energy estimate:
	
	\begin{lem}
	\label{struwelemma3.6}
		There exists a constant $C > 0$ such that for every $u: [0,T] \times S^1 \to S^{n-1}$ in $H^{1}([0,T] \times S^1) \cap L^{\infty}([0,T]; \dot{H}^{1/2}(S^1))$ solving the half-harmonic flow equation \eqref{gradflowuv} and satisfying the energy decrease property as in Lemma \ref{monotone1/2energydecayinglemma}, any $0 < R < 1/2$ and $(t,x_0) \in [0,T] \times S^1$, the following estimate holds:
		\begin{align}
		\label{struweest03}
			E_{R}(u;x_0,t) 	&\leq E_{2R}(u;x_0,0) + C \left( \frac{t}{R^2} E(u_0) + \frac{\sqrt{t}}{R} \sqrt{\varepsilon(2R) E(u_0)} \right) \notag \\
						&\leq E_{2R}(u;x_0,0) + C \left( \frac{t}{R^2} + \frac{\sqrt{t}}{R} \right) E(u_0),
		\end{align}
		where $E(u_0) = E_{1/2}(u_0)$. In the second inequality, we used the trivial estimate between the local energy and the global one under the energy decay.
	\end{lem}
	
	The proof is as in \cite{struwe1}.
	
	\begin{proof}
		Letting $\varphi$ be any smooth, compactly supported, time-independent function on $B_{2R}(x_0)$, such that $\varphi = 1$ on $B_{R}(x_0)$ and $0 \leq \varphi \leq 1$, $| \nabla \varphi | \lesssim 1/R$ (see our choice in the proof of the previous Lemma). We now test \eqref{gradflowuv} with $u_t \varphi^2$ and observe that $u_t \cdot u = 0$, as $u$ maps to $S^{n-1}$. Therefore, we find:
		\begin{align}
			0 	&= \int_{0}^{t} \int_{S^{1}} | u_{t} |^2 \varphi^2 dx ds + \int_{0}^{t} \int_{S^1} (-\Delta)^{1/2} u \cdot u_t \varphi^2 dx ds \notag \\
				&= \int_{0}^{t} \int_{S^{1}} | u_{t} |^2 \varphi^2 dx ds + \int_{0}^{t} \int_{S^1} (-\Delta)^{1/4} u \cdot (-\Delta)^{1/4} \left( u_t \varphi^2 \right) dx ds
		\end{align}
		We observe that for smooth $f$:
		\begin{align}
			(-\Delta)^{1/4} \left( f \varphi^2 \right)(x) = (-\Delta)^{1/4} f(x) \varphi(x)^2 + P.V. \int_{S^1} f(y) \frac{\varphi(x)^2 - \varphi(y)^2}{| x-y |^{3/2}} dy,
		\end{align}
		and therefore:
		\begin{align}
			\int_{S^1} (-\Delta)^{1/4} u \cdot (-\Delta)^{1/4} \left( f \varphi^2 \right) dx 	&= \int_{S^1} (-\Delta)^{1/4} u \cdot (-\Delta)^{1/4} f \cdot \varphi^2 dx \notag \\
																	&+ \int_{S^1} (-\Delta)^{1/4} u \cdot P.V. \int_{S^1} f(y) \frac{\varphi(x)^2 - \varphi(y)^2}{| x-y |^{3/2}} dy dx
		\end{align}
		by approximation, the same holds true for $L^2$-functions like $u_t$, and thus:
		\begin{align}
			&\int_{0}^{t} \int_{S^{1}} | u_{t} |^2 \varphi^2 dx ds + E_{R}(u;x_0,t) - E_{2R}(u;x_0,0) \notag \\
			&\leq \int_{0}^{t} \int_{S^{1}} | u_{t} |^2 \varphi^2 dx ds + \int_{0}^{t} \int_{S^1} \frac{1}{2} \frac{d}{dt} \left( | (-\Delta)^{1/4} u |^2 \varphi^2 \right)  dx ds \notag \\
			&\leq \left| \int_{0}^{t} \int_{S^1} (-\Delta)^{1/4} u \cdot P.V. \int_{S^1} u_t (y) \frac{\varphi(x)^2 - \varphi(y)^2}{| x-y |^{3/2}} dy dx ds \right| \notag \\
			&\lesssim \frac{1}{R} \Big{|} \int_{0}^{t} \int_{S^1} (-\Delta)^{1/4} u \cdot P.V. \int_{S^1} u_t (y) \varphi(y) \frac{1}{| x-y |^{1/2}} dy dx ds \Big{|} \notag \\
			&+ \frac{1}{R} \Big{|} \int_{0}^{t} \int_{S^1} (-\Delta)^{1/4} u \cdot P.V. \int_{S^1} u_t (y) \varphi(x) \frac{1}{| x-y |^{1/2}} dy dx ds \Big{|},
		\end{align}
		where we used the estimate for the gradient of $\varphi$ in the last line and $\varphi(x)^2 - \varphi(y)^2 = (\varphi(x) + \varphi(y)) (\varphi(x) - \varphi(y))$. Using H\"older's inequality, the RHS may be bounded by, up to a constant:
		$$\frac{\sqrt{t}}{R} \sqrt{E_{1/2}(u_0)} \left( \int_{0}^{t} \int_{S^1} | u_t |^2 \varphi^2 dx dt \right)^{1/2} + \frac{\sqrt{t}}{R} \sqrt{\varepsilon(2R) E_{1/2}(u_0)},$$
		the latter summand following from (the first summand may be estimated analogously):
		\begin{align}
			&\frac{1}{R} \Big{|} \int_{0}^{t} \int_{S^1} (-\Delta)^{1/4} u \cdot P.V. \int_{S^1} u_t (y) \varphi(x) \frac{1}{| x-y |^{1/2}} dy dx ds \Big{|} \notag \\
			&\lesssim \frac{1}{R} \Big{|} \int_{0}^{t} \int_{S^1} \int_{S^1} | (-\Delta)^{1/4} u(x) |^2 \varphi(x)^2 \frac{1}{| x-y |^{1/2}} dy dx ds \cdot \Big{|}^{1/2} \Big{|} \int_{0}^{t} \int_{S^1} \int_{S^1} | u_t (y) |^2 \frac{1}{| x-y |^{1/2}} dy dx ds \Big{|}^{1/2} \notag \\
			&\lesssim \frac{1}{R} \Big{|} \int_{0}^{t} \int_{B_{2R}(x_0)} | (-\Delta)^{1/4} u |^2 \varphi^2 dx ds \cdot \Big{|}^{1/2} \Big{|} \int_{0}^{t} \int_{S^1} | u_t |^2 dx ds \Big{|}^{1/2} \notag \\
			&\lesssim \frac{\sqrt{t}}{R} \sqrt{\varepsilon(2R) E(u_0)},
		\end{align}
		where the second factor can be estimated as in Lemma 3.4 of \cite{struwe1}, see also the monotone energy decay estimate for solutions of the half-harmonic flow. Therefore, the result follows after absorption in an obvious manner.
	\end{proof}
	
	Having these tools available renders us able to establish the results (of course slightly adapted to our current situation) in Lemma 3.7, 3.8 and 3.10 of \cite{struwe1} and thus establish uniform $L^{p}$-estimates for the RHS of the fractional gradient flow \eqref{gradflowuv} under restrictions on the local energy, global energy, $R$ and $T$ and independent of $u$. Let us state the appropriate adaptions to our current situation:
	
	\begin{lem}
	The following generalisations of the results in \cite{struwe1} hold true:
		\begin{enumerate}
			\item Lemma 3.7 in \cite{struwe1}: There exists $\epsilon_1 > 0$ such that for any $u \in H^{1}([0,T] \times S^1) \cap L^{\infty}([0,T]; H^{1/2}(S^1))$ solving \eqref{gradflowuv} and any $R < 1/2$, there holds:
			\begin{equation}
				\int_{0}^{T} \int_{S^1} | \nabla u |^2 dx dt \leq C E(u_0) \left( 1 + \frac{T}{R^3} \right),
			\end{equation}
			with $C$ independent of $u, T, R$, provided $\varepsilon(R) < \varepsilon_{1}$. Here, $u(0, \cdot) = u_0$ is the initial value.
			
			\item Lemma 3.8, Remark 3.9 in \cite{struwe1}: For any numbers $\varepsilon, \tau, E_0 > 0$, and if $u_0$ is smooth also $\tau = 0$, and $R_1 < 1/2$, there is a $\delta > 0$ such that for any $u$, satisfying the conditions as in 1., solving \eqref{gradflowuv} and any $I \subset [\tau, T]$ with measure $| I | < \delta$, there holds:
			\begin{equation}
				\int_{I} \int_{S^1} | (-\Delta)^{1/4} u |^2 dx dt < \varepsilon,
			\end{equation}
			provided $\varepsilon(R_1) < \varepsilon_1, E(u_0) \leq E_0$.
			
			\item Lemma 3.10, Remark 3.11 in \cite{struwe1}: Let $u$ be, in addition to the assumptions in 1., a $C^2([\tau, T] \times S^1)$-solution to \eqref{gradflowuv}, then, for every $1 \leq p < +\infty$, there exists a $L^{p}([\tau, T] \times S^1)$-bound on $u_t + (-\Delta)^{1/2}u$ with a constant only depending on $E(u_0), \tau, T$ and $R$, provided $\varepsilon( R ) < \varepsilon_1$. Here, $\tau > 0$ in general and $\tau \geq 0$ in case $u_0$ is smooth.
		\end{enumerate}
	\end{lem}
	
	For example, Lemma 3.7 in \cite{struwe1} follows by using $(-\Delta)^{1/2} u$ instead of $\Delta u$ and applying the estimates in Lemma \ref{struwelemma3.1}. Lemma 3.8 relies on choosing subsequences which can equally well be chosen for $(-\Delta)^{1/4} u$ and $(-\Delta)^{1/2} u$, compactness remains valid and the energy estimate in Lemma \ref{struwelemma3.6} replaces local energy estimate used in \cite{struwe1}. Naturally, Remark 3.9 also carries over, as the uniform absolute continuity is guaranteed by Lemma 3.8. Lastly, arguing as in \cite{struwe1} Lemma 3.10, using twice differentiable solutions of the half-harmonic flow, we may differentiate with respect to $t$ and test against $u_t$ to deduce precisely the same estimates for the $L^p$-norm of the RHS independent of $u$, i.e. only depending on the analogous terms as in Lemma 3.10 of \cite{struwe1}.
	
	This also leads to higher order estimates following the bootstrap techniques above and using the result \cite[Theorem 3.1]{hieber}, meaning that we may establish regularity up to time $T$. Extending as in \cite{struwe1} by restarting the flow at $T$ and using approximating sequences as in Lemma \ref{approximationlemmauhlenbeck} then show regularity of solutions with arbitrary initial datum by uniform convergence on sets with $t$ strictly bounded from below (Remark 3.11 applies to the case of regular initial datum, so in this case smoothness is also given at $t=0$). We emphasise that if we choose the initial energy sufficiently small, the localised energy $E_{R}$ will satisfy the necessary inequalities for all times, meaning global smooth existence is justified.\\
	
	We highlight at this point that the argument presented provides an alternative existence argument for the fractional harmonic gradient flow with values in $S^{n-1}$. Moreover, the techniques introduced can be used in order to study finite blow-up times and in the future investigate the types of blow-ups that can occur in finite time.

	\subsection{Convergence}
	
	Another important question is whether or not the solution $u$ of the fractional harmonic gradient flow converges as $t \to +\infty$, or rather for specific subsequences $t_k \to +\infty$. The considerations are completely analogous to \cite{struwe} and \cite{struwe1}.
	
	\begin{thm}
		Let $u \in L^{2}(\R_{+}; H^{1/2}(S^1))$ and $u_t \in L^{2}(\R_{+}; L^{2}(S^1))$ be a solution of the fractional harmonic gradient flow \eqref{gradflowuv} with values in $S^{n-1} \subset \R^{n}$ and with initial data $u_0$. Assume that:
		$$\| (-\Delta)^{1/4} u(t) \|_{L^2} \leq \| (-\Delta)^{1/4} u_0 \|_{L^2} \leq \varepsilon, \quad \forall t \in \R_{+},$$ 
		for $\varepsilon > 0$ sufficiently small. Then, for a suitably chosen subsequence $t_k \to +\infty$, the sequence of maps $(u(t_k, \cdot))_{k \in \mathbb{N}} \subset H^{1}(S^1;S^{n-1})$ converges weakly in $H^{1}(S^1)$ to a $1/2$-harmonic map in $S^{n-1}$.
	\end{thm}
	
	The proof proceeds completely analogous to the one for Theorem 6.6 in \cite{struwe}.
	
	\begin{proof}
	By the considerations in the proof of Proposition \ref{uniqueinenergyclassunderh1assumption}, we know that for $\varepsilon > 0$ sufficiently small, we have for almost every $t$:
	\begin{equation}
	\label{appendixdconveq00}
		\| \nabla u(t) \|_{L^{2}(S^1)} \lesssim \| \partial_{t} u(t) \|_{L^2(S^1)} + 1
	\end{equation}
	As in \cite{struwe}, this implies the following:
	\begin{equation}
	\label{appendixdconveq01}
		\int_{t}^{t+1} \int_{S^1} | \nabla u |^2 dx dt \lesssim \int_{t}^{t+1} \int_{S^1} | u_{t} |^2 dx dt + 1 \lesssim \| u_t \|_{L^2(\R_{+} \times S^1)} + 1,
	\end{equation}
	for all $t \in [0, \infty[$. Observe that the right handside of the estimate is bounded independently of $t$. It is also clear that:
	\begin{equation}
	\label{appendixdconveq02}
		\lim_{t \to + \infty} \int_{t}^{t+1} \int_{S^1} | u_{t} |^2 dx dt = 0,
	\end{equation}
	due to $u_t \in L^2(\R_{+} \times S^1) = L^{2}(\R_{+}; L^{2}(S^1))$.\\
	
	The observations in \eqref{appendixdconveq01} and \eqref{appendixdconveq02} show that we may choose a subsequence $t_k \to \infty$, such that:
	$$u(t_k) \to u_{\infty} \quad \text{ in $H^1(S^1)$ weakly},$$
	and $u_t (t_k) \to 0$ strongly in $L^2$. In fact, we may at first choose $t_k$ such that $L^2$-convergence is satisfied and such that \eqref{appendixdconveq00} holds for all elements in the sequence. Then extracting another subsequence, weak convergence in $H^1(S^1)$ is immediate due to the boundedness in \eqref{appendixdconveq00}. In addition, up to extracting another subsequence, the convergence also holds everywhere pointwise and thus:
	$$u_{\infty}(x) \in S^{n-1} \quad \text{ for almost every } x \in S^1$$
	Let now $\varphi \in C^{\infty}(S^1)$ and test the equation \eqref{gradflowuv} at the time $t_k$ with $\varphi$ which shows:
	\begin{align}
		\int_{S^1} \big{(} u_{t} (t_k) &+ (-\Delta)^{1/2} u(t_k) \big{)} \varphi dx \notag \\
															&= \int_{S^1} u_{t} (t_k) \varphi dx + \int_{S^1} \int_{S^1} d_{1/2} \left( u(t_k) \right)(x,y) d_{1/2}\varphi (x,y) \frac{dy dx}{| x-y |} \notag \\
															&= \int_{S^1} \int_{S^1} | d_{1/2} \left( u(t_k ) \right)(x,y)|^2  u(x) \varphi(x) \frac{dy dx}{| x-y |},
	\end{align}
	and since we know by $u_{t}(t_k) \to 0$ in $L^{2}(S^1)$ strongly and $u(t_k) \to u_{\infty}$ in $H^1(S^1)$ weakly, the left hand side converges for $t_k \to \infty$ to:
	$$\int_{S^1} (-\Delta)^{1/2} u_{\infty} \varphi dx$$
	On the other hand, the right handside does converge as well. Namely, observe that due to the compactness of $H^{1}(S^1) \hookrightarrow H^{1/2}(S^1)$ and H\"older's inequality, we have:
	$$\int_{S^1} \int_{S^1} d_{1/2} \left( u(t_k) - u_{\infty} \right)(x,y) d_{1/2}u(t_k)(x,y) u(x) \varphi(x) \frac{dy dx}{| x-y |} \to 0, \quad \text{ as $t_k \to \infty$}$$
	Similarily, we may see:
	$$\int_{S^1} \int_{S^1} d_{1/2} \left( u(t_k) - u_{\infty} \right)(x,y) d_{1/2}u_{\infty}(x,y) u(x) \varphi(x) \frac{dy dx}{| x-y |} \to 0, \quad \text{ as $t_k \to \infty$}$$
	So we merely have to consider:
	$$\int_{S^1} \int_{S^1} |d_{1/2} u_{\infty}(x,y) |^2 \left( u(x) - u_{\infty}(x) \right) \varphi(x) \frac{dy dx}{| x-y |},$$
	which converges to $0$ as well, which is an immediate consequence of dominated convergence and the boundedness of $u, u_{\infty}$.
	%by using the compactness of the embeddings as before. 
	Thus we have:
	\begin{align}
		\int_{S^1} \int_{S^1} | d_{1/2} u(x,y) |^2 &u(x) \varphi(x) \frac{dy dx}{| x-y |} \notag \\ 
					&\to \int_{S^1} \int_{S^1} | d_{1/2} u_{\infty} (x,y) |^2 u_{\infty}(x) \varphi(x) \frac{dy dx}{| x-y |}
	\end{align}
	So, we find the following by passing to the limit $t_k \to \infty$:
	\begin{equation}
		\int_{S^1} (-\Delta)^{1/2} u_{\infty} \varphi dx = \int_{S^1} u_{\infty} | d_{1/2} u_{\infty} |^{2} \varphi dx,
	\end{equation}
	which is equivalent to:
	$$(-\Delta)^{1/2} u_{\infty} \perp T_{u_{\infty}} N$$
	Therefore, $u_{\infty}$ is actually $1/2$-harmonic.
	\end{proof}
	
	One may even say more. By convergence, we may deduce:
	$$\| (-\Delta)^{1/2} u_{\infty} \|_{L^2} \leq \| (-\Delta)^{1/2} u_0 \|_{L^2} \leq \varepsilon,$$
	meaning an energy bound for the limit function. If $\varepsilon > 0$ is sufficiently small, we may deduce:
	$$u_{\infty} \text{ is a constant map}$$
	This assertion follows by lower energy bounds for $1/2$-harmonic maps, see for example \cite{sireweizheng} and the references presented therein.

	\begin{appendices}
	
	\section{Alternative Conclusion of Theorem \ref{uniqueness}: Estimate \eqref{estd1/2u}}
	
	\subsection{Preliminary Discussion}
	
	The goal of this first appendix is to provide an alternative proof of the final estimate \eqref{estd1/2u} by using direct methods rather than Theorem \ref{schiwangthm1.4} on $S^1$, see also Appendix B. We define the stereographic projection $\Pi: S^{1} \setminus \{ -i \} \to \R$ as follows:
	$$\Pi(\cos(\alpha) + i \sin(\alpha)) := \frac{\cos(\alpha)}{1 + \sin(\alpha)}, \quad \forall \alpha \in \R, \alpha \neq - \frac{\pi}{2} + 2 \pi \mathbb{Z}$$
	Let us state the following result found in \cite{dalio} as Proposition 1.1:
	
	\begin{prop}
	\label{connectioncircline}
		Let $u: \R \to \R^{n}$ and $v := u \circ \Pi: S^1 \setminus \{ -i \} \to \R^n$. Then we have:
		\begin{equation}
			(-\Delta)^{1/2}_{S^1} v(e^{i\theta}) = \frac{(-\Delta)^{1/2}_{\R} u(\Pi(e^{i \theta}))}{1+ \sin( \theta)},
		\end{equation}
		and where we observe:
		\begin{equation}
			\Pi'(\theta) = \frac{1}{1+\sin(\theta)}
		\end{equation}
	\end{prop}
	This hints at a connection between the $1/2$-Laplacian on $S^1$ and the one on $\R$. We would like to exploit this relationship using the stereographic projection in order to apply the result in \cite{schiwang}, namely Theorem \ref{schiwangthm1.4} on $\R$, directly as needed in our proof above. Our starting point is the following identity which was part of an earlier argument, where we now denote by $\Pi(x_0) = x$ and $v := u \circ \Pi^{-1}$. 
	
	\begin{prop}
	\label{propappendixa1}
		We have the following identity for $u,v, x, x_0$ as previously introduced:
		\begin{equation}
			\int_{\R} \frac{| v(x) - v(y) |^2}{| x-y |^2}dy = \int_{S^{1}} \frac{| u(x_{0}) - u(y) |^{2}}{| x_{0}-y |^2} dy \cdot (1 + \sin(x_0))
		\end{equation}
	\end{prop}
	
	\begin{proof}
	After a change of variables and obvious estimates, we arrive at:
	\begin{equation}
	\label{appendixa1ep001}
		\int_{\R} \frac{| v(x) - v(y) |^2}{| x-y |^2}dy = \int_{S^1} \frac{| v(\Pi(x_0)) - v(\Pi(y)) |^2}{4 \sin\left( \frac{x_0 - y}{2} \right)^2} \frac{4 \sin\left( \frac{x_0 -y}{2} \right)^2}{| \Pi(x_0)-\Pi(y) |^2} \frac{1}{1 + \sin(y)} dy
	\end{equation}
	Thus, the fractional gradient norm over $\R$ is bounded for $v$. Let us note:
	\begin{align}
	\label{appendixa1ep002}
		| \Pi(x_0) - \Pi(y) |	&= \frac{| \cos(x_0) + \cos(x_0) \sin(y) - \cos(y) - \cos(y) \sin(x_0) |}{(1+\sin(x_0))(1+\sin(y))} \notag \\
						&= \frac{| \cos(x_0) - \cos(y) + \sin(y - x_0) |}{(1+\sin(x_0))(1+\sin(y))} \notag \\
						&= \frac{| - 2 \sin \left( \frac{y + x_0}{2} \right) \sin \left( \frac{x_0 - y}{2} \right) + 2 \sin \left( \frac{y- x_0}{2} \right) \cos \left( \frac{y - x_0 }{2} \right) |}{(1+\sin(x_0))(1+\sin(y))} \notag \\
						&= \frac{| 2 \sin \left( \frac{y + x_0}{2} \right) \sin \left( \frac{y - x_0}{2} \right) + 2 \sin \left( \frac{y- x_0}{2} \right) \cos \left( \frac{y - x_0 }{2} \right) |}{(1+\sin(x_0))(1+\sin(y))} \notag \\
						&= 2 \frac{| \sin\left( \frac{y - x_0}{2} \right) |}{(1+ \sin(x_0))(1+\sin(y))} \left| \sin\left( \frac{y + x_0}{2} \right) + \cos\left( \frac{y - x_0}{2}\right) \right| \notag \\
						&= 2 \frac{| \sin\left( \frac{y - x_0}{2} \right) |}{(1+ \sin(x_0))(1+\sin(y))} \cdot 2 \left| \sin \left( \frac{y}{2} + \frac{\pi}{4} \right) \sin \left( \frac{x_0}{2} + \frac{\pi}{4} \right)  \right| 
	\end{align}
	Therefore:
	\begin{align}
	\label{appendixa1ep003}
		\frac{4 \sin\left( \frac{x_0 -y}{2} \right)^2}{| \Pi(x_0)-\Pi(y) |^2} \frac{1}{1 + \sin(y))}	&= \frac{(1+ \sin(y)) (1 + \sin(x_0 ))^2}{4\left| \sin \left( \frac{y}{2} + \frac{\pi}{4} \right) \sin \left( \frac{x_0}{2} + \frac{\pi}{4} \right) \right|^2} \notag \\
																		&= \frac{1 + \sin(y)}{2 \left| \sin \left( \frac{y}{2} + \frac{\pi}{4} \right) \right|^2} \cdot \frac{(1 + \sin(x_0))^2}{2 \left| \sin \left( \frac{x_0}{2} + \frac{\pi}{4} \right) \right|^2} 
	\end{align}
	 This is already sufficient to conclude the proof by combining \eqref{appendixa1ep001}, \eqref{appendixa1ep002} and \eqref{appendixa1ep003}. Indeed, it can be obtained by observing that:
	$$\frac{1 + \sin(y)}{2 \left| \sin \left( \frac{y}{2} + \frac{\pi}{4} \right) \right|^2} = 1,$$
	by using the half-angle formula that:
	\begin{align}
		\int_{\R} \frac{| v(x) - v(y) |^2}{| x-y |^2}dy 	&= \int_{S^{1}} \frac{| u(x_{0}) - u(y) |^{2}}{| x_{0}-y |^2} dy \cdot \frac{(1 + \sin(x_0))^2}{2 \left| \sin \left( \frac{x_0}{2} + \frac{\pi}{4} \right) \right|^2} \notag \\
										&= \int_{S^{1}} \frac{| u(x_{0}) - u(y) |^{2}}{| x_{0}-y |^2} dy \cdot (1 + \sin(x_0))
	\end{align}
	and thus providing the desired connecting identity between $\R$ and $S^1$.
	\end{proof}
	
	However, this immediately yields that changing domains by virtue of the stereographic projection is not sufficiently well-behaved for the fractional gradient to preserve arbitrary fractional norms, as the $L^4$-norm does not transform as required and thus obstructing an equivalence between $L^4$ on the circle and on $\R$. The obstruction is visible in the remaining factor $1+\sin(x_0))$ of Proposition \ref{propappendixa1}. Therefore, further ideas are necessary.\\
	
	A different approach involves periodically extending the function on $S^1$ to a function $U$ with a cut-off after a finite number of periods, a technique explored afterwards. Let us present the main ideas informally first: This extension procedure allows us to have an immediate equivalence of the $L^4$-norms at the beginning of \eqref{estd1/2u} and the corresponding one for $U$ with the distance function changed suitably. The argument in \eqref{estd1/2u} then carries on as specified on $\R$. The $(-\Delta)^{1/2} U$ norm can be easily estimated by using the Riesz transform to go over to the classical weak derivative which can be estimated by the corresponding quantity on $S^1$. It thus remains to connect the $L^{2}$-norm of $(-\Delta)^{1/4}U$ with the one of $u$. This is done in the same way as connecting the $L^4$-norms due to the immediate estimates for $| d_{1/2} u |^2$. This then finishes the proof of \eqref{estd1/2u} and therefore also of Lemma \ref{uniqueness}.\\
	
	As a final comment on the previous result, let us mention that naturally, by combining the results for periodic distributions in \cite{schmeitrieb} with the ideas in \cite{schiwang}, we could obtain the very same identifications as there, thus providing the first inequality in \eqref{estd1/2u} immediately for free. This is the very same argument as we already mentioned in section 2 and explored in Appendix B. It then suffices to apply Lemma \ref{ladys1} to conclude.

	\subsection{Estimate for Fractional Gradients using Periodic Extension}
	
	Let us first take $u \in H^{1/2}(S^1)$ and extend it periodically to $\R$ and denote this extension by $U$. Next, we choose any $\varphi \in C^{\infty}_{c}([-3\pi, 3\pi])$ and define:
	$$V := U \cdot \varphi$$
	We may assume that $\varphi = 1$ on $[-2\pi, 2\pi]$ and $\varphi = 0$ for $x \in \R \setminus ]-\frac{5}{2}\pi, \frac{5}{2} \pi[$. We notice that for every $x \in [-\pi, \pi]$:
	\begin{align}
		| d_{1/2} V |^{2}(x)	&= \int_{\R} \frac{| V(x) - V(y) |^2}{| x - y |^2} dy \notag \\
						&\geq \int_{[x-\pi, x + \pi]} \frac{| V(x) - V(y) |^2}{| x - y |^2} dy \notag \\
						&\geq C \int_{S^1} \frac{| u(x) - u(y) |^2}{| x - y |^2} dy \notag \\
						&= C | d_{1/2} u |^2(x)
	\end{align}
	Observe that we exchanged the distance function on the real line for the one on the circle which are equivalent on the interval we consider with constants independent of $x$. Notice that the cut-off function $\varphi$ has been chosen in such a way that the argument works. Therefore, we may deduce:
	$$\int_{S^1} | d_{1/2} u |^4(x) dx \leq C \int_{\R} | d_{1/2} V |^{4}(x) dx$$
	Assuming even that $u \in H^1(S^1)$, it is clear that $V \in H^1(\R)$ and because of $H^{1}(S^1) \subset H^{1/2}(S^1)$, the estimate for the $L^{4}$-norm applies to this situation. We may therefore deduce from the Ladyzhenskaya-type estimate in Lemma \ref{ladys1} and the equivalent characterisation of the norm in \cite{schiwang}, see Theorem \ref{schiwangthm1.4} for $\R$:
	$$\int_{\R} | d_{1/2} V |^{4}(x) dx \sim \| (-\Delta)^{1/4} V \|_{L^4}^{4} \leq C \| (-\Delta)^{1/4} V \|_{L^2}^2 \cdot \| (-\Delta)^{1/2} V \|_{L^2}^2 \leq C' \| (-\Delta)^{1/4} V \|_{L^2}^2 \cdot \| \nabla V \|_{L^2}^2$$
	Notice that the equivalence at the beginning of the estimate is due to \cite{schiwang}. We observe that:
	$$\nabla V = \nabla U \cdot \varphi + U \cdot \nabla \varphi,$$
	therefore the $H^1$-norm of $V$ may be estimated by the $H^1$-norm of $u$:
	$$\| \nabla V \|_{L^2}^2 \leq C \| u \|_{H^1(S^1)}^2$$
	On the other hand, we may deduce that:
	$$\| (-\Delta)^{1/4} V \|_{L^2}^2 \leq C \int_{\R} \int_{\R} \frac{| V(x) - V(y) |^2}{| x - y |^2} dx dy \leq C' \| u \|_{H^{1/2}(S^1)},$$
	where the second inequality is easily established by direct means using the cut-off $\varphi$. Namely, if we write $I := [-3 \pi, 3\pi]$:
	\begin{align}
		\int_{\R} \int_{\R} \frac{| V(x) - V(y) |^2}{| x - y |^2} dx dy	&= \int_{I} \int_{I} \frac{| V(x) - V(y) |^2}{| x - y |^2} dx dy + 2 \int_{I} \int_{I^{c}} \frac{| V(x) |^2}{| x-y |^2} dy dx \notag \\
													&\lesssim \int_{I} \int_{I} \frac{| U(x) - U(y) |^2}{| x-y |^2} dy dx + \int_{I} \int_{I} | U(y) |^2 \| \nabla \varphi \|_{\infty} dy dx + \int_{I} | V(x) |^2 dx \notag \\
													&\lesssim \| U \|_{H^{1/2}(I)},
	\end{align}
	where we used that the integral of $1/|x-y|^2$ in the second summand of the first line is taken over a domain $| x-y | > \delta > 0$ thanks to the cut-off ensuring that $x$ lying in a strict subset of $I$ is necessary for $| V(x) |/| x-y |^2 \neq 0$. We thus need to establish a connection between the norm of $U$ and the one of $u$. For the $L^2$-norms, such a relationship is obvious. Regarding the $H^{1/2}$-seminorm, this follows rather easily as well by means of a direct comparison and using the decrease of $1/| x-y |^2$ and comparing it to the periodic distance on $S^1$. The claim is thus established.\\
	
	Let us now observe that in the beginning of these calculations, we could have assumed that $\dashint_{S^1} u dx = 0$ or, alternatively, used $u - \hat{u}(0)$ instead of $u$, simply because of $d_{1/2} u$ annihilating constants. Notice that:
	$$\| u - \hat{u}(0) \|_{H^{1/2}(S^1)} \leq C \| u - \hat{u}(0) \|_{\dot{H}^{1/2}(S^1)} = C \| u \|_{\dot{H}^{1/2}(S^1)}$$
	So we arrive at the following estimate by combining all these considerations (using similar ones for $H^{1}$ versus $\dot{H}^{1}$) for $u - \hat{u}(0)$:
	\begin{equation}
		\int_{S^1} | d_{1/2} u |^4(x) dx = \int_{S^1} | d_{1/2} (u - \hat{u}(0)) |^4(x) dx \leq \tilde{C} \| u \|_{\dot{H}^{1/2}(S^1)}^2 \cdot \| u \|_{\dot{H}^{1}(S^1)}^2
	\end{equation}
	This is precisely the inequality used in the proof of improved regularity in the proof of uniqueness and/or the proof of improved regularity.

	\section{Further useful Results}

	\subsection{Wente-type result for Fractional Gradients on the Circle: Lemma \ref{fractionalwentelemmaons1insec2}}
	
	Let us assume that $F \in L^{2}_{od}(S^1 \times S^1)$ and $g \in H^{1/2}(S^1)$. Moreover, we assume that:
	$$\div_{1/2} F = 0,$$
	i.e. that:
	$$\int_{S^1} \int_{S^1} F(x,y) d_{1/2} \varphi(x,y) \frac{dx dy}{| x-y |} = 0, \quad \forall \varphi \in C^{\infty}(S^1)$$
	Our goal is to show that the following holds:
	$$\forall \varphi \in C^{\infty}(S^1): \left| \int_{S^1} F \cdot d_{1/2} g(x) \varphi(x) dx \right| \leq C \| \varphi \|_{\dot{H}^{1/2}},$$
	with:
	$$C \lesssim \| F \|_{L^{2}_{od}} \| g \|_{\dot{H}^{1/2}}$$
	This implies that $F \cdot d_{1/2} g \in H^{-1/2}(S^1)$ which would in turn enable us to solve equations like:
	$$(-\Delta)^{1/2} u = F \cdot d_{1/2} g - c,$$
	where $c = \int_{S^1} F \cdot d_{1/2} g(x) dx$ for some $u \in H^{1/2}(S^1)$ with appropriate estimates. This is the kind of fractional Wente-type estimate we would like to use. To prove this, we observe that by using the stereographic projection $\Pi$ as in Proposition \ref{connectioncircline}, we then have for:
	$$F': \R \times \R \to \R, F'(x,y) = F(\Pi^{-1}(x), \Pi^{-1}(y)), \quad g': \R \to \R, g'(x) = g(\Pi^{-1}(x))$$
	We observe the following for $\varphi \in C^{\infty}(S^1)$ compactly supported in $S^1 \setminus \{ -i \}$ and the previously studied factor $h(z) = \frac{1+ \sin(z)}{2 \left| \sin \left( \frac{z}{2} + \frac{\pi}{4} \right) \right|^2} = 1$ and thus also for:
	$$\tilde{h}(x) := \frac{1}{h(\Pi^{-1}(x))} = 1,$$
	which we may use to obtain the following chain of equations following the computations in the proof of Proposition \ref{propappendixa1}, especially \eqref{appendixa1ep003} to rewrite the quotient of the distance functions on $\R$ and $S^1$, and a change of variables:
	\begin{align}
		\int_{S^1} \int_{S^1} 		&F(z,w) \frac{g(z) - g(w)}{| z - w |^{1/2}} \varphi(z) \frac{dz dw}{| z-w |}	\notag \\
							&=	\int_{\R} \int_{\R} F'(x,y) \frac{g'(x) - g'(y)}{| x-y |^{1/2}} \varphi(\pi^{-1}(x))\left( \frac{8 \left| \sin \left( \frac{\Pi^{-1}(x)}{2} + \frac{\pi}{4} \right) \right|^3 \cdot 8 \left| \sin \left( \frac{\Pi^{-1}(y)}{2} + \frac{\pi}{4} \right) \right|^3}{(1+\sin(\Pi^{-1}(x)))(1+\sin(\Pi^{-1}(y)))} \right)^{1/2} \frac{dx dy}{| x-y |} \notag \\
							&= \int_{\R} \int_{\R} \tilde{F}(x,y) \frac{g'(x) - g'(y)}{| x-y |^{1/2}} \varphi(\Pi^{-1}(x)) \frac{dx dy}{| x-y |},
	\end{align}
	(see below for the definition of $\tilde{F}$) and we observe that $\tilde{h} = 1$ on $\R$ and that:
	\begin{align}
		\int_{\R} \int_{\R} 	&\left(2 \left| \sin \left( \frac{\Pi^{-1}(x)}{2} + \frac{\pi}{4} \right) \right|^{1/2} \left| \sin \left( \frac{\Pi^{-1}(y)}{2} + \frac{\pi}{4} \right) \right|^{1/2} \cdot F'(x,y) \right)^2 \frac{dxdy}{| x-y |}	\notag \\
						&= \int_{\R} \int_{\R} 4 \left| \sin \left( \frac{\Pi^{-1}(x)}{2} + \frac{\pi}{4} \right) \right| \left| \sin \left( \frac{\Pi^{-1}(y)}{2} + \frac{\pi}{4} \right) \right| \cdot |F'(x,y)|^2\frac{dxdy}{| x-y |} \notag \\
						&= \int_{S^1} \int_{S^1} |F(z,w)|^2 \frac{dz dw}{| z-w |},
	\end{align}
	so we observe that if $F \in L^{2}_{od}(S^1 \times S^1)$, then the same holds true for 
	$$\tilde{F} := 2 \left| \sin \left( \frac{\Pi^{-1}(x)}{2} + \frac{\pi}{4} \right) \right|^{1/2} \left| \sin \left( \frac{\Pi^{-1}(y)}{2} + \frac{\pi}{4} \right) \right|^{1/2} \cdot F'(x,y)$$
	for the domain $\R$ instead of $S^1$. This is the ideal starting point for a generalisation of Theorem 2.1 in \cite{mazoschi}, as we have now found the substitute for $F$ on the real line. 
	Next, we observe that we have for any constant $C \in \R$:
	\begin{align}
		\int_{S^1} \int_{S^1} 	F(z,w) \frac{g(z) - g(w)}{| z - w |^{1/2}} \varphi(z) \frac{dz dw}{| z-w |}	&= \int_{S^1} \int_{S^1} F(z,w) (g(z) - C) \frac{\varphi(z) - \varphi(y)}{| z-w |^{1/2}} \frac{dz dw}{| z-w |} \notag \\
																				&= \int_{\R} \int_{\R} \tilde{F}(x,y) (g'(x) - C) d_{1/2} \varphi(\Pi^{-1}(x),\Pi^{-1}(y)) \frac{dx dy}{| x-y |} \notag
	\end{align}
	by using $\div_{1/2} F = 0$. If $\varphi = 1$, we may even notice (observe that the compact support is not relevant to the computations above) that then $\div_{1/2} ( \tilde{F}(x,y)) = 0$. Therefore, the arguments in the proof of Theorem 2.1 become immediately applicable to $\tilde{F}$ and $g'$.
	Hence, this leads us to the realisation:
	$$\tilde{F} \cdot d_{1/2} g' \in \mathcal{H}^{1}(\R),$$
	with the Wente-type estimate found in the preliminary section as well as in \cite{mazoschi}. Observing that $\dot{H}^{1/2}(\R)$ continuously embeds into $BMO(\R)$, we therefore find that $\tilde{F} \cdot d_{1/2} g' \in H^{-1/2}(\R)$. Pulling now back to $S^1$, we may obtain use for smooth compactly supported $\varphi$ on $S^1 \setminus \{ -i \}$:
	\begin{align}
		\int_{S^{1}} \varphi(z) F \cdot d_{1/2}g(z) dz	&= \int_{S^1} \int_{S^1} \varphi(z) F(z,w) \frac{g(z) - g(w)}{| z-w |^{1/2}} \frac{dz dw}{| z-w |} \notag \\
												&= \int_{\R} \int_{\R} \tilde{F}(x,y) \frac{g'(x) - g'(y)}{| x-y |^{1/2}} \varphi(\Pi^{-1}(x)) \frac{dx dy}{| x-y |},
	\end{align}
	The estimate on the circle may thus be obtained from the one on the real line, at least for smooth compactly supported functions on the complement of a point, since the very same argument works with respect to the stereographic projection with respect to any point on the circle.\\
	
	To deduce the result on the entire circle, i.e. $F \cdot d_{1/2} g \in H^{-1/2}(S^1)$, we split any smooth function using a fixed partition of unity into two parts supported each on a compact subset of the complement of a point, the points for example being the north and south pole, and apply the estimate from the real line to each of these parts, using stereographic projections with respect to two different points. Observe that the Gagliardo seminorn and the $L^2$-norm of the parts are controlled by the original (semi-)norm of the smooth function. Therefore, we obtain the desired Wente-type estimate.\\
	
	To close this argument, let us observe that for a suitable $c \in \R$ (given by the integral of $F \cdot d_{1/2} g$ over the circle), we can thus obtain the following estimate:
	$$\left| \int_{S^{1}} (F \cdot d_{1/2} g(z) - c) \varphi(z) dz \right| \leq C \| F \|_{L^{2}_{od}} \| g \|_{H^{1/2}} \| \varphi \|_{\dot{H}^{1/2}}$$
	This is clear by going over to Fourier coefficients on the circle. This can be rephrased as:
	\begin{equation}
	\label{fracwenteons1}
		\| F \cdot d_{1/2} g - c \|_{H^{-1/2}} \leq C \| F \|_{L^{2}_{od}} \| g \|_{H^{1/2}}
	\end{equation}
	
	\subsection{Version of Theorem \ref{schiwangthm1.4} on $S^1$}

	In this section, we shall prove the following:
	
	\begin{thm}
	\label{schiwangthm1.4ons1}
		Let $s \in (0,1)$, $p,q \in ]1, \infty[$ and $f \in L^p(\R)$. Then, if $p > \frac{nq}{n + sq}$, we also have the inclusion $\dot{F}^{s}_{p,q}(S^1) \subset \dot{W}^{s,(p,q)}(S^1)$ together with an estimate:
		$$\| f \|_{\dot{W}^{s, (p,q)}(S^1)} \lesssim \| f \|_{\dot{F}^{s}_{p,q}(S^1)}$$
		The constant depends on $s, p, q, n$.
	\end{thm}
	
	This is in fact the only part of Theorem \ref{schiwangthm1.4} we use throughout the current paper. The proof proceeds as in \cite{schiwang}, see in particular the fourth section in this reference.
	
	\begin{proof}
		First, we notice that the following result, Lemma 4.4 in \cite{schiwang}, continues to hold true:
		
		\begin{lem}
		\label{estimatefortrigpoly}
			Let $k \in \Z,j \in \mathbb{N}$ and $f_j$ be the $j$-th Littlewood-Paley projection of $f$ a periodic distribution on $\R$ (or equivalently an distribution on $S^1$):
			$$f_{j}(x) := \sum_{k \in \Z} \varphi_{j}(k) \hat{f}(k) e^{ikx},$$
			where $\varphi_{j}$ are as in the definition of the Triebel-Lizorkin spaces in section 2. Assume that $x,y \in \R$ together with $| x-y | \sim 2^{-k}$. Then for every $r > 0$, we have:
			\begin{align}
				| f_{j}(x) - f_{j}(y) | 	&\lesssim 2^{j-k} (1 + 2^{j-k})^{n/r} \left( M | f_{j} |^{r}(x) \right)^{1/r} \\
				| f_{j}(y) | 			&\lesssim (1 + 2^{j-k})^{n/r} \left( M | f_{j} |^{r}(x) \right)^{1/r}
			\end{align}
			where $Mg$ denotes the Littlewood-Paley maximal function and the constants only depend on $r$.
		\end{lem}
		
		The proof is exactly the same as in \cite{schiwang}, only referring to \cite{schmeitrieb} Proposition 3.3.5 and Theorem 3.3.5 instead of the results for $\R^n$. Observe that only $j > 0$ need to be considered due to the discrete nature of Fourier coefficients.\\
		
		Having Lemma \ref{estimatefortrigpoly} available, we can argue analogous to \cite{schiwang}. Let us observe that:
		\begin{align}
			\| f \|_{\dot{W}^{s,(p,q)}(S^1)}^{p}	&= \int_{S^1} \left( \int_{S^1} \frac{\big{|} \sum_{j \in \mathbb{N}} f_{j}(x) - f_{j}(y) \big{|}^q}{|x-y|^{1+sq}} dy \right)^{p/q} dx \notag \\
										&\lesssim \int_{-\pi}^{\pi} \left( \sum_{k \in \Z} 2^{k(1+sq)} \int_{A_{k}(x)}  \big{|} \sum_{j} f_{j}(x) - f_{j}(y) \big{|}^{q} dy \right)^{p/q} dx,
		\end{align}
		where $A_{k}(x) := \{ y | 2^{-k} \leq | x-y | < 2^{-k+1} \}$. Notice that we replaced the distance function on the circle $S^1$ by the one on $\R$ and chose the integration domain appropriately to still estimate the expression $\| f \|_{\dot{W}^{s,(p,q)}(S^1)}^{p}$.\\
		
		As in \cite{schiwang}, let us introduce:
		\begin{equation}
			\int_{-\pi}^{\pi} \left( \sum_{k \in \Z} 2^{k(1+sq)} \int_{A_{k}(x)}  \big{|} \sum_{j} f_{j}(x) - f_{j}(y) \big{|}^{q} dy \right)^{p/q} dx \lesssim R_1 + R_2 + R_3,
		\end{equation}
		where:
		\begin{align}
			R_1	&:= \int_{-\pi}^{\pi} \left( \sum_{k \in \Z} 2^{k(1+sq)} \int_{A_{k}(x)} \left( \sum_{j \leq k} \big{|} f_{j}(x) - f_{j}(y) \big{|} \right)^{q} dy \right)^{p/q} dx \\
			R_2	&:= \int_{-\pi}^{\pi} \left( \sum_{k \in \Z} 2^{k(1+sq)} \int_{A_{k}(x)} \left( \sum_{j > k} \big{|} f_{j}(x) \big{|} \right)^{q} dy \right)^{p/q} dx \\
			R_3	&:= \int_{-\pi}^{\pi} \left( \sum_{k \in \Z} 2^{k(1+sq)} \int_{A_{k}(x)} \left( \sum_{j > k} \big{|} f_{j}(y) \big{|} \right)^{q} dy \right)^{p/q} dx
		\end{align}
		The estimate for each contribution now proceeds as in \cite{schiwang}: For example, $R_1$ can be dealt with by noticing that for some $s > \varepsilon > 0$
		\begin{align}
		\label{helpingequation01}
			\left( \sum_{j \leq k} \big{|} f_{j}(x) - f_{j}(y) \big{|} \right)^{q}	&= \left( \sum_{j \leq k} 2^{j\varepsilon} 2^{-j\varepsilon} \big{|} f_{j}(x) - f_{j}(y) \big{|} \right)^{q} \notag \\
														&\lesssim \left( \sum_{j \leq k} 2^{j\varepsilon} \right)^{q} \sup_{j \leq k} 2^{-j\varepsilon q} | f_{j}(x) - f_{j}(y) |^{q} \notag \\
														&\lesssim 2^{k\varepsilon q} \sum_{j \leq k} 2^{-j \varepsilon q} | f_{j}(x) - f_{j}(y) |^{q}
		\end{align}
		Using now Lemma \ref{estimatefortrigpoly}, we arrive at the following identity completely analogous to \cite{schiwang}:
		\begin{equation}
		\label{helpingequation02}
			| f_{j}(x) - f_{j}(y) | \leq C(r) 2^{j-k} \left( M |f_{j} |^{r}(x) \right)^{1/r}, \quad \forall y \in A_{k}(x), \forall j \leq k,
		\end{equation}
		for some constant $C(r) > 0$ depending only on $r > 0$. Combining \eqref{helpingequation01} and \eqref{helpingequation02}, we find:
		\begin{align}
			R_{1}	&\lesssim \int_{-\pi}^{\pi} \left( \sum_{k \in \Z} 2^{k(1+sq)} \int_{A_{k}(x)} 2^{k\varepsilon q} \sum_{j \leq k} 2^{-j\varepsilon q} \big{|} f_{j}(x) - f_{j}(y) \big{|}^{q} dy \right)^{p/q} dx \notag \\
					&\lesssim \int_{-\pi}^{\pi} \left( \sum_{k \in \Z} 2^{k(1+sq)} \int_{A_{k}(x)} 2^{k\varepsilon q} \sum_{j \leq k} 2^{-j\varepsilon q} C(r)^q 2^{(j-k)q} \left( M |f_{j} |^{r}(x) \right)^{q/r} dy \right)^{p/q} dx \notag \\
					&\lesssim \int_{-\pi}^{\pi} \left( \sum_{k \in \Z} 2^{k(1+sq)} 2^{-k} 2^{k\varepsilon q} \sum_{j \leq k} 2^{-j\varepsilon q} 2^{(j-k)q} \left( M |f_{j} |^{r}(x) \right)^{q/r} \right)^{p/q} dx \notag \\
					&\lesssim \int_{-\pi}^{\pi} \left( \sum_{j > 0} 2^{-j \varepsilon q} \left( M | f_{j} |^{r}(x) \right)^{q/r} 2^{jq} \sum_{k \geq j} 2^{k(s-1+\varepsilon)q} \right)^{p/q} dx \notag \\
					&\lesssim \int_{-\pi}^{\pi} \left( 2^{jsq} \left( M | f_{j} |^{r}(x) \right)^{q/r} \right)^{p/q} dx,
		\end{align}
		where we use $\varepsilon > 0$ suffciently small, such that $s + \varepsilon < 1$. Applying Proposition 3.2.4 in \cite{schmeitrieb}, i.e. the maximal function estimate for $L^{p}l^{q}$-functions on $S^1$, we thus deduce by the very definition of $f_j$ and the Triebel-Lizorkin norm:
		$$R_{1} \lesssim \| f \|_{\dot{F}^{s}_{p,q}}$$
		The other contributions $R_2$ and $R_3$ may also be deduced completely analogous to \cite{schiwang}, but using the corresponding results for $S^1$ as found in \cite{schmeitrieb}. We thus may conclude.
	\end{proof}

	\end{appendices}

\end{document}